\documentclass[reqno, 12pt]{amsart}
\usepackage[letterpaper,hmargin=1in,vmargin=1in]{geometry}
\usepackage{amsmath,amssymb,amsthm, hyperref}
\usepackage{color}
\usepackage{tikz}
\usepackage{caption,subcaption}
\usetikzlibrary{shapes}
\usetikzlibrary{calc}

\newtheorem{theorem}[equation]{Theorem}
\newtheorem{proposition}[equation]{Proposition}
\newtheorem{corollary}[equation]{Corollary}
\newtheorem{lemma}[equation]{Lemma}
\newtheorem{definition}[equation]{Definition}

\theoremstyle{remark}
\newtheorem{remark}[equation]{Remark}
\newtheorem{question}[equation]{Question}
\newtheorem{example}[equation]{Example}

\numberwithin{equation}{section}

\newcommand{\sech}{\operatorname{sech}}
\newcommand{\phg}{\operatorname{phg}}
\newcommand{\Ai}{\operatorname{Ai}}
\newcommand{\Bi}{\operatorname{Bi}}

\begin{document}
\title[Joint asymptotic expansions for Bessel functions]
{Joint asymptotic expansions for Bessel functions}
\author{David A. Sher}
\address{Department of Mathematical Sciences, DePaul University, 2320 N. Kenmore Ave, 60614, Chicago, IL, USA.}
\email{dsher@depaul.edu}

\begin{abstract}
We study the classical problem of finding asymptotics for the Bessel functions $J_{\nu}(z)$ and $Y_{\nu}(z)$ as the argument $z$ and the order $\nu$ approach infinity. We use blow-up analysis to find asymptotics for the modulus and phase of the Bessel functions; this approach produces polyhomogeneous conormal joint asymptotic expansions, valid in any regime. As a consequence, our asymptotics may be differentiated term by term with respect to either argument or order, allowing us to easily produce expansions for Bessel function derivatives. We also discuss applications to spectral theory, in particular the study of the Dirichlet eigenvalues of a disk.
\end{abstract}
\maketitle

\section{Introduction}

We consider the asymptotic analysis of the Bessel functions $J_{\nu}(z)$ and $Y_{\nu}(z)$ as the argument $z$ and/or the order $\nu$ approach infinity. This analysis is in many ways quite well understood. The extensive 1922 treatise of Watson \cite{watson} summarizes and expands on what was known at that time; see \cite{watson} for further references. Starting in the 1950s, F.~W.~J.~Olver made a systematic and exhaustive study of these asymptotics, with particular emphasis on the regimes in which $z$ and $\nu$ both approach infinity. See \cite{olver54} among others. Summaries of his results may be found in \cite{dlmf} and \cite{olver97}.

The goal of this work is to present a unified description of the asymptotics of $J_{\nu}(z)$ and $Y_{\nu}(z)$ as both the argument $z$ and the order $\nu$ approach infinity through positive real values, using the language of polyhomogeneous conormal functions on manifolds with corners. This provides a new, uniform framework in which the existing results of Olver and others may be interpreted. Our results are \emph{joint} asymptotic expansions, essentially treating the Bessel functions as functions of two variables. 

Our results are, in addition, a slight sharpening of existing results in two ways:
\begin{itemize}
\item We are able to handle \emph{any} regime in which $z$ and $\nu$ both approach infinity. There are some regimes, for example $z<<\nu$, where Olver's asymptotic expansions do not apply well, see \cite[10.41(v)]{dlmf}. In those regimes we obtain more precise information.
\item Our asymptotic expansions may be differentiated term by term, not just with respect to $z$ but with respect to the order $\nu$. Differentiation with respect to $z$ has classically been handled by using complex analysis, for example Ritt's theorem \cite[2.1(ii)]{dlmfch2}. The existing results on differentiation with respect to $\nu$ are much less clean \cite[10.15]{dlmf}.
\end{itemize}

Throughout the paper we will use the concepts of polyhomogeneous conormal functions, blow-ups, and manifolds with corners. In the analytic setting these concepts were originally introduced by Richard Melrose \cite{melrosebook} and have been used, for example, to study index theory on singular manifolds \cite{tapsit}. Expository treatments may be found in \cite{gribc,griscales,maz91}.

In order to state our main results we introduce a compactification of the first quadrant in the variables $(z,\nu)$. Specifically, we introduce the variables
\[\mu:=\nu^{-1/3};\qquad \zeta:=z^{-1/3}.\]
Then we let $Q_0$ be the rectangular compactification of the first quadrant with respect to $\mu$ and $\zeta$, so that in a neighborhood of $\{\nu=z=\infty\}$, $Q_0$ has the same smooth structure as $[0,1)_{\mu}\times[0,1)_{\zeta}$. This compactification $Q_0$ is a manifold with corners. It has four ``boundary hypersurfaces'', which in this case are line segments.

The asymptotics of Bessel functions as the arguments approach the corner $\{\mu=\zeta=0\}$ depend very closely on the specific path that is taken to approach the corner. So in order to give uniform expansions, we must perform some blow-ups. We first create a new manifold with corners, $Q_1$, by performing a radial blow-up of $\{\mu=\zeta=0\}$. This corresponds to introducing polar coordinates in $(\mu,\zeta)$, and replaces the single point at the corner with a quarter-circle. If we fix different values of the ratio $z/\nu$ and let $\nu\to\infty$, we approach different points on $Q_1$. See Figure \ref{fig:creationofq}.

Although this blow-up is helpful for resolving the singular behavior of the Bessel functions, it is not yet sufficient to distinguish between all the types of behavior we may observe when $z$ and $\nu$ both go to infinity with $z/\nu\to 1$. For example, the asymptotics of Bessel functions with $z=\nu+a\nu^{1/3}$ depend on $a$. So we need a second blow-up. Let $D\subseteq Q_1$ be the closure in $D$ of the set $\{(z,z)\}$, or equivalently the closure in $D$ of the set $\{(\mu,\mu)\, |\, \mu>0\}$. We perform a blow-up of the intersection of $D$ with the boundary of $Q_1$ to obtain a new space $Q$. This blow-up, unlike the first, is a quasi-homogeneous blow-up; it is parabolic with respect to the boundary of $Q_1$. This means that it does not distinguish between rays approaching the intersection, but rather distinguishes between parabolas approaching the intersection. See Figure \ref{fig:creationofq}. The point $D\cap\partial Q_1$ is replaced with a semi-circle. 

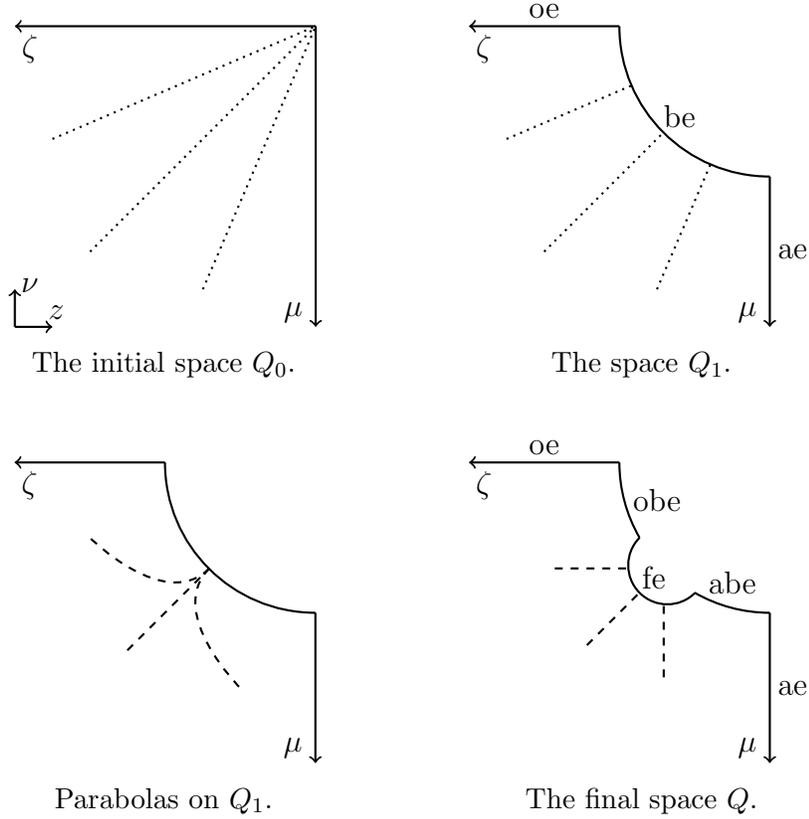
\begin{figure}
\centering
\begin{tabular}{p{6cm} p{6cm}}

\begin{subfigure}[b]{0.3\textwidth}
\centering
\begin{tikzpicture}
\draw[black,thick,->](2,2)--(-2,2);
\draw[black,thick,->](2,2)--(2,-2);
\draw[dotted,thick](0.5,-1.5)--(2,2);
\draw[dotted,thick](-1,-1)--(2,2);
\draw[dotted,thick](-1.5,0.5)--(2,2);
\node at (1.7,-1.8) {$\mu$};
\node at (-1.8,1.7) {$\zeta$};
\draw[black,thick,->](-2,-2)--(-2,-1.5);
\draw[black,thick,->](-2,-2)--(-1.5,-2);
\node at (-1.8,-1.45) {$\nu$};
\node at (-1.45,-1.8) {$z$};
\end{tikzpicture}
\captionsetup{labelformat=empty}
\caption{The initial space $Q_0$.}
\end{subfigure}

&
\begin{subfigure}[b]{0.3\textwidth}
\centering
\begin{tikzpicture}
\draw[black,thick,<-](-2,2)--(0,2);
\draw[black,thick,->](2,0)--(2,-2);
\draw[thick](0,2) arc (180:270:2);
\draw[dotted,thick](0.5,-1.5)--(1.21,0.16);
\draw[dotted,thick](-1,-1)--(0.59,0.59);
\draw[dotted,thick](-1.5,0.5)--(0.16,1.21);
\node at (-1,2.2) {oe};
\node at (.8,.8) {be};
\node at (2.3,-1) {ae};
\node at (1.7,-1.8) {$\mu$};
\node at (-1.8,1.7) {$\zeta$};
\end{tikzpicture}\captionsetup{labelformat=empty}
\caption{The space $Q_1$.}
\end{subfigure}
\\
\begin{subfigure}[b]{0.3\textwidth}
\centering
\begin{tikzpicture}
\draw[black,thick,<-](-2,2)--(0,2);
\draw[black,thick,->](2,0)--(2,-2);
\draw[thick](0,2) arc (180:270:2);
\draw[dashed,thick](-0.5,-0.5)--(0.59,0.59);
\draw[dashed,thick] (0.59,0.59) .. controls (0,0) and (1,-1) .. (1,-1);
\draw[dashed,thick] (0.59,0.59) .. controls (0,0) and (-1,1) .. (-1,1);
\node at (1.7,-1.8) {$\mu$};
\node at (-1.8,1.7) {$\zeta$};
\end{tikzpicture}\captionsetup{labelformat=empty}
\caption{Parabolas on $Q_1$.}
\end{subfigure}

&
\begin{subfigure}[b]{0.3\textwidth}
\centering
\begin{tikzpicture}
\draw[black,thick,<-](-2,2)--(0,2);
\draw[black,thick,->](2,0)--(2,-2);
\draw[thick](0,2) arc (180:210:2);
\draw[thick](2,0) arc (270:240:2);
\draw[thick](.27,1) arc (135:315:0.52);
\draw[dashed,thick](0.24,0.24)--(-0.5,-0.5);
\draw[dashed,thick](0.09,0.59)--(-0.9,0.59);
\draw[dashed,thick](0.59,0.09)--(0.59,-0.9);
\node at (-1,2.2) {oe};
\node at (.5,1.5) {obe};
\node at (.46,.47) {fe};
\node at (1.5,.4) {abe};
\node at (2.3,-1) {ae};
\node at (1.7,-1.8) {$\mu$};
\node at (-1.8,1.7) {$\zeta$};
\end{tikzpicture}\captionsetup{labelformat=empty}
\caption{The final space $Q$.}
\end{subfigure}

\end{tabular}
\caption{The creation of the space $Q$.}
\label{fig:creationofq}
\end{figure}

We denote the resulting manifold with corners by $Q$, and label its five boundary hypersurfaces oe (``order edge''), obe, fe (``front edge''), abe, and ae (``argument edge'') as in Figure \ref{fig:creationofq}. As we will see, each of these boundary hypersurfaces corresponds to a different asymptotic regime:
\begin{itemize}
\item The asymptotics \cite[10.19.1-2]{dlmf} for the Bessel functions with large order and fixed argument are asymptotics at oe;
\item Debye's expansions \cite[10.19.3]{dlmf} for fixed ratios of $z/\nu<1$ are asymptotics at obe;
\item The transition region expansions \cite[10.19.8]{dlmf} for $z=\nu+a\nu^{1/3}$, with $a$ fixed, are asymptotics at fe;
\item Debye's expansions \cite[10.19.6]{dlmf} for fixed ratios of $z/\nu>1$ are asymptotics at abe;
\item The asymptotics \cite[10.17.3-4]{dlmf} for large argument and fixed order are asymptotics at ae.
\end{itemize}
We should also mention the uniform asymptotics of Olver \cite[10.20.4-5]{dlmf}, originally published in \cite{olver54}. These are valid in a neighborhood of fe. Indeed they may be seen as asymptotic expansions on $Q_1$ rather than $Q$, valid on the interior of be. In fact, with proper care they may be extended to be valid even up to the corner be$\cap$ae, though not to be$\cap$oe \cite[10.41(v)]{dlmf}. Observe, however, that the coefficients of those expansions have singular behavior at $D\cap\partial Q_1$, corresponding to $\{z=1,\nu=\infty\}$ in the notation of \cite[10.20.4-5]{dlmf}. Our second blow-up, creating fe, resolves that singular behavior.

The Bessel functions themselves are highly oscillatory for $z>>\nu$ and therefore do not lend themselves well to polyhomogeneous conormal expansions directly. However, it is well known \cite[10.18]{dlmf} that there exist a \emph{modulus} which we denote $M_{\nu}(z)$ and a \emph{phase} which we denote $\theta_{\nu}(z)$ for which
\[J_{\nu}(z)=M_{\nu}(z)\cos\theta_{\nu}(z);\quad Y_{\nu}(z)=M_{\nu}(z)\sin\theta_{\nu}(z).\]

Our main results concern the modulus and the phase. Let $U$ denote a neighborhood of obe, fe, and abe in $Q$ whose closure does not intersect $\{\nu=0\}$ or $\{z=0\}$. Let $U_a$ denote the intersection of $U$ with any neighborhood of abe and ae whose closure does not intersect oe or obe, and we let $U_o$ denote the intersection of $U$ with any neighborhood of obe and oe whose interior contains only points $(z,\nu)$ with $z^{-1/3}-\nu^{-1/3} > 3\nu^{-1}$ (which in practice means that it avoids the lift of $\nu=z$). See Figure \ref{fig:uoua}, in which $U$ should be thought of as $U_o\cup U_a$.

\begin{theorem}\label{thm:mainalpha} The phase $\theta_{\nu}(z)$ lifts to a function which is polyhomogeneous conormal on $U$.
\end{theorem}

\begin{theorem}\label{thm:mainmodulus} The modulus $M_{\nu}(z)$ lifts to a function which is polyhomogeneous conormal on $U_a$. Moreover, the function
\[M_{\nu}(z)e^{\sqrt{\nu^2-z^2}-\nu\cosh^{-1}(\nu/z)}\]
lifts to a function which is polyhomogeneous conormal on $U_o$.
\end{theorem}
\noindent Note the appearance of the exponential in Theorem \ref{thm:mainmodulus}. The point is that $M_{\nu}(z)$ grows exponentially at oe and obe. This theorem characterizes the exponential growth, showing that $M_{\nu}(z)$ equals $\exp[-(\sqrt{\nu^2-z^2}-\nu\cosh^{-1}(\nu/z))]$ times a function which is polyhomogeneous conormal.

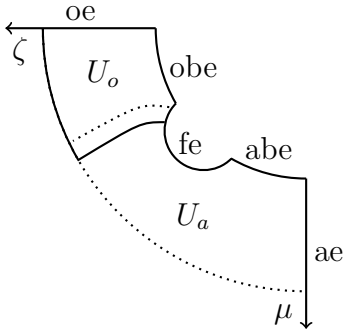
\begin{figure}
\centering
\begin{tikzpicture}
\draw[black,thick,<-](-2,2)--(0,2);
\draw[black,thick,->](2,0)--(2,-2);
\draw[thick](0,2) arc (180:210:2);
\draw[thick](2,0) arc (270:240:2);
\draw[thick](.27,1) arc (135:315:0.52);
\draw[thick,dotted] (-1.5,2) arc (180:270:3.5);
\draw[thick] (-1.5,2) arc (180:210:3.5);
\draw[thick] (-1.04,0.24) .. controls (-0.2,0.75) .. (0.12,0.75);
\draw[thick,dotted] (0.18,0.95) .. controls (-0.2,1) .. (-1.1,0.5);
\node at (-1,2.2) {oe};
\node at (.5,1.5) {obe};
\node at (.46,.47) {fe};
\node at (1.5,.4) {abe};
\node at (2.3,-1) {ae};
\node at (1.7,-1.8) {$\mu$};
\node at (-1.8,1.7) {$\zeta$};
\node at (-0.7,1.4) {$U_o$};
\node at (0.5,-0.5) {$U_a$};
\end{tikzpicture}
\caption{Illustration of $U_o$ and $U_a$.}
\label{fig:uoua}
\end{figure}

For the reader who is unfamiliar with polyhomogeneous conormal functions, these results should be thought of in the following sense: the functions involved have asymptotic expansions at each boundary hypersurface of $U$, of the form
\[f\sim\sum_{(s,p)\in E}a_{(s,p)}x^s(\log x)^p,\]
where $x$ is a defining function for that boundary hypersurface and $E\subseteq\mathbb C\times\mathbb N_0$ satisfies certain finiteness conditions. The existence of expansions of this form for Bessel functions is, as discussed, not new. However, polyhomogeneity means that these previously-known asymptotic expansions are joint, in a suitable sense, down to each corner of $U$, which allows us to analyze transitional asymptotic regimes. And conormality means that these asymptotic expansions, and the associated remainder estimates, are preserved when a vector field tangent to the boundary is applied. As we will see, this allows us to differentiate the previously existing asymptotic expansions term by term to get expansions for derivatives of Bessel functions.

The key idea of the proofs is to take advantage of integral formulas for the modulus and phase. The first is Nicholson's formula for the square of the modulus \cite[6.664.4]{gr}:
\begin{equation}\label{eq:nicholson}
M_{\nu}^2(z)=\frac{8}{\pi^2}\int_0^{\infty}K_0(2z\sinh t)\cosh(2\nu t)\, dt.
\end{equation}
The second follows immediately from \cite[10.18.3,10.18.8]{dlmf}:
\begin{equation}\label{eq:phaseformula}
\theta_{\nu}(z) = -\frac{\pi}{2}+\frac{2}{\pi}\int_0^z\frac{1}{uM_{\nu}^2(u)}\, du.
\end{equation}
Polyhomogenous conormal functions are well-behaved under integration, which allows the analysis first of $M_{\nu}^2(z)$ and then of $M_{\nu}(z)$ and $\theta_{\nu}(z)$.

The plan of the paper is as follows. Section 2 is devoted to putting our results in context without going too deeply into the technical details. We connect them with the existing literature and discuss the specific form of the asymptotic expansions at each boundary hypersurface. We give several examples of their utility, analyzing the behavior of Bessel functions in an unusual asymptotic regime and also proving term-by-term differentiability of the ``transition region'' expansions \cite[10.19.8]{dlmf} at fe. We also discuss potential applications to long-standing conjectures in spectral theory concerning the eigenvalues of the Laplacian on a planar disk.

Sections 3--6 comprise the proof of the main results. We begin in section 3 by analyzing the square of the modulus using Nicholson's formula \eqref{eq:nicholson}. We first prove polyhomogeneity of a relative of $K_0(x)$, then construct a space on which we can prove that the integrand in \eqref{eq:nicholson} is polyhomogeneous conormal. An application of Melrose's pushforward theorem completes this analysis. In section 4, we expand the analysis in section 3 by proving refined asymptotics for $M_{\nu}^2(z)$ in the regime $z<<\nu$. Although $M_{\nu}^2(z)$ blows up, we are able to show that it equals an explicit exponential function times a polyhomogeneous conormal function. The strategy is much the same as section 3, constructing a blown-up space, proving that an integrand is polyhomogeneous conormal, then using the pushforward theorem. In section 5 we pass from modulus-squared to modulus and prove Theorem \ref{thm:mainmodulus}. In section 6 we use \eqref{eq:phaseformula} combined with some blow-up analysis and some explicit estimates to prove Theorem \ref{thm:mainalpha}.

There are three appendices as well; the first is expository, containing background information on polyhomogeneous conormal functions on manifolds with corners. The second contains a selection of useful new results on composition and exponentiation of polyhomogeneous conormal functions. And the third contains combinatorial analysis of the specific maps between manifolds with corners that we use in sections 3, 4, and 6.\\

\textbf{Acknowledgements:} The author is deeply grateful to Daniel Grieser for helpful conversations and feedback, as well as to Rafe Mazzeo for encouragement. Thanks also to Iosif Polterovich and Asma Hassannezhad for discussions. Special thanks to Ian Petrow for asking the question, many years ago, that served as the inspiration for this research. This research was partially supported by a Faculty Summer Research Grant in summer 2021 from DePaul University.


\section{Applications and Examples}

Our first task is to connect our main results to the previously known asymptotic expansions of Bessel functions summarized in \cite{dlmf}. We focus on the more interesting faces ae, abe, and fe. In order to do this we must discuss local coordinate systems on our manifold with corners $Q$. It is easiest to use projective coordinates and we do so throughout. 

First consider coordinates on $Q_1$. Set
\begin{equation}
\lambda:=\frac{\mu}{\zeta};\quad \eta:=\frac{\zeta}{\mu};\quad w=\eta-1.
\end{equation}
Then, near the corner be$\cap$oe but away from ae, we may use the coordinates $(\zeta,\lambda)$. Conversely, near be$\cap$ae but away from oe, we may use $(\mu,\eta)$. In the interior of be away from both oe and ae, we may use either of these, but we may also use $(\mu,w)$, in which $\{w=0\}$ is the lift of the diagonal $\{z=\nu\}$. 

As for coordinates on $Q$, $Q$ is obtained from $Q_1$ via a quasihomogeneous blow-up of $\{\mu=w=0\}$. Introduce the coordinates
\[\hat w_{\pm}:=\sqrt{\pm w};\quad \hat\mu_{\pm}:=\frac{\mu}{\hat w_{\pm}};\quad \tilde w:=\frac{w}{\mu^2}.\]
Then near the intersection of fe and abe we use $(\hat w_{-},\hat \mu_{-})$, near the intersection of fe and obe we use $(\hat w_{+},\hat \mu_{+})$, and in the interior of fe we use $(\tilde w,\mu)$. Elsewhere the situation is unchanged: near obe$\cap$oe we use $(\zeta,\lambda)$, and near abe$\cap$ae we use $(\mu,\eta)$. The polyhomogeneous expansions claimed in our main Theorems should be read as joint asymptotic expansions in these coordinates. These coordinates are illustrated in Figure \ref{fig:coordinatesonq}.

\begin{figure}
\centering
\begin{tabular}{p{6cm} p{6cm}}
\begin{subfigure}[b]{0.3\textwidth}
\centering
\begin{tikzpicture}
\draw[black,thick,<-](-2,2)--(0,2);
\draw[black,thick,->](2,0)--(2,-2);
\draw[thick](0,2) arc (180:270:2);
\draw[thick,->] (-0.2,1.8)--(-0.8,1.8);
\draw[thick,->] (-0.2,1.8) arc (183:200:2);
\draw[thick,->] (1.8,-0.2)--(1.8,-0.8);
\draw[thick,->] (1.8,-0.2) arc (267:250:2);
\draw[thick,->] (0.45,0.45) arc (225:216:2);
\draw[thick] (0.45,0.45) arc (225:233:2);
\draw[dashed,thick](-0.5,-0.5)--(0.59,0.59);
\node at (1.7,-1) {$\mu$};
\node at (-1,1.7) {$\zeta$};
\node at (-0.3,1.2) {$\lambda$};
\node at (1.2,-0.4) {$\eta$};
\node at (-0.75,-0.75) {$\{w=0\}$};
\node at (0.05,0.5) {$w$};
\node at (-1,2.2) {oe};
\node at (.8,.8) {be};
\node at (2.3,-1) {ae};
\end{tikzpicture}
\end{subfigure}
&
\begin{subfigure}[b]{0.3\textwidth}
\centering
\begin{tikzpicture}
\draw[black,thick,<-](-2,2)--(-1,2);
\draw[black,thick,->](2,-1)--(2,-2);
\draw[thick](-1,2) arc (180:210:3);
\draw[thick](2,-1) arc (270:240:3);
\draw[thick](-0.6,0.5) arc (135:315:0.78);
\draw[thick,->] (-1.2,1.8)--(-1.8,1.8);
\draw[thick,->] (-1.2,1.8) arc (182:192:3);
\draw[thick,->] (-0.85,0.5) arc (210:200:3);
\draw[thick,->] (-0.85,0.5) arc (135:175:0.8);
\draw[thick,->] (1.8,-1.2)--(1.8,-1.8);
\draw[thick,->] (1.8,-1.2) arc (268:258:3);
\draw[thick,->] (0.5,-0.85) arc (240:250:3);
\draw[thick,->] (0.5,-0.85) arc (315:275:0.8);
\draw[dashed,thick](-0.6,-0.6)--(-1.5,-1.5);
\draw[thick,->] (-0.8,-0.8) arc (225:206:1);
\draw[thick] (-0.8,-0.8) arc (225:243:1);
\node at (-1.25,-1.75) {$\{\tilde w=0\}$};
\node at (-1.5,2.2) {oe};
\node at (-0.35,1.25) {obe};
\node at (-0.3,-0.3) {fe};
\node at (1.25,-0.55) {abe};
\node at (2.3,-1.5) {ae};
\node at (1.55,-1.8) {$\mu$};
\node at (-1.9,1.7) {$\zeta$};
\node at (-1.4,1.4) {$\lambda$};
\node at (-1.3,0.75) {$\hat w_{+}$};
\node at (-1.35,0.1) {$\hat\mu_{+}$};
\node at (-1.25,-0.65) {$\tilde w$};
\node at (0.05,-1.3) {$\hat\mu_{-}$};
\node at (0.8,-1.3) {$\hat w_{-}$};
\node at (1.4,-1.45) {$\eta$};
\end{tikzpicture}
\end{subfigure}
\end{tabular}
\caption{Local projective coordinates on $Q_1$ and $Q$.}
\label{fig:coordinatesonq}
\end{figure}
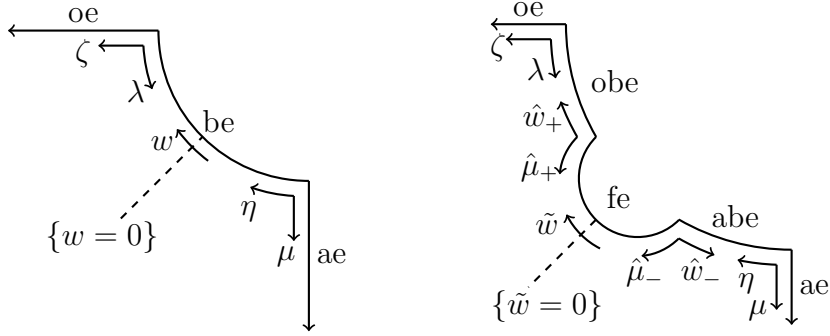

\subsection{Expansions at ae}
Let us discuss what this means in practice, beginning with the face ae. The face ae is the large-argument face. In the interior of ae, we can use $\zeta=z^{-1/3}$ as a boundary defining function and $\nu$ itself as the other coordinate. So the expansions for the modulus-squared and phase given in \cite[10.18.17-18]{dlmf} are polyhomogeneous conormal as functions of $(\zeta,\nu)$. These expansions are:
\[M_{\nu}^2(z) \sim \frac{2}{\pi}\zeta^3(1+\frac 18(4\nu^2-1)\zeta^6+\dots); \]
\begin{equation}\label{eq:thetaae}\theta_{\nu}(z) \sim \zeta^{-3} - (\frac 12\nu+\frac 14)\pi + \frac 18(4\nu^2-1)\zeta^3 + \dots \end{equation}
A key advantage of polyhomogeneity in this context is that it allows differentiation term by term. Specifically, we may apply any number of b-vector fields to this expansion term by term, preserving the same error bounds. Note that 
\[z\frac{\partial}{\partial z}=-\frac 13\zeta\frac{\partial}{\partial\zeta},\]
where $\zeta\partial_{\zeta}$ is a $b$-vector field. Thus a $z$-derivative is $z^{-1}$ times a derivative which does not change the size of the error term in these expansions. Therefore the expansions for $M_{\nu}(z)$ and $\theta_{\nu}(z)$ can be differentiated term by term with respect to $z$, which indeed brings down an extra factor of $z^{-1}$. This can be used to study $J_{\nu}(z)$, $Y_{\nu}(z)$, and arbitrarily many derivatives of either. For example, we can compute
\[J'_{\nu}(z) = M_{\nu}'(z)\cos(\theta_{\nu}(z)) - \theta'_{\nu}(z)M_{\nu}(z)\sin(\theta_{\nu}(z))\]
and use our asymptotic understanding of $M_{\nu}(z)$, $\theta_{\nu}(z)$, and their derivatives to get precise asymptotic understanding of $J'_{\nu}(z)$. This can be used to re-derive, and extend to higher derivatives, the asymptotic expansions of Bessel function derivatives for large argument given in \cite[10.17.9-10]{dlmf}. 

Additionally, since $\partial_{\nu}$ is a b-vector field (away from abe), this expansion may be differentiated term by term with respect to order, with the same error bounds.

\subsection{Expansions at abe and compatibility with ae}
The known asymptotic expansions for Bessel functions at abe, where $\nu$ and $z$ both go to infinity with $\nu/z$ approaching a constant less than 1, are Debye's expansions \cite[10.19.6]{dlmf} (the expansions in \cite[10.20.4-5]{dlmf} also work). Modifying that notation slightly, if $\beta\in(0,\pi/2)$ is fixed, and $\xi=\nu(\tan\beta-\beta)-\frac 14\pi$, then
\[J_{\nu}(\nu\sec\beta) = (\frac{2}{\pi\nu\tan\beta})^{1/2}(\cos\xi\sum_{k=0}^{\infty}u_{2k}(\beta)\nu^{-2k} + \sin\xi\sum_{k=0}^{\infty}u_{2k+1}(\beta)\nu^{-2k-1}),\]
with a similar expansion for $Y_{\nu}$. Here $u_{0}(\beta)=1$ and $u_1(\beta)=\frac{1}{24}(3\cot\beta+5\cot^3\beta)$.

On $Q$, good coordinates on abe are $(\eta,\mu)$. In terms of $\eta$ and $\mu$, $\nu=\mu^{-3}$ and $\beta=\arccos(\eta^3)$. So Debye's expansions are asymptotic expansions as $\mu\to 0$ with $\eta$ fixed and thus are expansions at abe. We know that $M_{\nu}^2(z)$ and $\theta_{\nu}(z)$ have polyhomogeneous expansions at abe, and we can use Debye's expansions to figure out what these expansions are. After some calculation, the first term and error are
\begin{equation}\label{eq:mabe}
M_{\nu}^2(\mu^{-3}\eta^{-3})=\frac{2\mu^3\eta^3}{\pi\sqrt{1-\eta^6}}+O(\mu^9).
\end{equation}
We can also get $\theta_{\nu}(\mu^{-3}\eta^{-3})$. This is a little bit more complicated because the expansions are stated in terms of $\cos\xi$ and $\sin\xi$ rather than $\cos\theta$ and $\sin\theta$. Nevertheless, using some trig identities, we see that for some $n_0\in\mathbb Z$ (which we later show is zero),
\[\theta=\xi+\arctan\big(\frac{\sum_{k=0}^{\infty}u_{2k+1}(\beta)\nu^{-2k-1}}{\sum_{k=0}^{\infty}u_{2k}(\beta)\nu^{-2k}}\big)+2\pi n_0.\]
Rewriting and expanding yields
\begin{equation}\label{eq:thetaabe}
\theta_{\nu}(\mu^{-3}\eta^{-3})=\mu^{-3}(\frac{\sqrt{1-\eta^6}}{\eta^3}-\arccos(\eta^3))-\frac 14\pi + 2\pi n_0 +\mu^3(\frac {\eta^3}{8\sqrt{1-\eta^6}}
+ \frac{5}{24}\frac{\eta^9}{(1-\eta^6)^{3/2}}) + O(\mu^9).\end{equation}

The coordinates $(\eta,\mu)$ are valid down to the intersection of abe and ae. So the expansion \eqref{eq:thetaabe} must be consistent with the expansion \eqref{eq:thetaae}. To check this, write \eqref{eq:thetaae} in these coordinates:
\begin{equation}\label{eq:thetaae2}\theta_{\nu}(\mu^{-3}\eta^{-3}) \sim \mu^{-3}\eta^{-3} - (\frac 12\mu^{-3}+\frac 14)\pi + \frac 18(4\mu^{-3}-\mu^3)\eta^3 + O(\eta^6). \end{equation}
To check consistency we must compute the leading order behavior of the coefficients of \eqref{eq:thetaabe} as $\eta\to 0$. Taylor expansions yield
\[(\frac{\sqrt{1-\eta^6}}{\eta^3}-\arccos(\eta^3))=\eta^{-3} - \frac{\pi}{2} +  \frac 12\eta^3 + O(\eta^6);\]
\[(\frac {\eta^3}{8\sqrt{1-\eta^6}}
+ \frac{5}{24}\frac{\eta^9}{(1-\eta^6)^{3/2}})=\frac 18\eta^3+O(\eta^9).\]
So the expansion \eqref{eq:thetaabe} is
\[\mu^{-3}(\eta^{-3}-\frac{\pi}{2}+\frac 12\eta^3+O(\eta^6))-\frac 14\pi+2\pi n_0+\mu^3(\frac 18\eta^3+O(\eta^9))+O(\mu^9).\]
Comparing this to \eqref{eq:thetaae2}, we see all of the same terms, as long as $n_0=0$. So we must indeed have $n_0=0$ as claimed.

\subsection{Transition region expansions and differentiation}

To analyze the expansions at fe, it is instructive to observe that the coordinate
\[\tilde w=\frac{\nu-z}{z^{1/3}+z^{2/3}\nu^{-1/3}+z\nu^{-2/3}}=-\frac{z-\nu}{\nu^{1/3}}\frac{1}{(z/\nu)^{1/3}+(z/\nu)^{2/3}+(z/\nu)}.\]
The second fraction is a smooth function on $Q$ away from oe, equal to $\frac 13$ at fe. Thus we may change coordinates from $\tilde w$ to $a:=\frac{z-\nu}{\nu^{1/3}}$, and use the coordinates $(a,\nu^{-1/3})$ on $Q$. Note that $a$ has the opposite sign to $\tilde w$. Theorems \ref{thm:mainalpha} and \ref{thm:mainmodulus} now guarantee that $\theta_{\nu}(z)$ and $M_{\nu}(z)$ have polyhomogeneous expansions in these coordinates for $a$ lying in any compact set. In fact, here, we get more: since $\theta_{\nu}(z)$ is bounded on any compact subset of the interior of fe, $\cos(\theta_{\nu}(z))$ and $\sin(\theta_{\nu}(z))$ are both themselves polyhomogeneous by Theorem \ref{thm:comp}. And polyhomogeneous functions form an algebra, so multiplying by $M_{\nu}(z)$ preserves polyhomogeneity. This shows that $J_{\nu}(z)$ and $Y_{\nu}(z)$ themselves have polyhomogeneous expansions in $(a,\nu^{-1/3})$ at fe.

But we know what those expansions must be. Since asymptotic expansions are unique, they must be the known transition region expansions of Bessel functions \cite[10.19.8]{dlmf}:
\begin{equation}\label{eq:dlmf10198}
J_{\nu}(\nu+a\nu^{1/3})=\frac{2^{1/3}}{\nu^{1/3}}\textrm{Ai}(-2^{1/3}a)\sum_{k=0}^{\infty}\frac{P_k(a)}{\nu^{2k/3}} + \frac{2^{2/3}}{\nu}\textrm{Ai}'(-2^{1/3}a)\sum_{k=0}^{\infty}\frac{Q_k(a)}{\nu^{2k/3}},\end{equation}
\begin{equation}\label{eq:dlmf10198v2}
Y_{\nu}(\nu+a\nu^{1/3})=-\frac{2^{1/3}}{\nu^{1/3}}\textrm{Bi}(-2^{1/3}a)\sum_{k=0}^{\infty}\frac{P_k(a)}{\nu^{2k/3}} - \frac{2^{2/3}}{\nu}\textrm{Bi}'(-2^{1/3}a)\sum_{k=0}^{\infty}\frac{Q_k(a)}{\nu^{2k/3}}.\end{equation}
Here Ai and Bi are the Airy functions and $P_k(a)$, $Q_k(a)$ are explicit polynomials, see \cite[10.9.10-11]{dlmf}. The upshot is:
\begin{proposition} For $a$ in any compact subset of $\mathbb R$, the asymptotic expansions \eqref{eq:dlmf10198} and \eqref{eq:dlmf10198v2} are polyhomogeneous expansions in $(a,\nu^{-1/3})$.
\end{proposition}

As for differentiation term by term, the space of b-vector fields consists of $C^{\infty}$-linear combinations of
\[\frac{\partial}{\partial a},\quad \nu^{-1/3}\frac{\partial}{\partial\nu^{-1/3}}.\]
The vector field $\partial/\partial z$ lifts in these coordinates to
\[\frac{1}{\nu^{1/3}}\frac{\partial}{\partial a}.\]
Therefore:
\begin{proposition} For any $\ell$, we have the following asymptotic expansions:
\begin{multline}\label{eq:dlmf10198k}
J_{\nu}^{(\ell)}(\nu+a\nu^{1/3})=\frac{2^{1/3}}{\nu^{(\ell+1)/3}}\sum_{k=0}^{\infty}\nu^{-2k/3}\frac{\partial^{\ell}}{\partial a^{\ell}}(P_k(a)\Ai(-2^{1/3}a)) \\ 
+ \frac{2^{2/3}}{\nu^{(k+3)/3}}\sum_{k=0}^{\infty}\nu^{-2k/3}\frac{\partial^{\ell}}{\partial a^{\ell}}(Q_k(a)\Ai'(-2^{1/3}a)),\end{multline}
\begin{multline}\label{eq:dlmf10198v2k}
Y_{\nu}^{(\ell)}(\nu+a\nu^{1/3})=-\frac{2^{1/3}}{\nu^{(\ell+1)/3}}\sum_{k=0}^{\infty}\nu^{-2k/3}\frac{\partial^{\ell}}{\partial a^{\ell}}(P_k(a)\Bi(-2^{1/3}a)) \\ 
- \frac{2^{2/3}}{\nu^{(k+3)/3}}\sum_{k=0}^{\infty}\nu^{-2k/3}\frac{\partial^{\ell}}{\partial a^{\ell}}(Q_k(a)\Bi'(-2^{1/3}a)).\end{multline}
These expansions are valid uniformly for $a$ in any compact subset of $\mathbb R$, and in each, the error in any partial sum is $O(\nu^{-s})$, where $\nu^{-s}$ is the order of the next largest term in the expansion.
\end{proposition}

The same analysis may be applied to analyze derivatives of these Bessel functions with respect to order. Note that the vector field $\partial/\partial\nu$ lifts to
\[\frac{\partial a}{\partial\nu}\frac{\partial}{\partial a}+\frac{\partial\mu}{\partial\nu}\frac{\partial}{\partial\mu},\]
with $\mu=\nu^{-1/3}$. Thus this lift is
\[(-\mu-\frac 13a\mu^3)\frac{\partial}{\partial a} -\frac 13\mu^3(\mu\frac{\partial}{\partial\mu}).\]
This calculation allows differentiation of the expansions term by term: the vector fields $\frac{\partial}{\partial a}$ and $\mu\frac{\partial}{\partial\mu}$ preserve the form of the expansions, and multiplication by the pre-factors is easy to analyze. We choose not to formulate the results explicitly here.

Finally, observe that the asymptotic expansions of $M_{\nu}(z)$ and $\theta_{\nu}(z)$ at fe may in principle be read off from \eqref{eq:dlmf10198} and \eqref{eq:dlmf10198v2}. For example, since $P_0(a)=1$, we have
\begin{equation}\label{eq:mfe}
M_{\nu}(\nu+a\nu^{1/3})=\frac{2^{1/3}}{\nu^{1/3}}\sqrt{(\textrm{Ai}^2+\textrm{Bi}^2)(-2^{1/3}a)} + O(\nu^{-1}).
\end{equation}
As many terms as desired may be obtained in this fashion. Reading off $\theta_{\nu}(z)$ is slightly more complicated but can be done via the observation that $\tan\theta_{\nu}(z)=Y_{\nu}(z)/J_{\nu}(z)$.

\subsection{Application to asymptotics in an intermediate regime}

Since polyhomogeneous expansions are joint expansions, they can be used to analyze the asymptotic behavior of Bessel functions in intermediate regimes. Let us illustrate with an example.

Suppose that we want to understand the asymptotic behavior as $\nu\to\infty$ of
\[J_{\nu}(\nu+\sqrt{\nu}).\]
This is a regime that is in between the transition region expansions (where $z=\nu+a\nu^{1/3}$) and the Debye expansions (where $z=c\nu$ for $c>1$), and as such neither of these is sufficient for this purpose. The expansions \cite[10.20.4-5]{dlmf} do cover this regime, but those expansions are not polyhomogeneous.

To begin the analysis, observe that if $z=\nu+\sqrt{\nu}$, then as $\nu\to\infty$, the point $(z,\nu)$ approaches the intersection of the edges fe and abe in $Q_1$ -- literally in between the face fe where the transition region expansions apply and the face abe where the Debye expansions apply. In a neighborhood of the intersection of fe and abe, we use the coordinates $(\hat\mu_{-}, \hat w_{-})$. We now write the expansions for $M_{\nu}(z)$ and $\theta_{\nu}(z)$ at either fe or abe (we can choose!) in these coordinates.

For $M_{\nu}(z)$, let us use the expansion \eqref{eq:mfe} and convert it to the coordinates $(\hat\mu_{-},\hat w_{-})$. Before doing so we rewrite \eqref{eq:mfe} by using asymptotics of Airy functions \cite[9.8.20]{dlmf}:
\begin{equation}\label{eq:mfemod}
M_{\nu}(\nu+a\nu^{1/3})=\frac{2^{1/4}}{\nu^{1/3}\sqrt{\pi}}a^{-1/4}(1+O(a^{-3})) + O(\nu^{-1}).
\end{equation}
Now use the formulas
\[\nu=\hat\mu_-^{-3}\hat w_-^{-3};\quad a=\hat\mu_-^{-2}\frac{3-3\hat w_{-}^2+\hat w_-^4}{(1-\hat w_-^2)^{3}}.\]
We get
\begin{equation}\label{eq:mfe2}
M_{\nu}(\hat\mu_-^{-3}\hat w_-^{-3}(1-\hat w_-^2)^{-3})=\frac{2^{1/4}}{\sqrt{\pi}}\hat \mu_-^{3/2}(1+O(\hat\mu^6))\hat w_{-} + O(\hat w_{-}^{3}).
\end{equation}
The leading order of $M_{\nu}(z)$ at fe is 1 (as indicated by the single factor of $\hat w_-$) and at abe is 3/2. By the fact that our expansion is polyhomogeneous, the $O(\hat w_{-}^3)$ term also has leading order $3$ at abe and so it is bounded by $C\hat w_{-}^3\hat\mu_-^{3/2}$. The upshot is that
\[M_{\nu}(\hat\mu_-^{-3}\hat w_-^{-3}(1-\hat w_-^2)^{-3})-\frac{2^{1/4}}{\sqrt{\pi}}\hat \mu_-^{3/2}\hat w_{-} = O(\hat w_{-}\hat\mu_{-}^{15/2})+O(\hat w_{-}^3\hat\mu_{-}^{3/2}).\]
Now we go back to the original problem with its variable $\nu$. If $z=\nu+\nu^{1/2}$, then
\[\hat w_-=(\sqrt\nu+1)^{-1/2}=O(\nu^{-1/2});\quad \hat\mu_{-}=\nu^{-1/3}(\sqrt\nu + 1)^{1/2}=O(\nu^{-1/12}).\]
Thus
\[M_{\nu}(\nu+\sqrt{\nu})=\frac{2^{1/4}}{\sqrt{\pi}}(\nu^{-1/2}(\sqrt{\nu}+1)^{1/4}) + O(\nu^{-9/8}),\]
which becomes a two-term asymptotic expansion:
\[M_{\nu}(\nu+\sqrt{\nu})=\frac{2^{1/4}}{\sqrt{\pi}}\nu^{-3/8}(1+\frac 14\nu^{-1/2}) + O(\nu^{-9/8}).\]
More terms can be obtained by using more terms in the expansion of $M_{\nu}(z)$
\footnote{The same expression could have been found by using \eqref{eq:mabe} instead of \eqref{eq:mfe}, of course taking the square root of \eqref{eq:mabe} along the way.}..

A similar expansion can be found for $\theta_{\nu}(\nu+\sqrt{\nu})$ by using the same methods. In this case it is easier to use the expansion at abe, namely \eqref{eq:thetaabe}. And then the expansion of $J_{\nu}(\nu+\sqrt{\nu})$ can be analyzed by using the expansions for the modulus and phase. We do not include the details here. 

This type of calculation can be repeated to obtain expansions for $J_{\nu}(z)$ and $Y_{\nu}(z)$ in any asymptotic regime as $\nu$ and $z$ go to infinity. We give one further example -- a path approaching the intersection of oe and obe, in which none of the asymptotic expansions in \cite{dlmf} apply. Specifically, we consider the asymptotic behavior as $\nu\to\infty$ of
\[M_{\nu}(\sqrt{\nu}).\]
Here the $Y$-Bessel function dominates, so we are in essence analyzing $Y_{\nu}(\sqrt{\nu})$. From Theorem \ref{thm:mainmodulus}, we know that $M_{\nu}(z)$ equals $e^{-\sqrt{\nu^2-z^2}+\nu\cosh^{-1}(\nu/z)}$ times a polyhomogeneous function. As we later show in section~\ref{sec:modandinv}, the leading order term of that polyhomogeneous function at obe is
\[\sqrt{\frac{2}{\pi}}(\nu^2-z^2)^{-1/4}.\]
In the coordinates $(\zeta,\lambda)$ valid near obe$\cap$oe, this becomes
\[\sqrt{\frac{2}{\pi}}\zeta^{3/2}\lambda^{3/2}(1-\lambda^6)^{-1/4}.\]
Go back to the original variables by observing that along the path $z=\sqrt{\nu}$ we have $\zeta=\lambda=\nu^{-1/6}$. Since $\lambda\to 0$, the leading order behavior is just a multiple of $\nu^{-1/2}$. All in all,
\[M_{\nu}(\sqrt{\nu})=e^{-\sqrt{\nu^2-\nu}+\nu\cosh^{-1}(\sqrt{\nu})}(\sqrt{\frac{2}{\pi}}\nu^{-1/2}+o(\nu^{-1/2})).\]
Again, we can get more terms by taking more terms at obe.

\subsection{Applications to spectral theory on a disk} Consider the Dirichlet Laplacian $\Delta$ on the unit disk. By separation of variables, the eigenvalues of $\Delta$ are the squares of the positive zeroes of the Bessel functions $J_n(z)$, $n\in\mathbb Z$. Let $N_{\textrm{disk}}(\lambda)$ be the number of eigenvalues of $\Delta$ which are less than or equal to $\lambda^2$ (matching the notation of \cite{cdv}). It is a well-known result of Ivrii \cite{ivrii} that we have two-term Weyl asymptotics
\[N_{\textrm{disk}}(\lambda) = \frac 14\lambda^2 -\frac 12\lambda + R_{\textrm{disk}}(\lambda);\quad R_{\textrm{disk}}(\lambda)=o(\lambda).\]
Despite the two-term Weyl asymptotics, there are still open questions. A natural one is:
\begin{question} What is the best possible bound on $R(\lambda)$?
\end{question}
Another open question is a special case of a conjecture due to P\'olya:
\begin{question}\label{q:polya} Is it true that $N_{\textrm{disk}}(\lambda)\leq\frac 14\lambda^2$?
\end{question}
Although the Weyl asymptotics certainly imply that this is true for sufficiently large $\lambda$, it is not clear how large $\lambda$ must be or whether there are counterexamples below that threshhold. For a more complete discussion of P\'olya's conjecture, see \cite{laugesen}.

The problem of estimating $R_{\textrm{disk}}(\lambda)$ has been studied by comparing the eigenvalue counting function to a counting function for a lattice point problem, going back at least to the work of Kuznecov and Fedosov in the 1960s \cite{kufe}. Their results were later rediscovered by Colin de Verdi\`{e}re \cite{cdv}, with the same proof strategy. Following the notation of \cite{cdv}, let $D$ be the domain in $\mathbb R^2$ given by
\[D:=\{(x,y)\, |\, -1\leq x\leq 1,\, \max(0,-x)\leq y\leq \frac{1}{\pi}(\sqrt{1-x^2}-x\arccos x)\}.\]
Then define the counting function
\[N_{D}(\lambda):=\#\{(n,k)\in\mathbb Z^2\, | (n,k-\frac 14)\in \lambda D\}.\] 
In \cite{kufe} it is proved that
\[R_{\textrm{disk}}(\lambda) = O(\lambda^{2/3})\]
by showing that
\[|N_{\textrm{disk}}(\lambda)-N_D(\lambda)|=O(\lambda^{2/3}),\]
\begin{equation}\label{eq:latticeerror}
N_D(\lambda)=\frac 14\lambda^2-\frac 12\lambda + O(\lambda^{2/3}).
\end{equation}
The strategy has been improved more recently by Guo, Wang, and Wang \cite{gww} and then by those three authors together with M\"uller \cite{gmww}. They obtain better estimates on both the difference between the two counting functions and on the remainder in the lattice point problem. Their state of the art estimate is given in \cite{gmww} and is
\[R_{\textrm{disk}}(\lambda) = O(\lambda^{131/208}(\log\lambda)^{18627/8320}).\]
It should also be noted that Eswarathasan, Polterovich, and Toth have shown that the remainder is \emph{not} $o(\lambda^{1/2})$ \cite{ept}. The question of whether the remainder is $O(\lambda^s)$ for $s\in[1/2,131/208]$ remains open.

The relationship with the present work is that there is a simple way to write the eigenvalue counting function $N_{\textrm{disk}}$ in terms of the Bessel phase. Zeroes of the $J$-Bessel functions are precisely the points where $\theta_{n}(z)=\frac{-\pi}{2}+m\pi$, $m\in\mathbb N$. So introduce the function
\[A(z,\nu):=\frac{1}{\pi}(\theta_{\nu}(z)+\frac{\pi}{2}).\]
 The point is that
\begin{equation}\label{eq:counting}
N_{\textrm{disk}}(\lambda) = \sum_{n\in\mathbb Z}\lfloor A(\lambda,|n|)\rfloor.
\end{equation}
There is an analogous way to rewrite the counting function for the lattice point problem. Introduce the function
\[B(z,\nu):=\frac{1}{\pi}(\sqrt{z^2-\nu^2}-\nu\arccos\frac{\nu}{z})+\frac 14.\]
A direct calculation shows that
\begin{equation}\label{eq:latticecounting}
N_{D}(\lambda) = \sum_{n\in\mathbb Z}\lfloor B(\lambda,|n|)\rfloor.
\end{equation}
We can thus compare $N_{\textrm{disk}}$ and $N_D$ by comparing $A(z,\nu)$ and $B(z,\nu)$.
\begin{proposition} The function $B(z,\nu)-A(z,\nu)$ is non-negative. Moreover it is polyhomogeneous conormal on the portion of $O_a$ with $z>\nu$, with leading orders 3 at ae, 3 at abe, and 0 at fe.
\end{proposition}
\noindent Note that the statement about leading orders implies that $B(z,\nu)-A(z,\nu)=O(\rho_{ae}^3\rho_{abe}^3)$. Writing this out, we see that whenever $z>\nu$,
\begin{equation}
0\leq B(z,\nu)-A(z,\nu)\leq Cz^{-1}(1-(\frac\nu z)^{1/3})^{3/2},
\end{equation}
which is a potentially useful estimate in its own right. 

Now let us prove the proposition.
\begin{proof} For non-negativity, a direct calculation shows that for each fixed $\nu$,
\[\lim_{z\to\infty}(B(z,\nu)-A(z,\nu))=0.\]
So it is enough to prove that $\frac{\partial}{\partial z}(B(z,\nu)-A(z,\nu))$ is non-\emph{positive}, which means proving that
\[\frac{1}{\pi}\frac{\sqrt{z^2-\nu^2}}{z}\leq\frac{2}{\pi^2zM_{\nu}^2(z)},\]
which is equivalent to showing that
\begin{equation}\label{eq:thisisinhorsley}
\frac{\pi}{2}M_{\nu}^2(z)\leq\frac{1}{\sqrt{z^2-\nu^2}}.
\end{equation}
By Nicholson's formula and the fact that $K_0$ is a decreasing function, the left-hand side satisfies
\[\frac{\pi}{2}M_{\nu}^2(z)\leq\frac{4}{\pi}\int_0^{\infty}K_0(2zt)\cosh(2\nu t)\, dt.\]
But the right-hand side can be evaluated using \cite[6.661.2]{gr} and gives precisely $(z^2-\nu^2)^{-1/2}$, proving non-negativity. An alternative proof of \eqref{eq:thisisinhorsley} may be found in \cite{horsley}.

Polyhomogeneity follows from proving polyhomogenity of $B(z,\nu)$. Near the intersection of abe and ae, $B(z,\nu)$ lifts to
\begin{equation}\label{eq:bznu}
\frac{1}{\pi}\mu^{-3}\eta^{-3}(\sqrt{1-\eta^6}-\eta^3\arccos(\eta^3))+\frac 14,
\end{equation}
which is polyhomogeneous. Near the intersection of abe and fe, some Taylor expansions show that $B(z,\nu)-\frac 14$ lifts to
\[\frac{4\sqrt{3}}{\pi\sqrt{2}}\hat\mu_{-}^{-3}f(\hat w_{-}),\]
where $f(\hat w_{-})$ is smooth down to $\hat w_-=0$ and has value 1 at $w_-=0$. This is polyhomogeneous. The statements about the leading orders follow from comparing the asymptotic expansions at abe and ae, in particular comparing \eqref{eq:bznu} and \eqref{eq:thetaabe}
\end{proof}
An immediate corollary which follows from \eqref{eq:counting} and \eqref{eq:latticecounting} is
\begin{corollary} $N_{\textrm{disk}}(\lambda)\leq N_{D}(\lambda)$.
\end{corollary}

In a forthcoming paper with M. Levitin and I. Polterovich \cite{lps22}, we use this observation to prove P\'olya's conjecture for a disk, Question~\ref{q:polya}. In particular, the error $O(\lambda^{2/3})$ from \eqref{eq:latticeerror} can be quantified -- that is, bounded by $p(\lambda)$ where $p(\lambda)$ is an explicit function of $\lambda$ of order $O(\lambda^{2/3})$. Thus
\[N_{\textrm{disk}}(\lambda)\leq N_{D}(\lambda)\leq\frac 14\lambda^2-\frac 12\lambda + p(\lambda),\]
which implies that P\'olya's conjecture is true for all sufficiently large $\lambda$. We then use a computer-assisted approach to check P\'olya's conjecture for all smaller $\lambda$, thus completing the proof. We also follow a similar program to prove the version of P\'olya's conjecture for a disk with Neumann boundary conditions \cite{lps22}.


\section{Asymptotics of the modulus-squared}

Our starting point is the analysis of $M_{\nu}^2(z)$. By Nicholson's formula \cite[6.664.4]{gr}, for all $z>0$ and all $\nu$,
\begin{equation}\label{eq:nicholsonv2}
M_{\nu}^2(z)=\frac{8}{\pi^2}\int_0^{\infty}K_0(2z\sinh t)\cosh(2\nu t)\, dt.
\end{equation}
We analyze the behavior of this for large $z$ and large $\nu$. The goal of this section is to prove the following theorem:

\begin{theorem}\label{thm:modsquared} The modulus-squared $M_{\nu}^2(z)$ lifts to a function which is polyhomogeneous conormal on $U_a$. 
\end{theorem}

The first thing we must understand is the behavior of the Bessel function $K_0(x)$. We have well-known asymptotic expansions for $K_0(x)$ at $x=0$ and at $x=\infty$. The following lemma shows that these expansions are polyhomogeneous conormal and illuminates the structure of the expansion at $x=\infty$. We define
\[G(x):=e^xK_0(x).\]
\begin{lemma} $G(x)$ has a polyhomogeneous conormal expansion at $x=0$, and also at $x=\infty$ in the variable $x^{-1}$. At $x=0$ this expansion has leading term $-\log x$. At $x=\infty$ it has only odd half-integer powers of $x^{-1}$, with leading term $\sqrt{\pi/2}x^{-1/2}$.
\end{lemma}
\begin{proof} By \cite[10.32.8]{dlmf},
\begin{equation}\label{eq:intforko}
e^xK_0(x)=\int_0^{\infty}e^{-x(\cosh t-1)}\, dt=\int_0^{1}e^{-x(\cosh t-1)}\, dt + \int_1^{\infty}e^{-x(\cosh t-1)}\, dt.
\end{equation}
We analyze each of the two integrals in \eqref{eq:intforko} separately.

For the first integral, the $x\to 0$ behavior is simple: the integrand is jointly smooth in $x$ and $t$ and so the integral is smooth as $x\to 0$. The $x\to\infty$ behavior is more interesting. Observe that $\cosh t - 1 = t^2f(t)$, where $f(t)$ is smooth, even, and has $f(0)=1$. If we let $\hat x=x^{-1}$ then our integral becomes
\[\int_0^1 e^{-\frac{t^2}{\hat x}f(t)}\, dt.\]
The integrand is not a smooth function of $t$ and $\hat x$. However, it is polyhomogeneous as soon as one performs a parabolic blow-up at $\{\hat x=t=0\}$. Integration in $t$ is then a b-fibration onto $[0,1)_{\hat x}$, which means that our first integral is polyhomogeneous in $\hat x=x^{-1}$ as desired.

For the second integral in \eqref{eq:intforko} we change variables to $r=e^{-t}$ and get
\[\int_0^{1/e} e^{-x(\frac r2+\frac 1{2r}-1)}\,\frac{dr}{r}.\]
First consider the regime when $x\to 0$. In this regime, we write the integral as
\[\int_0^{1/e} e^{-x(\frac r2-1)} e^{-\frac{x}{2r}}\,\frac{dr}{r}.\]
The integrand is not smooth in $x$ and $r$. However, it is polyhomogeneous as soon as we perform a (regular) blow-up at $\{x=r=0\}$, and integration is again a b-fibration onto $[0,1)_x$, giving the polyhomogeneous expansion we want as $x\to 0$. 

For the third integral, as $x\to\infty$, the integrand $e^{-x(\cosh t-1)}$ is actually a smooth function of $x^{-1}$ and $t^{-1}$, decaying to infinite order both as $x^{-1}\to 0$ and as $t^{-1}\to 0$. This yields a trivial polyhomogeneous expansion as $x\to\infty$, completing the proof. 

It is of course possible to read off the complete asymptotic expansions themselves from the pushforward theorem. However, the asymptotics are already well-known, so there is no need; we have now proved that these previously-known expansions are polyhomogeneous conormal. \end{proof}

We now rewrite Nicholson's formula as follows:
\begin{equation}\label{eq:nicholsonmod}
M_{\nu}^2(z)=\frac{4}{\pi^2}\int_0^{\infty}(1+e^{-4\nu t})G(2z\sinh t)e^{-h(z,\nu,t)}\, dt,
\end{equation}
where, now and throughout the manuscript, we write
\begin{equation}\label{eq:expshort}
h(z,\nu,t):=2z\sinh t - 2\nu t.
\end{equation}
It turns out to be easiest to break this into two pieces:
\begin{equation}\label{eq:nicholsonmodbroken}
M_{\nu}^2(z)=\widetilde M_{\nu}^2(z)+R_{\nu}^2(z),
\end{equation}
where
\begin{equation}\label{eq:nicholsonmainterm}
\widetilde M_{\nu}^2(z)=\frac{4}{\pi^2}\int_0^{\infty}G(2z\sinh t)e^{-h(z,\nu,t)}\, dt,
\end{equation}
\begin{equation}\label{eq:nicholsonremainder}
R_{\nu}^2(z)=\frac{4}{\pi^2}\int_0^{\infty}G(2z\sinh t)e^{-2z\sinh t - 2\nu t}\, dt.
\end{equation}
As indicated by the notation, $\widetilde M_{\nu}^2(z)$ is the main term and $R_{\nu}^2(z)$ is a remainder. For the moment we focus on the main term.

There are multiple scales relevant to the asymptotic analysis of \eqref{eq:nicholsonmainterm}. On the one hand, the first term has nontrivial asymptotic behavior at the scale $t\sim z^{-1}$. On the other hand, \eqref{eq:expshort} may be rewritten as
\[-2(\nu-z)\sinh t + 2\nu(\sinh t - t)\sim 2(\nu - z)t + \frac 13\nu t^3.\]
This has nontrivial asymptotic behavior at the scale $t\sim\frac{1}{v-z}\sim \nu^{-1/3}$. Indeed, this accounts for the appearance in the classical asymptotics \cite[10.19.8]{dlmf} of the variable $(\nu-z)/\nu^{1/3}$.

The multi-scale nature of this problem is reflected in the blow-up analysis by the fact that two different collections of blow-ups are needed to deal with the exponential term and with the remainder of the expression. When constructing a blown-up space on which the integrand in \eqref{eq:nicholsonmainterm} is nice, both sets of blow-ups must be present. Which order to do them in is more or less a matter of taste. However, we will proceed by first dealing with the exponential term and then with the remainder of the integrand. The exponential term has a $\nu^{-1/3}$ scale, so we define
\[\zeta:= z^{-1/3},\quad \mu=\nu^{-1/3}.\]
We now consider the integrand in \eqref{eq:nicholsonmainterm} as a function on the interior of a manifold with corners
\[X:= [0,\infty]_t \times [0,1)_{\zeta}\times [0,1)_{\mu}.\]
For a boundary defining function at $t=\infty$ we make the unusual choice to set
\[\tau:=e^{-t/6}.\] 
This exponential scale turns out to be more relevant to the analysis than a non-exponential scale would be. 
Note that the boundaries $\zeta=1$ and $\mu=1$ are arbitrary -- any value in $(0,\infty)$ may be chosen for each -- and do not play a significant role in the analysis. Finally, let
\[f_1(t):=\frac{\sinh t}{t};\quad f_2(t)=6\frac{\sinh t - t}{t^3},\]
and observe that both $f_1(t)$ and $f_2(t)$ are smooth, even functions of $t$ with $f_1(0)=f_2(0)=1$. With this notation we may write \eqref{eq:expshort} in either of the following two ways:
\begin{equation}\label{eq:expshort2}
h(z,\nu,t)=(2z-2\nu)tf_1(t)+\frac{\nu t^3}{3}f_2(t)=-2t(\mu^{-3}-\zeta^{-3})f_1(t) + \frac {t^3}3\mu^{-3}f_2(t).
\end{equation}
Recall that the integrand in \eqref{eq:nicholsonmainterm} is
\begin{equation}\label{eq:intnew}
G(2\zeta^{-3}tf_1(t))e^{-(-2t(\mu^{-3}-\zeta^{-3})f_1(t) + \frac {t^3}3\mu^{-3}f_2(t))}.
\end{equation}
It is now evident that, as expected, the two scales that are relevant are $t\sim\mu^{3}\sim\zeta^3$ (for both terms) and $t\sim\mu$ (for the exponential term).

Now we start performing a sequence of blow-ups on $X$, giving local projective coordinates as we do so. The goal of these blow-ups is to make the integrand in \eqref{eq:nicholsonmainterm} polyhomogeneous.  Begin by labeling the boundary hypersurfaces of $X$: tf for $t=0$, af for $\zeta=0$ (large values of the argument), of for $\mu=0$ (large values of the order), and tif for $t=\infty$ (that is, $\tau=0$). 

\begin{remark} Even without any blow-ups, \eqref{eq:expshort} is polyhomogeneous on $X$, with leading orders $-3$ at of and af, $1$ at tf, and $-6$ at tif, the latter of which can be seen by writing \eqref{eq:expshort} in the coordinates $(\mu,\zeta,\tau)$:
\[h(\mu,\zeta,\tau)=\frac{1}{\zeta^3\tau^6}(1-\tau^{12})  + \frac{12\log\tau}{\mu^3}.\]
The problem is that its negative exponential is not (to say nothing of the other term in \eqref{eq:nicholsonmainterm}). \end{remark}

Our \textbf{first blow-up} is of of$\cap$af, creating a new front face bf and a new manifold with corners which we call $X_1$. Introduce the following projective coordinates:
\begin{equation}\label{eq:introcoords1}
\eta:=\frac{\zeta}{\mu};\quad \lambda:=\frac{\mu}{\zeta};\quad w:=\eta-1.
\end{equation}
Then near bf$\cap$of, we may use the projective coordinates $(\zeta,\lambda,t)$ and near bf$\cap$af, we may use $(\mu,\eta,t)$. The coordinate $w$ is a boundary defining function for the lift of the diagonal $\{\zeta=\mu\}$ in $X_1$, and so near the diagonal we may use the coordinates $(w,\mu,t)$. Near the diagonal the exponent $h$ from \eqref{eq:expshort} is
\begin{equation}\label{eq:exp1}
-\frac{6wt}{\mu^3}f_1(t)f_3(w) +\frac{t^3}{3\mu^3}f_2(t),
\end{equation}
where
\[f_3(w):=\frac{3+3w+w^2}{3(w+1)^3}.\]
'Note that $f_3(w)$ is a smooth function of $w$ away from of and af, and that $f_3(0)=1$.

Our \textbf{second blow-up} is of bf$\cap$tf, creating a face which we call bif and a new mwc $X_2$. This face represents the behavior at the order $t\sim\nu^{-1/3}$. We introduce some more projective coordinates:
\begin{equation}\label{eq:introcoords2}
s:=\frac{t}{\mu};\quad r:=\frac{t}{\zeta};\quad \kappa:=\frac{\mu}{t};\quad \theta:=\frac{\zeta}{t}.
\end{equation}
There are various different regimes now: 
\begin{itemize}
\item Near bf$\cap$bif$\cap$of we can use $(\theta,\lambda,t)$;
\item Near bif$\cap$tf$\cap$of, $(\zeta,\lambda,r)$;
\item Near bf$\cap$bif$\cap$af, $(\kappa,\eta,t)$;
\item Near bif$\cap$tf$\cap$af, $(\mu,\eta,s)$.
\end{itemize}
Near $D\cap$bf$\cap$bif we can use $w$ as a replacement for $\eta$ (it has the advantage of being a defining function for the diagonal) and so we use $(\kappa,w,t)$; near $D\cap$bif$\cap$tf we similarly use $(\mu,w,s)$. In these two coordinate systems \eqref{eq:exp1} becomes

\begin{equation}\label{eq:exp2}
-\frac{6w}{t^2\kappa^3}f_1(t)f_3(w) +\frac{1}{3\kappa^3}f_2(t)
\end{equation}
in the former and

\begin{equation}\label{eq:exp3}
-\frac{6ws}{\mu^2}f_1(s\mu)f_3(w) +\frac{s^3}{3}f_2(s\mu)
\end{equation}
in the latter. 

Our \textbf{third blow-up} is of bif$\cap$tf, and this blow-up is quasihomogeneous, quadratic with respect to tf. In the diagonal coordinates $(\mu,w,s)$, this is a blow-up of $\{w=\mu=0\}$ and introduces the new boundary defining function $\sqrt{|w|+\mu^2}$. Call the new front face bf$_0$ and the new manifold with corners $X_3$. The lift of $h(\mu,\zeta,t)$ is polyhomogeneous on $X_3$ with leading order $-3$ at af, of, and bf, leading order $-2$ at bif, leading order 0 at bf$_0$, leading order 1 at tf, and leading order $-6$ at tif. We introduce yet more projective coordinates:
\begin{equation}\label{eq:introcoords3}
s':=\frac{s}{\mu^2}=\frac{t}{\mu^3},\quad \kappa'=\frac{\mu}{\sqrt s}=\frac{\mu^{3/2}}{\sqrt t},\quad r':=\frac{r}{\zeta^2}:=\frac{t}{\zeta^3},\quad \theta':=\frac{\zeta}{\sqrt r}=\frac{\zeta^{3/2}}{\sqrt t}.
\end{equation}
New coordinate regimes are as follows. Near bf$\cap$bif$\cap$of and bf$\cap$bif$\cap$af we can use $(\theta,\lambda,t)$ and $(\kappa,\eta,t)$ as before. For the other regimes:
\begin{itemize}
\item Near bif$\cap$bf$_0\cap$of we use $(\theta',\lambda,\sqrt r)$;
\item Near bf$_0\cap$tf$\cap$of, $(\zeta,\lambda,r')$;
\item Near bif$\cap$bf$_0\cap$af, $(\kappa',\eta,\sqrt s)$;
\item Near bf$_0\cap$tf$\cap$af, $(\mu,\eta,s')$.
\end{itemize}
As before, near the diagonal, we can always replace $\eta$ with $w$. The space $X_3$ (away from tif) and local projective coordinates on it are illustrated in Figure \ref{fig:x3}.

\begin{figure}
\centering
\begin{tikzpicture}
\draw[black,thick,->](-1,1)--(-1,3);
\draw[black,thick,->](1,1)--(1,3);
\draw[black,thick,->](-2,-1)--(-3.5,-2);
\draw[black,thick,->](2,-1)--(3.5,-2);
\draw[thick](0,0.9) arc (270:281.6:5);
\draw[thick](0,0.9) arc (270:258.4:5);
\draw[thick](0,-1.4) arc (270:291.75:5.4);
\draw[thick](0,-1.4) arc (270:248.25:5.4);
\draw[thick](-1.5,0) .. controls (-1.5,0.5) and (-1.25,0.9) .. (-1,1);
\draw[thick](-1.5,0) .. controls (-1.75,-0.1) and (-2,-0.5) .. (-2,-1);
\draw[thick](1.5,0) .. controls (1.5,0.5) and (1.25,0.9) .. (1,1);
\draw[thick](1.5,0) .. controls (1.75,-0.1) and (2,-0.5) .. (2,-1);
\draw[thick](-1.5,0) .. controls (-0.5,-0.3) and (0.5,-0.3) .. (1.5,0);

\draw[thick,->](-1.2,1.18)--(-1.2,1.8);
\draw[thick,->] (-1.2,1.18) arc (120:150:1);
\draw[thick,->] (-1.7,0.15) arc (180:165:2);
\draw[thick,->] (-1.7,0.15) arc (120:150:1);
\draw[thick,->] (-2.2,-0.85) arc (180:165:2);
\draw[thick,->] (-2.2,-0.85)--(-2.65,-1.15);
\draw[thick,->] (-0.8,1.16) arc (255:265:2);
\draw[thick,->] (-1.3,0.13) arc (252:267:2);
\draw[thick,->] (-1.8,-0.88) arc (250:265:2);

\draw[thick,->](1.2,1.18)--(1.2,1.8);
\draw[thick,->] (1.2,1.18) arc (60:30:1);
\draw[thick,->] (1.7,0.15) arc (0:15:2);
\draw[thick,->] (1.7,0.15) arc (60:30:1);
\draw[thick,->] (2.2,-0.85) arc (0:15:2);
\draw[thick,->] (2.2,-0.85)--(2.65,-1.15);
\draw[thick,->] (0.8,1.16) arc (285:275:2);
\draw[thick,->] (1.3,0.13) arc (288:273:2);
\draw[thick,->] (1.8,-0.88) arc (290:275:2);

\draw[thick](0,1.06) arc (270:275:2);
\draw[thick,->](0,1.06) arc (270:265:2);
\draw[thick](0,-1.2) arc (270:278:2);
\draw[thick,->](0,-1.2) arc (270:262:2);
\draw[thick](0,-0.05) arc (270:277:2);
\draw[thick,->](0,-0.05) arc (270:263:2);

\node at (0,1.23) {$w$};
\node at (0,-1.05) {$w$};
\node at (0,0.1) {$w$};
\node at (-0.6,1.36) {$\lambda$};
\node at (-1.0,0.33) {$\lambda$};
\node at (-1.5,-0.68) {$\lambda$};
\node at (0.6,1.31) {$\eta$};
\node at (1.0,0.28) {$\eta$};
\node at (1.5,-0.73) {$\eta$};

\node at (-1.4,1.68) {$t$};
\node at (-1.6,1.18) {$\theta$};
\node at (-2.0,0.6) {$\sqrt{r}$};
\node at (-2.1,0.1) {$\theta'$};
\node at (-2.4,-0.35) {$r'$};
\node at (-2.8,-1) {$\zeta$};

\node at (1.4,1.68) {$t$};
\node at (1.6,1.18) {$\kappa$};
\node at (2.0,0.6) {$\sqrt{s}$};
\node at (2.2,0.1) {$\kappa'$};
\node at (2.4,-0.35) {$s'$};
\node at (2.8,-1) {$\mu$};

\node at (-1.25,2.75) {$t$};
\node at (1.25,2.75) {$t$};
\node at (-3.3,-1.65) {$\zeta$};
\node at (3.3,-1.65) {$\mu$};

\node at (0,2) {bff};
\node at (0,0.6) {bif};
\node at (0,-0.55) {bf$_0$};
\node at (-2.5,1.5) {of};
\node at (2.5,1.5) {af};
\node at (0,-1.8) {tf};

\end{tikzpicture}
\caption{The space $X_3$ and local projective coordinate systems.}
\label{fig:x3}
\end{figure}
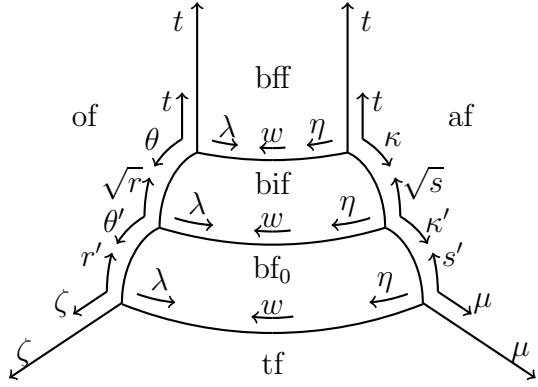

Our \textbf{fourth blow-up} is necessary both for the $G$ term in \eqref{eq:nicholsonmod} and for the exponential term therein. We blow up of$\cap$tf and af$\cap$tf, cubically with respect to tf. This creates two new boundary hypersurfaces of$_0$ and af$_0$. The term $h$ has order zero at each. Denote the new manifold with corners by $X_4$.

The \textbf{fifth blow-up} resolves the singularity in the exponential term. We blow up $D\cap$bif, quadratically with respect to $D$ This creates a new face ff and a new manifold with corners $X_5$. It splits bif in two and we give the two halves the names obif and abif according to whether they are adjacent to of or to af.

Finally, the \textbf{sixth blow-up} is two disjoint blow-ups: $D\cap$bf, quadratic with respect to $D$, creating bff, and $D\cap$bf$_0$, quadratic with respect to $D$, creating ff$_0$. This creates a new manifold with corners $X_6$. The fifth and sixth blow-ups have split each of the faces bf and bf$_0$ in two, and we give them names obf, abf, obf$_0$, abf$_0$ respectively according to whether they are adjacent to of or to af. The final space, at least the portion of it away from tif, is illustrated in Figure \ref{fig:x6}.

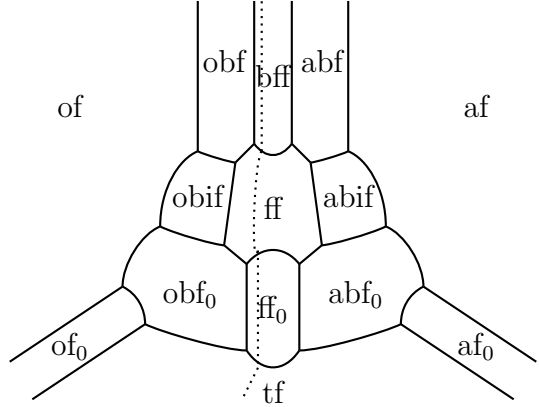
\begin{figure}
\centering
\begin{tikzpicture}
\draw[black,thick](-1,1)--(-1,3);
\draw[black,thick](1,1)--(1,3);
\draw[black,thick](-2,-0.8)--(-3.5,-1.8);
\draw[black,thick](-1.7,-1.3)--(-3.2,-2.3);
\draw[black,thick](2,-0.8)--(3.5,-1.8);
\draw[black,thick](1.7,-1.3)--(3.2,-2.3);
\draw[thick](-1.5,0) .. controls (-1.5,0.5) and (-1.25,0.9) .. (-1,1);
\draw[thick](-1.5,0) .. controls (-1.75,-0.1) and (-2,-0.5) .. (-2,-0.8);
\draw[thick](1.5,0) .. controls (1.5,0.5) and (1.25,0.9) .. (1,1);
\draw[thick](1.5,0) .. controls (1.75,-0.1) and (2,-0.5) .. (2,-0.8);
\draw[thick](-2,-0.8) .. controls (-1.8,-0.9) and (-1.7,-1.1) .. (-1.7,-1.3);
\draw[thick](2,-0.8) .. controls (1.8,-0.9) and (1.7,-1.1) .. (1.7,-1.3);

\draw[thick](-0.25,1.1) .. controls (-0.25,3) .. (-0.25,3);
\draw[thick](0.25,1.1) .. controls (0.25,3) .. (0.25,3);

\draw[thick](-1,1) .. controls (-0.9,0.95) and (-0.55,0.85) .. (-0.5,0.85);
\draw[thick](1,1) .. controls (0.9,0.95) and (0.55,0.85) .. (0.5,0.85);

\draw[thick](-0.5,0.85) .. controls (-0.5,0.85) and (-0.25,1.1) .. (-0.25,1.1);
\draw[thick](0.5,0.85) .. controls (0.5,0.85) and (0.25,1.1) .. (0.25,1.1);

\draw[thick](-1.5,0) .. controls (-1.3,-0.1) and (-0.75,-0.25) .. (-0.65,-0.25);
\draw[thick](1.5,0) .. controls (1.3,-0.1) and (0.75,-0.25) .. (0.65,-0.25);

\draw[thick](-0.65,-0.25) .. controls (-0.65,-0.25) and (-0.5,0.85) .. (-0.5,0.85);
\draw[thick](0.65,-0.25) .. controls (0.65,-0.25) and (0.5,0.85) .. (0.5,0.85);

\draw[thick](-0.65,-0.25) .. controls (-0.65,-0.25) and (-0.35,-0.5).. (-0.35,-0.5);
\draw[thick](0.65,-0.25) .. controls (0.65,-0.25) and (0.35,-0.5) .. (0.35,-0.5);

\draw[thick](-0.35,-0.5)--(-0.35,-1.65);
\draw[thick](0.35,-0.5)--(0.35,-1.65);

\draw[thick](-1.7,-1.3) .. controls (-1.5,-1.4) and (-0.65,-1.65) .. (-0.35,-1.65);
\draw[thick](1.7,-1.3) .. controls (1.5,-1.4) and (0.65,-1.65) .. (0.35,-1.65);

\draw[thick] (-0.25,1.1) .. controls (-0.1,0.9) and (0.1,0.9) .. (0.25, 1.1);
\draw[thick] (-0.35,-0.5) .. controls (-0.15,-0.25) and (0.15,-0.25) .. (0.35, -0.5);
\draw[thick] (-0.35,-1.65) .. controls (-0.15,-1.95) and (0.15,-1.95) .. (0.35,-1.65);

\node at (0,0.25) {ff};
\node at (0,2) {bff};
\node at (0,-1.1) {ff$_0$};
\node at (-0.65,2.2) {obf};
\node at (0.65,2.2) {abf};
\node at (-1,0.4) {obif};
\node at (1,0.4) {abif};
\node at (-1.1,-0.9) {obf$_0$};
\node at (1.1,-0.9) {abf$_0$};
\node at (0,-2.2) {tf};
\node at (-2.7,-1.6) {of$_0$};
\node at (2.7,-1.6) {af$_0$};
\node at (-2.7, 1.6) {of};
\node at (2.7,1.6) {af};

\draw[thick,dotted](-0.15,1.0)--(-0.15,3);
\draw[thick,dotted](-0.15,1.0) .. controls (-0.25,0.9) and (-0.3,-0.6) .. (-0.2,-0.32);
\draw[thick,dotted] (-0.2,-0.32)--(-0.2,-1.88);
\draw[thick,dotted] (-0.2,-1.88)--(-0.4,-2.28);

\end{tikzpicture}
\caption{The space $X_6$ (away from tif). The dotted line is the boundary of the lift of $U_a\times[0,\infty]_t$.}
\label{fig:x6}
\end{figure}

Now we begin analyzing each of the terms in \eqref{eq:intnew}.

\begin{lemma} The function $\exp[-h(z,\nu,t)]$ pulls back to a polyhomogeneous conormal function on an open set $\hat U\subseteq X_5$ containing the lift of $U_a\times[0,\infty]_t$, with empty index sets at tif, bf$\cap\hat U$, af, and abif.
\end{lemma}
\begin{proof}The proof uses Theorem \ref{thm:exppc}. In particular, we will show that away from of and obif, and the portion of bf outside $\hat U$, $h$ is type 1 at ff, bf$_0$, af$_0$, of$_0$, and tf, and is type 2 at tif, af, bf$\cap\hat U$, and abif. The result follows immediately from Theorem \ref{thm:exppc}.

It is easy to check that the hypersurfaces ff, bf$_0$, af$_0$, of$_0$, and tf are type 1; indeed the leading order of $h$ is zero at the first four of these and $1$ at tf, which means that $h$ is continuous down to each. We must show that each of tif, af, bf$\cap\hat U$, and abif are type 2.

First consider $h$ near tif. Here, away from of, we may use the coordinates $(\tau,\mu,\eta)$, in which $h$ is
\[\frac{2(1-\tau^{12})}{\mu^3\eta^3\tau^6}+\frac{12\log\tau}{\mu^3}.\]
The coordinates $\tau,\mu,\eta$ are boundary defining functions for tif, bf, and af respectively, so the dominant term is $2\mu^{-3}\eta^{-3}\tau^{-6}$ and we have the required growth at each boundary hypersurface. Thus tif, bf, and af are all type 2 boundary hypersurfaces whenever these coordinates are legitimate, that is, in any neighborhood of the corner tif$\cap$bf$\cap$af which does not intersect any other boundary hypersurfaces. This is the only coordinate system we need near tif, so we have shown that tif is type 2 overall.

Now we move to bf$\cap\hat U$. The triple intersection af$\cap$abf$\cap$abif has coordinates $(\kappa,\eta,t)$ and there $h$ is
\[\frac{2f_1(t)}{t^2\kappa^3}(\frac{1}{\eta^3}-1)+\frac{1}{3\kappa^3}f_2(t).\]
The dominant term is $\frac{2f_1(t)}{t^2\kappa^3\eta^3}$ and this means we have the required growth at each of the three boundary hypersurfaces.

We must now examine h near bf$\cap$ff. Before the creation of ff we use $(\kappa,w,t)$. To create ff we perform a quadratic blow-up of $\{t=w=0\}$, quadratic with respect to $\{w=0\}$. This means that on the interior of bf$\cap$ff, good coordinates are 
\begin{equation}\label{eq:coordsforS}
(\kappa,W:=\frac{w}{t^2},t),
\end{equation}
in which, using \eqref{eq:exp2}, we find
\[h=-\frac{6W}{\kappa^3}f_1(t)f_3(Wt^2) + \frac{1}{3\kappa^3}f_2(t)=\frac 13\kappa^{-3}(f_2(t)-18Wf_1(t)f_3(Wt^2)).\]
Since $f_1(0)=f_2(0)=f_3(0)=1$, we see that the key value is $W=\frac{1}{18}$. When $W<1/18$, at least for small $t$, $h$ grows at the required rate at bf. So as long as $\hat U$ is, near bf$\cap$ff, contained in $\{W<c<1/18\}$ for some $c$, bf$\cap\hat U$ is a type 2 bhs. The point $(\kappa,W,t)=(0,1/18,0)$ lifts to be in the interior of obf$\cap$ff on $X_6$. So such a region does contain the lift of $U_a\times[0,\infty]_t$, whose lift to $X_6$ is in turn illustrated in Figure \ref{fig:x6}.

On the other hand, near bf$\cap$ff$\cap$abif, good coordinates are
\[(\kappa,w_{-}=\sqrt{-w},\frac{t}{w_{-}}),\]
in which
\[h=6\kappa^{-3}(\frac{t}{w_{-}})^{-2}f_1(w_{-}\cdot\frac{t}{w_-})f_3(-w_-^2)+\frac 1{3}\kappa^{-3}f_2(w_{-}\cdot\frac{t}{w_-}).\]
The leading term is $6\kappa^{-3}(\frac{t}{w_-})^{-2}$, so we have the required growth at abif (with bdf $t/w_{-}$) and bf (with bdf $\kappa$). This completes the proof that bf is type 2 overall.

For abif, we must analyze the coordinate neighborhoods abif$\cap$af$\cap$bf$_0$ and abif$\cap$ff$\cap$bf$_0$. In the former, coordinates are $(\kappa',\eta,\sqrt s)$, in which $h$ is given by
\[2(\kappa')^{-2}(\frac{1}{\eta^3}-1)f_1(\kappa'\sqrt s^3)+\frac 13\sqrt s^6f_2(\kappa'\sqrt s^3).\]
The dominant term is $2(\kappa')^{-2}\eta^{-3}$, which satisfies the required growth estimate. On the other hand, near abif$\cap$ff$\cap$bf$_0$, we had $(\kappa',w,\sqrt s)$ before the creation of ff. Creating ff is blowing up $\{\kappa'=w=0\}$, quadratically with respect to $\{w=0\}$. So coordinates near the aforementioned corner are
\[(w_{-},\frac{\kappa'}{w_{-}},\sqrt s).\]
In these coordinates $h$ is given by
\[6(\frac{\kappa'}{w_{-}})^{-2}f_1(\frac{\kappa'}{w_-}w_{-}\sqrt s^3)f_3(-w_{-}^2)+\frac 13\sqrt s^6f_2(\frac{\kappa'}{w_{-}}w_-\sqrt s^3),\]
which has dominant term $6(\kappa'/w_-)^{-2}$ and therefore grows at abif. This means that abif is type 2.

Finally, for af, we just need to analyze af$\cap$af$_0\cap$bf$_0$. Before the creation of af$_0$ we had $(\mu,\eta,s')$. Creating af$_0$ blows up $\{s'=\eta=0\}$, cubically with respect to $s'$. Cooordinates near af$\cap$af$_0\cap$bf$_0$ are thus
\[(\mu,\frac{\eta}{(s')^{1/3}},(s')^{1/3}).\]
In these coordinates,
\[h=2((\frac{\eta}{(s')^{1/3}})^{-3}-((s')^{1/3})^3)f_1(((s')^{1/3})^3\mu^3)+\frac 13((s')^{1/3})^9\mu^6f_2(((s')^{1/3})^3\mu^3).\]
The dominant term is $2(\frac{\eta}{(s')^{1/3}})^{-3}$, which does have the required growth estimate at af (with bdf $\eta/(s')^{1/3}$). This means that af is type 2 as well, and this completes the proof of the proposition.
\end{proof}

Now we consider the $G$ term in the integrand of \eqref{eq:nicholsonmainterm}.

\begin{lemma}\label{lem:gonx4} The function $G(2z\sinh t)$ pulls back to a polyhomogeneous conormal function on $X_4$.
\end{lemma}
\begin{proof} This is a consequence of Melrose's pullback theorem and Proposition \ref{prop:easybmap}.
\end{proof}
Since the blow-down map from $X_5$ to $X_4$ is a b-map, it is an immediate consequence that $G(2z\sinh t)$ is a polyhomogeneous conormal function on $X_5$.

Overall, we claim:
\begin{proposition} The integrand in \eqref{eq:nicholsonmainterm}, namely
\[G(2z\sinh t)e^{-h(z,\nu,t)},\]
pulls back to a polyhomogeneous conormal function on $\hat U$, with empty index sets at tif, bff, abf, af, and abif.
\end{proposition}
\begin{proof} This follows from the preceding two lemmas and the fact that polyhomogeneous conormal functions form an algebra.
\end{proof}

This allows us, now, to prove that 
\begin{proposition} The term $\widetilde M_{\nu}^2(z)$ satisfies the conclusion of Theorem \ref{thm:modsquared}.
\end{proposition}
\begin{proof} As a result of the previous proposition, the integrand in \eqref{eq:intnew} is polyhomogeneous conormal on the lift of $U_a\times [0,\infty]$, which is contained in $\hat U$. Since $X_6$ is a blow-up of $X_5$, the same is true on the lift of $U_a\times[0,\infty]_t\subseteq X_6$. By Melrose's pushforward theorem and Proposition \ref{prop:bfibx6}, the integral \eqref{eq:intnew} is polyhomogeneous conormal on $U_a$.
\end{proof}

Finally, Theorem \ref{thm:modsquared} folllows from the following analysis of the remainder term:
\begin{proposition}\label{prop:remonq} The remainder term $R_{\nu}^2(z)$ is polyhomogeneous conormal on $\widetilde Q$.
\end{proposition}
\begin{proof}
We claim that the integrand in \eqref{eq:nicholsonremainder} is polyhomogeneous conormal on $X_4$. Given this claim, the proposition follows from the pushforward theorem and Proposition \ref{prop:bfibx4}. To prove the claim, note that by Lemma \ref{lem:gonx4} it suffices to show that
\begin{equation}\label{eq:exprem}
e^{-2z\sinh t-2\nu t}
\end{equation}
is polyhomogeneous conormal on $X_4$.

But this is easy enough to show. Near tif, the exponent $2z\sinh t + 2\nu t$ is polyhomogeneous conormal on $X$ (using $\tau=e^{-t/6}$ as a defining function as usual), and all three of the boundary hypersurfaces of, af, and tif are type 2 bhs for it. Thus by Theorem \ref{thm:exppc}, $e^{-2z\sinh t-2\nu t}$ is polyhomogeneous conormal on $X$ in a neighborhood of tif. By the pullback theorem it is also polyhomogeneous conormal on $X_4$ in a neighborhood of tif. On the other hand, away from tif (that is, for $t$ bounded above), Proposition \ref{prop:easybmap} tells us that, since $e^{-x}$ is polyhomogeneous conormal on $[0,\infty]$, $e^{-2z\sinh t}$ is polyhomogeneous conormal on $X_4$. And by Proposition \ref{prop:easybmap2}, $e^{-\nu t}$ is polyhomogeneous conormal on $X_4$ away from tif. Since polyhomogeneous conormal functions form an algebra, the claim follows.
\end{proof}

Combining the previous two Propositions yields Theorem \ref{thm:modsquared}.


\section{Refined asymptotics of the modulus-squared for $\nu>>z$}

In the previous section, the exponential growth of $M_{\nu}^2(z)$ near oe and obe complicated some of our statements. Here we analyze that exponential growth, characterize it precisely, and use this characterization to prove the following:

\begin{theorem}\label{thm:modsquaredrefined} The following lifts to a polyhomogeneous conormal function on $U_o$:
\[M_{\nu}^2(z)e^{2\sqrt{\nu^2-z^2}-2\nu\cosh^{-1}(\nu/z).}\]
\end{theorem}

As before we write $M_{\nu}^2(z)=\widetilde M_{\nu}^2(z)+R_{\nu}^2(z)$ and first discuss the part corresponding to $\widetilde M_{\nu}^2(z)$, which is given by \eqref{eq:nicholsonmainterm}. We do this by comparing the exponential portion of the integrand to its maximum value, via a variant on Laplace's method for integral estimation. We only care about the regime with $\nu>3z$ so assume $\nu/z>1$. Under that assumption, the exponent $h(z,\nu,t)$ is minimized at
\[t=t(z,\nu):=\cosh^{-1}(\nu/z)=\log((\nu/z)+\sqrt{(\nu/z)^2-1})\]
and we denote its value $h(z,\nu,t(z,\nu))$ by $h_{\min}(z,\nu)$:
\begin{multline}h_{\min}(z,\nu)=2z\sinh(\cosh^{-1}(\nu/z))-2\nu\cosh^{-1}(\nu/z)\\=2z\sqrt{(\nu/z)^2-1}-2\nu\log((\nu/z)+\sqrt{(\nu/z)^2-1}).\end{multline}
 We write
\[T(z,\nu):=e^{-h_{\min}(z,\nu)}.\]
The idea is that \eqref{eq:hardpart} should equal $T(z,\nu)$ times something polyhomogeneous on $U$. 

Unfortunately, technical complications in the proof require us to change the smooth structure on $\widetilde Q$. Let $\widetilde Q^{\sharp}$ be the same manifold with corners as $\widetilde Q$, but with boundary defining function at obe changed to the square root of the original one, i.e. $\sqrt{\rho_{obe}}$. Let $U^{\sharp}$ be the same neighborhood as $U$, but with the new smooth structure. It is immediate that any function which is polyhomogeneous conormal on $\widetilde Q$ lifts to a function which is polyhomogeneous on $\widetilde Q^{\sharp}$; indeed the identity map $\iota:\widetilde Q^{\sharp}\to\widetilde Q$ is a b-map.

Now let $\chi(z,\nu)$ be a cutoff function which is 1 when $w/\mu^2\geq 3$ and $0$ when $w/\mu^2\leq 2$. Our goal is to prove the following:
\begin{theorem}\label{thm:hardpart}
The expression
\begin{equation}\label{eq:hardpart}
\frac{\chi(z,\nu)}{T(z,\nu)}\int_0^{\infty}G(2z\sinh t) e^{-2z\sinh t + 2\nu t}\, dt
\end{equation}
is polyhomogeneous conormal on $U^{\sharp}$.
\end{theorem}
To prove this, we need to show that the following is polyhomogeneous on $U^{\sharp}$:
\begin{equation}\label{eq:hardpartmod}
\int_0^{\infty}\chi(z,\nu)G(2z\sinh t)e^{-2z\sinh t+2\nu t+2\sqrt{\nu^2-z^2}-2\nu\cosh^{-1}(\nu/z)}\, dt.
\end{equation}
We will do this as before, by showing that the integrand is polyhomogeneous conormal on an appropriate blown-up space and then applying the pushforward theorem.

The blown-up space in question starts with $X_6$. Now change boundary defining functions at obf, obif, and obf$_0$ -- all the faces which map to obe under the lifted projection in $t$ -- to be the square root of the prior bdfs. Call the resulting space $X_6^{\sharp}$. Any function which is smooth on $X_6$ is smooth on $X_6^{\sharp}$, but again, not vice versa.

Let $S$ be the set
\[S:=\{(z,\nu,t)\ :\ t=t(z,\nu)=\cosh^{-1}(\nu/z)\},\]
which is the set where the integrand is maximized. This lifts to sets on each of our blown-up spaces, and abusing notation, we denote each of these by $S$ as well.

Now we create a new space $X^{\sharp}$ by making three additional blow-ups. First we blow up of$\cap$tif, quadratically with respect to of, creating a new face tof. Then we blow up the intersection obf$\cap S$, and we do so cubically with respect to $S$, creating a new face obfx. Finally, we blow up the intersection tof$\cap S$, cubically with respect to $S$, creating a new face tofx. This, of course, requires:
\begin{proposition} $S$ is a p-submanifold of $[X_6^{\sharp};\textrm{of}\cap\textrm{tif}]$ in a neighborhood of obf and tofx.
\end{proposition}
\begin{proof} Near the interior of the boundary face obf, $S$ is a surface which smoothly and transversely intersects obf and thus is automatically a p-submanifold. However, we need to analyze a neighborhood of each endpoint of $S\cap$ obf. As we will see, the endpoints of $S\cap$ obf are at ff and at tof, and $S$ intersects both ff and tof transversely, meaning it is a p-submanifold.

Coordinates on $X_6$ valid near the interior of ff$\cap$obf, away from bff and obif, are given by namely $(t,W,\kappa)$ as in \eqref{eq:coordsforS}. However, $\kappa$ must be replaced by $\sqrt{\kappa}$ on $X_6^{\sharp}$, where obif is $\{\sqrt{\kappa}=0\}$. Tracking $t=\cosh^{-1}(\nu/z)$ through the various coordinates, we see it is
\[t=\cosh^{-1}(\nu/z)=\cosh^{-1}(\zeta^3/\mu^3)=\cosh^{-1}((w+1)^3)=\cosh^{-1}((Wt^2+1)^3).\]
Solving for $W$ yields:
\[W=\frac{(\cosh t)^{1/3}-1}{t^2},\]
which is smooth at $t=0$ (that is, $t=0$ is a removable singularity), and even in $t$, with $W(0)=1/6$. 
Thus $\{(t,W(t),\sqrt\kappa)\}$ is indeed a p-submanifold in the coordinates $(t,W,\sqrt\kappa)$, and moreover the intersection of $S$, obf, and ff occurs at $W=1/6$. To make the p-submanifold even more explicit, let us introduce the new coordinate
\[\hat W:=W-\frac{(\cosh t)^{1/3}-1}{t^2}.\]
The coordinate system $(t,\hat W,\sqrt\kappa)$ is also valid in the same region, and in this region, $S=\{\hat W=0\}$. 

As for tof, coordinates on $X_6^{\sharp}$ near the triple intersection tif, of, obf are given by $(\tau,\lambda,\sqrt\zeta)$, and $S$ is the set where $-6\log\tau=\cosh^{-1}(\lambda^{-3})$. Taking $\cosh$ of both sides (which introduces a spurious root, which we ignore) shows that this is
\[\tau^{-6}(1+\tau^{12})=2\lambda^{-3}.\]
When tof is created, we are blowing up $\{\tau=\lambda=0\}$, quadratically with respect to $\lambda$. Define
\begin{equation}\label{eq:introcoords5}
\sigma:=\frac{\tau}{\sqrt\lambda},\quad \lambda'=\frac{\lambda}{\tau^2}.
\end{equation}
Projective coordinates near tof$\cap$tif$\cap$obf are $(\sigma,\sqrt{\lambda},\sqrt{\zeta})$, and projective coordinates near tof$\cap$of$\cap$obf are $(\tau,\lambda',\sqrt{\zeta})$. As we will see $S$ intersects the center of tof$\cap$obf and therefore we can use either set for the present purposes. In the set $(\sigma,\sqrt{\lambda},\sqrt{\zeta})$, $S$ is the collection of points satisfying
\[(1+\sigma^{12}\sqrt{\lambda}^{12})=2\sigma^6.\]
By the quadratic formula, focusing on the root we care about, this is the set where
\[\sigma=(\frac{1-\sqrt{1-\sqrt\lambda^{12}}}{\sqrt\lambda^{12}})^{1/6}.\]
The right-hand side is a smooth function of $\sqrt\lambda$, equal to $2^{-1/6}$ at $\sqrt\lambda=0$. So $S$ is a p-submanifold. As before, we make this explicit by introducing the new coordinate
\[\hat\sigma:=\sigma-(\frac{1-\sqrt{1-\sqrt\lambda^{12}}}{\sqrt\lambda^{12}})^{1/6}.\]
Then $(\hat\sigma,\sqrt{\lambda},\sqrt{\zeta})$ are good coordinates near the triple intersection of tof, tif, and obf, and in these coordinates $S$ is given by $\{\hat\sigma=0\}$.
So $S$ is a p-submanifold at both endpoints of $S\cap$ obf and in the interior and therefore is a p-submanifold in the required region.
\end{proof}

The space $X^{\sharp}$ is illustrated in Figure \ref{fig:xsharp}. We claim:
\begin{proposition}\label{prop:keyprop} The integrand in \eqref{eq:hardpartmod} is polyhomogeneous conormal on a neighborhood of the support of $\chi(z,\nu)$ in $X^{\sharp}$, with empty index sets at all faces except tofx, obfx, ff, ff$_0$, and tf.
\end{proposition}

\begin{figure}
\centering
\begin{tabular}{p{6cm} p{6cm}}
\begin{subfigure}[b]{0.3\textwidth}
\centering
\begin{tikzpicture}
\draw[black,thick](-1,1)--(-1,3);
\draw[black,thick](1,1)--(1,3);
\draw[black,thick](-2,-0.8)--(-2.75,-1.3);
\draw[black,thick](-1.7,-1.3)--(-2.45,-1.8);
\draw[black,thick](2,-0.8)--(2.75,-1.3);
\draw[black,thick](1.7,-1.3)--(2.45,-1.8);
\draw[thick](-1.5,0) .. controls (-1.5,0.5) and (-1.25,0.9) .. (-1,1);
\draw[thick](-1.5,0) .. controls (-1.75,-0.1) and (-2,-0.5) .. (-2,-0.8);
\draw[thick](1.5,0) .. controls (1.5,0.5) and (1.25,0.9) .. (1,1);
\draw[thick](1.5,0) .. controls (1.75,-0.1) and (2,-0.5) .. (2,-0.8);
\draw[thick](-2,-0.8) .. controls (-1.8,-0.9) and (-1.7,-1.1) .. (-1.7,-1.3);
\draw[thick](2,-0.8) .. controls (1.8,-0.9) and (1.7,-1.1) .. (1.7,-1.3);

\draw[thick](-0.25,1.1) .. controls (-0.25,3) .. (-0.25,3);
\draw[thick](0.25,1.1) .. controls (0.25,3) .. (0.25,3);

\draw[thick](-1,1) .. controls (-0.9,0.95) and (-0.55,0.85) .. (-0.5,0.85);
\draw[thick](1,1) .. controls (0.9,0.95) and (0.55,0.85) .. (0.5,0.85);
\draw[thick](0.5,0.85) .. controls (0.5,0.85) and (0.25,1.1) .. (0.25,1.1);

\draw[thick](-0.45,0.9) .. controls (-0.35,0.8) and (-0.2, 0.95) .. (-0.3,1.05);
\draw[thick](-0.5,0.85)--(-0.45,0.9);
\draw[thick](-0.3,1.05)--(-0.25,1.1);

\draw[thick](-0.45,0.9) .. controls (-0.6,1.4) and (-0.7,2) .. (-0.7,3);
\draw[thick](-0.3,1.05) .. controls (-0.45,1.55) and (-0.52,2.15) .. (-0.52,3);

\draw[thick,->](-2.0,2.8)--(-0.57,2.4);
\node at (-2.4,2.8) {obfx};
\node at (-2.4, 1.9) {obf};
\draw[thick,->](-2.1,1.9)--(-0.85,1.4);
\draw[thick,->](-2.1,1.9)--(-0.3,1.75);

\draw[thick](-1.5,0) .. controls (-1.3,-0.1) and (-0.75,-0.25) .. (-0.65,-0.25);
\draw[thick](1.5,0) .. controls (1.3,-0.1) and (0.75,-0.25) .. (0.65,-0.25);

\draw[thick](-0.65,-0.25) .. controls (-0.65,-0.25) and (-0.5,0.85) .. (-0.5,0.85);
\draw[thick](0.65,-0.25) .. controls (0.65,-0.25) and (0.5,0.85) .. (0.5,0.85);

\draw[thick](-0.65,-0.25) .. controls (-0.65,-0.25) and (-0.35,-0.5).. (-0.35,-0.5);
\draw[thick](0.65,-0.25) .. controls (0.65,-0.25) and (0.35,-0.5) .. (0.35,-0.5);

\draw[thick](-0.35,-0.5)--(-0.35,-1.65);
\draw[thick](0.35,-0.5)--(0.35,-1.65);

\draw[thick](-1.7,-1.3) .. controls (-1.5,-1.4) and (-0.65,-1.65) .. (-0.35,-1.65);
\draw[thick](1.7,-1.3) .. controls (1.5,-1.4) and (0.65,-1.65) .. (0.35,-1.65);

\draw[thick] (-0.25,1.1) .. controls (-0.1,0.9) and (0.1,0.9) .. (0.25, 1.1);
\draw[thick] (-0.35,-0.5) .. controls (-0.15,-0.25) and (0.15,-0.25) .. (0.35, -0.5);
\draw[thick] (-0.35,-1.65) .. controls (-0.15,-1.95) and (0.15,-1.95) .. (0.35,-1.65);

\node at (0,0.25) {ff};
\node at (0,2) {bff};
\node at (0,-1.1) {ff$_0$};
\node at (0.65,2.2) {abf};
\node at (-1,0.4) {obif};
\node at (1,0.4) {abif};
\node at (-1.1,-0.9) {obf$_0$};
\node at (1.1,-0.9) {abf$_0$};
\node at (0,-2.2) {tf};
\node at (-2.5,-1.5) {of$_0$};
\node at (2.5,-1.5) {af$_0$};
\node at (-2.4, 1) {of};
\node at (2.4,1) {af};
\end{tikzpicture}
\end{subfigure}
&
\begin{subfigure}[b]{0.3\textwidth}
\centering
\begin{tikzpicture}
\draw[black,thick](-0.4,1)--(-0.4,-2);
\draw[black,thick](0.4,1)--(0.4,-2);
\draw[black,thick](1.5,1)--(1.5,-2);
\draw[black,thick](-1.5,0.5)--(-1.5,-2);
\draw[black,thick](1.5,1)--(3,2);
\draw[black,thick](-1.5,0.5)--(-3,1.5);
\draw[black,thick](-0.8,1.2)--(-2.6,2.4);
\draw[thick](-1.25,0.75)--(-2.84,1.81);
\draw[thick](-1.05,0.95)--(-2.7,2.05);

\draw[thick](-0.4,1) .. controls (-0.15,1.3) and (0.15,1.3) .. (0.4,1);
\draw[thick](0.4,1) .. controls (0.7,1.3) and (1.2,1.3) .. (1.5,1);

\draw[thick](-0.4,1) .. controls (-0.55,1.25) and (-0.7,1.25) .. (-0.8,1.2);
\draw[thick](-1.25,0.75) .. controls (-1.15,0.65) and (-0.95,0.85) .. (-1.05,0.95);
\draw[thick](-1.25,0.75) .. controls (-1.3,0.7) and (-1.3,0.55) .. (-1.3,0.55);
\draw[thick](-1.3,0.55) .. controls (-1.3,0.55) and (-1.5,0.5) .. (-1.5,0.5);
\draw[thick](-1.05,0.95) .. controls (-1,1) and (-0.85,1.00) .. (-0.85,1.00);
\draw[thick](-0.85,1.00) .. controls (-0.85,1.00) and (-0.8,1.2) .. (-0.8,1.2);

\draw[thick](-1.3,0.55) .. controls (-1.1,0.55) and (-1.05,-1) .. (-1.05,-2);
\draw[thick](-0.85,1) .. controls (-0.65,1) and (-0.75,-1) .. (-0.75,-2);

\node at (-2.4,-1.5) {of};
\node at (2.4,-0.5) {af};
\node at (0,-1.3) {bff};
\node at (0.95,-1.5) {abf};
\node at (-2.5,-0.7) {obf};
\node at (-2.5,0.1) {obfx};
\node at (0.95,2.2) {tif};
\node at (-1,2.3) {tof};
\node at (-0.5,1.7) {tofx};

\draw[thick,->] (-2.2,-0.7)--(-1.3,-1.3);
\draw[thick,->] (-2.2,-0.7)--(-0.55,-1.1);
\draw[thick,->] (-2.1,0.1)--(-1,0.3);
\draw[thick,->] (-1.3,2.3)--(-2.4,2);
\draw[thick,->] (-1.3,2.3)--(-2.4,1.2);
\draw[thick,->] (-0.9,1.7)--(-1.5,1.05);

\draw[thick,->] (1.7,-1.4)--(1.7,-2);
\draw[thick,->] (2.5,1.867)--(3,2.2);
\draw[thick,->] (-2.4,2.467)--(-2.9,2.8);

\node at (1.9,-1.8) {$\tau$};
\node at (2.7,2.3) {$\mu$};
\node at (-2.55,2.9) {$\zeta$};

\end{tikzpicture}
\end{subfigure}
\end{tabular}
\caption{The space $X^{\sharp}$, near $t=0$ (left) and near $t=\infty$ (right).}
\label{fig:xsharp}
\end{figure}
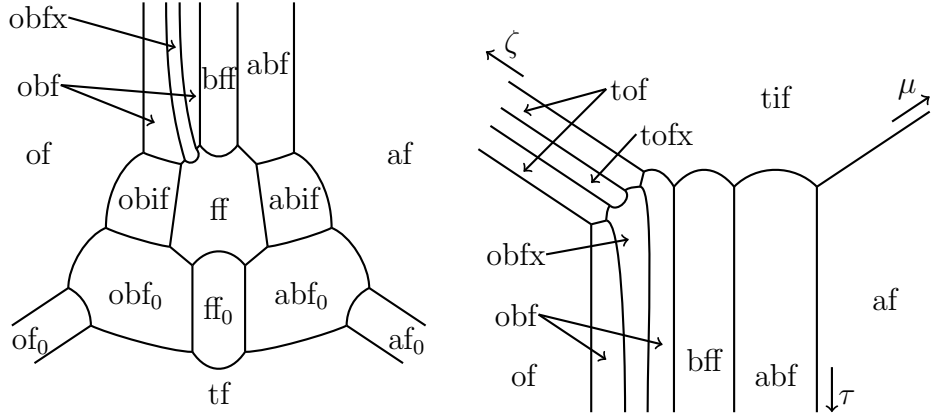

We now prove this. The first term, $\chi(z,\nu)$, is easy, as $\chi(z,\nu)$ is polyhomogeneous conormal on $U$, hence on $U^{\sharp}$, and hence on $X^{\sharp}$ by the pullback theorem and Proposition \ref{prop:bfibxsharp}. The second term was already polyhomogeneous conormal on $X_6$ and is therefore polyhomogeneous conormal on $X_6^{\sharp}$ and then on $X^{\sharp}$ by the pullback theorem. It remains to deal with the exponential term. We will again use Theorem \ref{thm:exppc}, so it suffices to show that the function
\begin{equation}\label{eq:modexp}
h(z,\nu,t)-h_{\min}(z,\nu)=2z\sinh t-2\nu t-2\sqrt{\nu^2-z^2}+2\nu\cosh^{-1}(\nu/z)
\end{equation}
satisfies the hypotheses of that theorem, with all hypersurfaces except the aforementioned five being type 2. Note that \eqref{eq:modexp} is non-negative and is zero exactly when $(z,\nu,t)\in S$.

We must first show that \eqref{eq:modexp} is polyhomogeneous on $X^{\sharp}$, at least near the support of $\chi(z,\nu)$ (that is, away from abf, af, abif, abf$_0$, and af$_0$). But we already know $h(z,\nu,t)$ is polyhomogeneous conormal on $X_6$ and therefore it is polyhomogeneous conormal on $X^{\sharp}$ as well. As for $h_{\min}(z,\nu)$, we claim it is polyhomogeneous conormal on the part of $U$ near the support of $\chi(z,\nu)$, that is, with $\nu/z>2$, away from abe and ae. Then Proposition \ref{prop:bfibx6} and the pullback theorem show that it is polyhomogeneous conormal on $X_6$ and therefore also on $X^{\sharp}$. Indeed, using the coordinate system $(\zeta,\lambda)$, which is valid near oe$\cap$obe (away from fe), we have
\[h_{\min}(z,\nu)=2\zeta^{-3}\lambda^{-3}(\sqrt{1-\lambda^6}+3\log\lambda-\log(1+\sqrt{1-\lambda^6})).\]
This is polyhomogeneous in $(\zeta,\lambda)$ with leading order $-3$ in $\zeta$ (at obe) and $\lambda^{-3}\log\lambda$ in $\lambda$ (at oe). The dominant behavior near the intersection is $6\zeta^{-3}\lambda^{-3}\log\lambda$. On the other hand, near obe$\cap$fe, we have the coordinates $(\mu/\sqrt w,\sqrt w)$, in which
\begin{multline}
h_{\min}(z,\nu)=2(\frac{\mu}{\sqrt w})^{-3}\sqrt{w}^{-3}\Big((\sqrt w^2+1)^{-3}\sqrt{(\sqrt w^{2}+1)^6-1}\\-\log((\sqrt w^2+1)^3+\sqrt{(\sqrt w^{2}+1)^6-1}\Big).\end{multline}
The expression in brackets is a smooth function of $\sqrt w$ down to $\sqrt w=0$. Moreover the leading nonzero term in its Taylor expansion is $-\sqrt 3(\sqrt w)^3$. This cancels the $(\sqrt w)^{-3}$ in front. Thus $h_{\min}(z,\nu)$ is a smooth function of $\sqrt w$ which equals $-\sqrt 3$ at $\sqrt w=0$, times $(\mu/\sqrt w)^{-3}$, and thus is polyhomogeneous conormal near obe$\cap$fe with leading order $-3$ at obe and $0$ at fe. This is enough to show that \eqref{eq:modexp} is polyhomogeneous on $X^{\sharp}$, as desired.

Now we must check the other hypotheses of Theorem \ref{thm:exppc}. First we do this away from obfx, tofx, and tof. Both $h$ and $h_{\min}$ are continuous down to ff, ff$_0$, and tf, and therefore their difference is as well, By a direct calculation, $-h_{\min}$ pulls back to a positive smooth multiple of the following function on $X^{\sharp}$ away from obfx, tof, and tofx:
\[\rho_{of}^{-3}\rho_{of_0}^{-3}(-\log\rho_{of}-\log\rho_{of_0})\rho_{obf}^{-6}\rho_{obif}^{-6}\rho_{obf_0}^{-6}.\]
Each of these faces, of, of$_0$, obf, obif, and obf$_0$, is thus a type 2 bhs for $-h_{\min}$. Moreover, at each of these faces, $-h_{\min}$ has lower leading order than $h$. Therefore each of these faces is a type 2 bhs for $h-h_{\min}$ as well. 

At tif and bff, a reversed version of this argument works. By the analysis in the previous section, each is a type 2 bhs for $h$ (away from tof in the case of tif), and $h_{\min}$ has leading order zero at each, so both tif and bff are type 2 for $h-h_{\min}$.

Finally, we analyze \eqref{eq:modexp} in a neighborhood of the new faces obfx, tofx, and tof. We begin this analysis near obfx$\cap$ff. Before the creation of obfx, we have the coordinates $(t,W,\sqrt{\kappa})$ or alternatively $(t,\hat W,\sqrt{\kappa})$. In the former set, $\nu=t^{-3}\sqrt{\kappa}^{-6}$ and $z=t^{-3}\sqrt{\kappa}^{-6}(Wt^2+1)^{-3}$. Moving to $(t,\hat W,\sqrt{\kappa})$ just means replacing $W$ with $\hat W+t^{-2}((\cosh t)^{1/3}-1)$, which replaces $Wt^2+1$ with $\hat Wt^2+(\cosh t)^{1/3}$. Using the notation $F(\hat W,t):=\hat Wt^2+(\cosh t)^{1/3}$, and after some simplification, we obtain
\begin{multline}h-h_{\min}=2t^{-3}\sqrt{\kappa}^{-6}(F(\hat W,t))^{-3}\Big(\sinh t - (F(\hat W,t))^3t\\-\sqrt{(F(\hat W,t))^6-1}+(F(\hat W,t))^3\cosh^{-1}((F(\hat W,t))^3)\Big).\end{multline}
Rewriting and shortening the notation:
\begin{equation}\label{eq:shortened}\frac 12\sqrt{\kappa}^6(h-h_{\min})=\frac{\sinh t}{t^3F^3}-\frac{1}{t^2}-\frac{\sqrt{F^6-1}}{t^3F^3}+\frac{\cosh^{-1}(F^3)}{t^3}.\end{equation}
Now we must analyze the right-hand side of \eqref{eq:shortened}. It looks to be singular as $t\to 0$, but this is an illusion: 
\begin{proposition} The right-hand side of \eqref{eq:shortened} may be written in a neighborhood of $\hat W=0$ as $\hat W^2\cdot a(t,\hat W)$, where $a(t,\hat W)$ is smooth down to $t=0$ and $a(t,0)$ is positive.
\end{proposition}
\begin{proof}
First we show that the right-hand side of \eqref{eq:shortened} is smooth down to $t=0$. Observe that $F(\hat W,t)$ is a smooth function, even in $t$. Its leading terms at $t=0$ are $1+t^2(\hat W+\frac 16)+O(t^4)$, so $F^{-1}$ is also smooth in a neighborhood of $\hat W=0$ with leading terms $1-t^2(\hat W+\frac 16)+O(t^4)$. The first two terms on the right-hand side of \eqref{eq:shortened} are thus
\[\frac{t+\frac 16t^3+O(t^5)}{t^3}(1-3t^2(\hat W+\frac 16)+O(t^4))-\frac{1}{t^2},\]
which is smooth as the $\frac{1}{t^2}$ terms cancel. The last two terms on the right-hand side of \eqref{eq:shortened} are
\[-t^{-3}(\sqrt{1-F^{-6}}-\cosh^{-1}(F^3))=-t^{-3}(\sqrt{1-F^{-6}}-\sqrt{F^3-1}\frac{\cosh^{-1}(1+(F^3-1))}{\sqrt{F^3-1}}).\]
Now $1-F^{-6}$ is smooth and even in $t$ and equals $t^2(6\hat W+1+O(t^2))$. So $\sqrt{1-F^{-6}}$ is smooth and equals $t\sqrt{6\hat W+1}+O(t^3)$. Similarly, $\sqrt{F^3-1}$ is smooth and equals $t\sqrt{3\hat W+\frac 12}+O(t^3)$. And finally, observe that $\cosh^{-1}(1+x)/\sqrt{|x|}$ has only a removable singularity at $x=0$, and when that singularity is removed it is smooth down to $x=0$ with value equal to $\sqrt 2$ at $x=0$. So
\[\frac{\cosh^{-1}(1+(F^3-1))}{\sqrt{F^3-1}}=\sqrt 2+O(F^3-1)=\sqrt 2+O(t^2).\]
Thus the last two terms are
\[-t^{-3}(t\sqrt{6\hat W+1}+O(t^3)-(t\sqrt{3\hat W+\frac 12}+O(t^3))(\sqrt 2+O(t^2)))=-t^{-3}(O(t^3)),\]
which is smooth. Thus the right-hand side of \eqref{eq:shortened} is indeed smooth down to $t=0$.

Now look at the Taylor series of the right-hand side of \eqref{eq:shortened} in $\hat W$ at $\hat W=0$. It is easy to see through plugging in $\hat W=0$ that we get zero. This is expected, as $h\geq h_{\min}$ with equality only at $\hat W=0$. Similarly, since the first derivative in $\hat W$ exists, it must be zero at $\hat W=0$. If we can show that the \emph{second} derivative in $\hat W$ is positive at $\hat W=0$, we are done.

So take that second derivative of the right-hand side of \eqref{eq:shortened}. The first derivative is
\[-\frac{3\sinh t}{tF^4}-\frac 3{tF^4\sqrt{F^6-1}}+\frac{3F^2}{t\sqrt{F^6-1}}=-\frac{3}{tF^4}(\sinh t - \sqrt{F^6-1}),\]
which as expected is zero at $\hat W=0$. The second derivative is
\[\frac{12t}{F^5}(\sinh t-\sqrt{F^6-1})+9Ft(F^6-1)^{-1/2}.\]
At $\hat W=0$, $F^3=\cosh t$, so $\sqrt{F^6-1}=\sinh t$ and thus this second derivative is just $9t(\cosh t)^{1/3}/\sinh t$. This is positive, as claimed, completing the proof of the proposition.\end{proof}
We have now shown that in the coordinates $(t,\hat W,\sqrt{\kappa})$,
\[h-h_{\min}=2(\frac{\hat W}{\sqrt{\kappa}^3})^2a(t,\hat W),\]
where $a(t,\hat W)$ is smooth and positive near $S$. Creating obfx blows up $\{\hat W=\sqrt{\kappa}=0\}$, cubically with respect to $\{\hat W=0\}$. Near the intersection of obf and obfx, $h-h_{\min}$ is a smooth positive function times a positive power of $\rho_{obf}$, and thus obfx is a type 1 bhs for $h-h_{\min}$ and obf is a type 2 bhs for $h-h_{\min}$, as we require.

Now we must analyze \eqref{eq:modexp} near tof, tofx, and their intersections with obfx. The first thing to do is to use the coordinates $(\tau,\lambda,\sqrt{\zeta})$ which are valid on $X_6^{\sharp}$ near the triple intersection of tif, of, and obf. We have
\[h-h_{\min}=\sqrt{\zeta}^{-6}(\tau^{-6}(1-\tau^{12})+12\lambda^{-3}\log\tau-2\lambda^{-3}\sqrt{1-\lambda^6}+2\lambda^{-3}\cosh^{-1}(\lambda^{-3}))\]
\[=\sqrt{\zeta}^{-6}(\tau^{-6}(1-\tau^{12})+12\lambda^{-3}\log\tau-6\lambda^{-3}\log\lambda-2\lambda^{-3}\sqrt{1-\lambda^6}+2\lambda^{-3}\log(1+\sqrt{1-\lambda^6})).\]

Let us see what happens when we blow up to create tof. Near tof$\cap$of$\cap$obf we have $(\tau,\lambda',\sqrt{\zeta})$, in which $h-h_{\min}$ is
\[\sqrt{\zeta}^{-6}\tau^{-6}(1-\tau^{12}-6(\lambda')^{-3}\log\lambda'-2(\lambda')^{-3}\sqrt{1-(\lambda')^6\tau^{12}}+2(\lambda')^{-3}\log(1+\sqrt{1-(\lambda')^6\tau^{12}})).\]
We already know this is polyhomogeneous conormal. Its leading term at tof$\cap$of$\cap$obf is $-6\sqrt{\zeta}^{-6}\tau^{-6}(\lambda')^{-3}\log\lambda'$, which, note, is positive and blows up at each of the three boundaries. Thus tof, of, and obf are all type 2 bhses for $h-h_{\min}$ near the triple junction.

On the other hand, near tof$\cap$tif$\cap$obf we use the coordinates $(\sigma,\sqrt{\lambda},\sqrt{\zeta})$, in which after again simplifying the cosh term we get
\begin{multline}\label{eq:modexpnears}h-h_{\min}=\sqrt{\zeta}^{-6}\sqrt{\lambda}^{-6}\Big(\sigma^{-6}(1-\sigma^{12}\sqrt{\lambda}^{12})\\+12\log\sigma-2\sqrt{1-\sqrt{\lambda}^{12}}+2\log(1+\sqrt{1-\sqrt{\lambda}^{12}})\Big).\end{multline}
The dominant term near the triple junction is $\sqrt{\zeta}^{-6}\sqrt{\lambda}^{-6}\sigma^{-6}$, and as before, we see that tof, tif, and obf are all type 2 bhses for $h-h_{\min}$ near the triple junction.

We must now examine what occurs near $S$ before we blow up tofx and obfx. We use $(\sigma,\sqrt{\lambda},\sqrt{\zeta})$ and shift to $(\hat\sigma,\sqrt{\lambda},\sqrt{\zeta})$. Writing $g(x)=x^{-2}(1-\sqrt{1-x^{12}})^{1/6}$ for simplicity, noting that $g(x)$ is smooth and equals $2^{-1/6}$ at $x=0$, we see that
\begin{multline}\label{eq:rhsofthis}
\sqrt{\zeta}^6\sqrt{\lambda}^6(h-h_{\min})\\ =\frac{1-(\hat\sigma+g(\sqrt{\lambda}))^{12}\sqrt{\lambda}^{12}}{(\hat\sigma+g(\sqrt{\lambda}))^{6}}+12\log(\hat\sigma+g(\sqrt{\lambda}))-2\sqrt{1-\sqrt{\lambda}^{12}}+2\log(1+\sqrt{1-\sqrt{\lambda}^{12}}).\end{multline}

\begin{proposition} The right-hand side of \eqref{eq:rhsofthis} equals $\hat\sigma^2$ times $b(\hat\sigma,\sqrt{\lambda})$, where $b$ is a smooth, positive function in a neighborhood of $\hat\sigma=0$, uniformly down to $\sqrt{\lambda}=0$.
\end{proposition}
\begin{proof}It is actually immediate from \eqref{eq:modexpnears} that the right-hand side of \eqref{eq:rhsofthis} is smooth in $(\sigma,\sqrt\lambda)$ away from $\sigma=0$, and thus is smooth in $(\hat\sigma,\sqrt\lambda)$ near $\hat\sigma=0$. As expected, it is zero when $\hat\sigma=0$ (that is, on $S$); the logarithmic terms cancel and so do the non-logarithmic terms. The first derivative of the right-hand side of \eqref{eq:rhsofthis} with respect to $\hat\sigma$ is:
\[-6(\hat\sigma+g(\sqrt{\lambda}))^{-7}-6\sqrt{\lambda}^{12}(\hat\sigma+g(\sqrt{\lambda}))^{5}+12(\hat\sigma+g(\sqrt{\lambda}))^{-1},\]
which does indeed equal zero when $\hat\sigma=0$, as expected. The second derivative at $\hat\sigma=0$ is, of course,
\[42(g(\sqrt{\lambda}))^{-8}-30\sqrt{\lambda}^{12}(g(\sqrt{\lambda}))^4-12(g(\sqrt{\lambda}))^{-2}.\]
This simplifies after some algebra to
\[(g(\sqrt{\lambda}))^{-2}(72\sqrt{1-\sqrt{\lambda}^{12}}).\]
This expression is smooth and positive, equal to $72\cdot 2^{1/3}$ at $\sqrt{\lambda}=0$, which completes the proof.
\end{proof}

With this proposition, we have shown that
\[h-h_{\min}=\sqrt{\zeta}^{-6}\sqrt{\lambda}^{-6}\hat\sigma^2 b(\hat\sigma,\sqrt{\lambda}).\]
Creating obfx is a blow-up of $\{\hat\sigma=\sqrt{\zeta}=0\}$, cubic with respect to $\{\hat\sigma=0\}$. Near the triple intersection of obfx, obf, and tof, $\sqrt\zeta/\sigma^{1/3}$ is a boundary defining function for obf, and so $h-h_{\min}$ pulls back to a positive smooth multiple of $\rho_{obf}^{-6}\rho_{tof}^{-6}$. This shows that tof and obf are type 2 bhses, and that obfx is a type 1 bhs, for $h-h_{\min}$ near this triple intersection. Finally, near obfx$\cap$tof$\cap$S, let $\sigma'=\hat\sigma/\sqrt{\zeta}^{3}$, and we can use the coordinates $(\sigma',\sqrt{\zeta},\sqrt{\lambda})$, in which $h-h_{\min}$ is a positive smooth multiple of $(\sigma')^2\sqrt{\lambda}^{-6}$. Creating tofx blows up $\sigma'=\sqrt{\lambda}=0$, cubically with respect to $\{\sigma'=0\}$. So near the triple intersection obfx$\cap$tofx$\cap$tof, $h-h_{\min}$ is a positive smooth multiple of $\rho_{tof}^{-6}$. Thus tof is a type 2 bhs there and obfx and tofx are type 1.

This, finally, is enough to show that \eqref{eq:modexp} satisfies the hypotheses of Theorem \ref{thm:exppc}, which completes the proof of Proposition \ref{prop:keyprop}. And Melrose's pushforward theorem, together with Proposition \ref{prop:bfibxsharp}, then completes the proof of Theorem \ref{thm:hardpart}.

Finally, we must analyze the portion corresponding to the remainder term $R_{\nu}^2(z)$. By Proposition \ref{prop:remonq} we know that $R_{\nu}^2(z)$ is polyhomogeneous on all of $Q$. As part of the proof of Proposition \ref{prop:keyprop}, we showed that $h_{\min}$ is polyhomogeneous conormal on $U_o$. It is easy to see that oe and obe are type 2 boundary hypersurfaces for $-h_{\min}$, and thus $e^{2h_{\min}}$ is polyhomogeneous conormal on $U_o$ by Theorem \ref{thm:exppc}. Thus $R_{\nu}^2(z)e^{2h_{\min}}$ is polyhomogeneous conormal on $U_o$. Combining this with Theorem \ref{thm:hardpart} proves Theorem \ref{thm:modsquaredrefined}.


\section{Asymptotics of the Bessel modulus and its inverse}\label{sec:modandinv}

In this section, we use our understanding of the asymptotics of $M^2_{\nu}(z)$ to analyze the functions $M_{\nu}(z)$ and $M_{\nu}^{-2}(z)$. In particular, we prove Theorem \ref{thm:mainmodulus} concerning $M_{\nu}(z)$, and we prove a proposition about $M_{\nu}^{-2}(z)$ that we will use in the next section to prove Theorem \ref{thm:mainalpha}. 

Both $M_{\nu}(z)$ and $M_{\nu}^{-2}(z)$ are powers of $M^2_{\nu}(z)$. In the appendix, we prove Theorem \ref{thm:appendixproperpowers}, which gives criteria (``proper" and ``non-logarithmic at leading order") for a polyhomogeneous conormal function to have polyhomogeneous conormal powers. We will take advantage of this by proving the following proposition:
\begin{proposition}\label{prop:idleading} The Bessel modulus-squared $M^2_{\nu}(z)$ has the following properties:

\begin{enumerate}
\item On the support of $1-\chi(z,\nu)$, $M^2_{\nu}(z)$ is proper and non-logarithmic at leading order, with leading orders $2$ at fe, 3 at abe, and 3 at ae.
\item On the support of $\chi(z,\nu)$, $M_{\nu}^2(z)(T(z,\nu))^{-1}$ is proper and non-logarithmic at leading order, with leading orders 2 at fe, 3 at obe, and 3 at oe.
\end{enumerate}
\end{proposition}
\begin{proof} We do this by analyzing the leading order terms of the expansion of $M^2_{\nu}(z)$ at each boundary hypersurface. To do this we use the previously known asymptotics of Bessel functions.

At fe, away from obe and abe, we use the coordinates
\[(w':=\frac{w}{\mu^2},\mu).\]
As discussed in Section 2, $w'$ is $\frac{\nu-z}{\nu^{1/3}}$ times a function which is smooth and equals 1 on fe. We also know from \cite[10.19.8]{dlmf} that for any fixed $a$,
\[M_{\nu}^2(\nu+a\nu^{1/3})\sim 2^{2/3}\nu^{-2/3}(Ai^2(-2^{1/3}a)+Bi^2(-2^{1/3}a)),\]
where $Ai$ and $Bi$ are the Airy functions.
Thus in the coordinates $(w',\mu)$, the leading order term of the polyhomogeneous expansion of $M_{\nu}^2(z)$ at fe is
\[2^{2/3}\mu^2(Ai^2+Bi^2)(2^{1/3}w').\]

Now switch to the coordinates $(\frac{\mu}{\sqrt{-w}},\sqrt{-w})$ which are valid near abe$\cap$fe. This leading order term is
\[2^{2/3}(\frac{\mu}{\sqrt{-w}})^{2}(\sqrt{-w})^2(Ai^2+Bi^2)(-\frac{2^{1/3}}{(\mu/\sqrt{-w})^2}).\]
The expression $Ai^2+Bi^2$ is the Airy modulus-squared, and from \cite[9.8.20]{dlmfch9} we see that as $\mu/\sqrt{-w}\to 0$, the leading order at abe of the leading order term at fe is
\[2^{2/3}(\frac{\mu}{\sqrt{-w}})^2(\sqrt{-w})^2\cdot\frac{1}{\pi}(\frac{2^{1/3}}{(\mu/\sqrt{-w})^2})^{-1/2}=\frac{\sqrt 2}{\pi}(\frac{\mu}{\sqrt{-w}})^3(\sqrt{-w})^2.\]
At abe, we can see from Debye's expansions \cite[10.19.6]{dlmf} that the leading order of $M^2_{\nu}(z)$ is at order $\nu^{-1}=\mu^{-3}$, and therefore the leading order there is indeed 3. Thus $M^2_{\nu}(z)$ is $\frac{\sqrt 2}{\pi}\rho_{fe}^2\rho_{abe}^3$ and therefore is proper, and non-logarithmic at leading order, near fe$\cap$abe.

We must also examine the corner abe$\cap$ae and can do this by looking at either Debye's expansions or the classical large-argument expansions and seeing what happens to their coefficients. The classical large-argument expansion for the modulus-squared $M^2_{\nu}(z)$ is given by \cite[10.18.17]{dlmf} and its leading term is $\frac{2}{\pi z}$, which is $\frac{2}{\pi}\zeta^{3}$. Coordinates near abe$\cap$ae are $(\mu,\frac{\zeta}{\mu})$, so the leading term at abe of the leading term at ae is
\[\frac{2}{\pi}(\frac{\zeta}{\mu})^3\mu^3,\]
which is $\frac{2}{\pi}\rho_{abe}^3\rho_{ae}^3$. Thus, again, $M^2_{\nu}(z)$ is proper and non-logarithmic at leading order here, proving part (1) of Proposition \ref{prop:idleading}.

For part (2), we must analyze the leading orders in the asymptotic expansions of the polyhomogeneous function
\begin{equation}\label{eq:fractionmod}
M^2_{\nu}(z)(T(z,\nu))^{-1}=M^2_{\nu}(z)\cdot e^{2\sqrt{\nu^2-z^2}-2\nu\cosh^{-1}(\nu/z)}
\end{equation}
at fe, obe, and oe. By Debye's expansions \cite[10.19.3]{dlmf} we know that the leading term in the expansion of $M^2_{\nu}(\nu\sech\alpha)$, for $\alpha>0$ fixed, is
\[\frac{2e^{2\nu(\alpha-\tanh\alpha)}}{\pi\nu\tanh\alpha}.\]
The relationship between the coordinates is that $\alpha=\cosh^{-1}(\nu/z)$. So $\tanh\alpha=\sqrt{1-(z/\nu)^2}$, and the exponent
\[2\nu(\alpha-\tanh\alpha)=2\nu\cosh^{-1}(\nu/z)-2\sqrt{\nu^2-z^2}=-h_{\min}(z,\nu).\]
Thus the exponential in Debye's expansions precisely matches our exponential $e^{-h_{\min}(z,\nu)}$ and as a result, the leading order term in the asymptotic expansion of \eqref{eq:fractionmod} at obe is
\[\frac{2}{\pi\nu\tanh\alpha}=\frac{2}{\pi\sqrt{\nu^2-z^2}}.\]
This allows us to interpret Debye's expansions themselves as polyhomogeneous expansions at obe (after dividing out the exponential term). 

Near the corner obe$\cap$oe we use the coordinates $(\zeta,\mu/\zeta)$, in which the leading term at obe is
\[\frac{2}{\pi}\zeta^{3}(\mu/\zeta)^3\frac{1}{\sqrt{1-(\mu/\zeta)^6}}.\]
Thus near that corner, $M^2_{\nu}(z)=\frac{2}{\pi}\rho_{oe}^3\rho_{obe}^3$ and is thus proper and non-logarithmic there. On the other hand, near obe$\cap$fe we use the coordinates $(\frac{\mu}{\sqrt w},\sqrt{w})$, in which the leading term at obe is
\[\frac{2}{\pi}(\frac{\mu}{\sqrt w})^3(\sqrt w)^3\frac{1}{\sqrt{1-(\sqrt{w}^2+1)^{-6}}}.\]
Performing a Puiseux series expansion of that last term shows that it equals $\sqrt{w}^{-1}$ times a smooth function in $\sqrt{w}$ which is $\frac{1}{\sqrt 6}$ at $\sqrt{w}=0$. Thus the leading term at obe, in these coordinates, has leading term at fe given by
\[\frac{2}{\pi\sqrt{6}}(\frac{\mu}{\sqrt w})^3(\sqrt w)^2.\]
This is again proper and non-logarithmic, completing the proof of Proposition \ref{prop:idleading}.
\end{proof}
Theorem \ref{thm:appendixproperpowers} now gives some corollaries:
\begin{corollary} The Bessel modulus $M_{\nu}(z)$ has the following properties:
\begin{enumerate}
\item On the support of $1-\chi(z,\nu)$, $M_{\nu}(z)$ is polyhomogeneous conormal with leading order 1 at fe, $3/2$ at abe, and $3/2$ at ae.
\item On the support of $\chi(z,nu)$, $M_{\nu}(z)$ is $\sqrt{T(z,\nu)}$ times a polyhomogeneous conormal function with leading orders 1 at fe, $3/2$ at obe, and $3/2$ at oe.
\end{enumerate}
\end{corollary}
\noindent Note that this proves Theorem \ref{thm:mainmodulus}.

The following corollary takes only slightly more work:
\begin{corollary} The reciprocal $M_{\nu}^{-2}(z)$ has the following properties:
\begin{enumerate}
\item On the support of $1-\chi(z,\nu)$, $M_{\nu}^{-2}(z)$ is polyhomogeneous conormal with leading order $-2$ at fe, $-3$ at abe, and $-3$ at ae.
\item On the support of $\chi(z,\nu)$, $M_{\nu}^{-2}(z)$ is $T(z,\nu)^{-1}$ times a polyhomogeneous conormal function with leading orders $-2$ at fe, $-3$ at obe, and $-3$ at oe.
\item On all of $U$, the function
\[\frac{1}{zM^2_{\nu}(z)}\]
is polyhomogeneous conormal, with leading order 0 at ae, leading order 0 at abe, leading order 1 at fe, and empty index sets at obe and oe.
\end{enumerate}
\end{corollary}
\begin{proof} Parts (1) and (2) are immediate from Theorem \ref{thm:appendixproperpowers}. Since $z^{-1}$ is polyhomogeneous conormal on $U$ with leading orders 3 at ae, abe, fe, and obe, part (3) follows immediately from the claim that $T(z,\nu)^{-1}$ is polyhomogeneous conormal on the support of $\chi(z,\nu)$, with empty index sets at abe and ae.

To prove that claim we use Theorem \ref{thm:exppc}. We  have $T(z,\nu)^{-1}=\exp[h_{\min}(z,\nu)]$. We showed in the analysis of the remainder term $R_{\nu}(z)$ in the prior section that $h_{\min}(z,\nu)$ 
is polyhomogeneous conormal on the support of $\chi(z,\nu)$. Moreover, its leading order behavior near obe$\cap$oe is $6\rho_{obe}^{-3}\rho_{oe}^{-3}\log\rho_{oe}$ (note the logarithm is negative) and its leading order behavior near obe$\cap$fe is $-\sqrt{3}\rho_{obe}^{-3}$. Thus obe and oe are type 2 bhses for $h_{\min}(z,\nu)$ and fe is a type 1 bhs. By Theorem \ref{thm:exppc}, the exponential $T(z,\nu)^{-1}=\exp[h_{\min}(z,\nu)]$ has the desired properties. \end{proof}

\begin{remark} Part (3) tells us that $z^{-1}M_{\nu}^{-2}(z)$ decays to infinite order at obe and oe -- that is, faster than any polynomial. This infinite-order decay is actually exponential decay and can be characterized more precisely by using part (2). In theory, this could be used to sharpen our estimates on $\theta_{\nu}(z)$ in the next section, but the extra work needed to do so will be deferred to a further paper.
\end{remark}


\section{Asymptotics of the Bessel phase}

The proof of Theorem \ref{thm:mainalpha} is obtained by analysis of \eqref{eq:phaseformula}. We integrate $u^{-1}M_{\nu}(u)^{-2}$ from $0$ to $z$. We break this integral up at $u=1$. For larger $u$ we understand $u^{-1}M_{\nu}(u)^{-2}$ and can use more geometric microlocal analysis. For $u<1$, though, we do not know about polyhomogeneity of $M_{\nu}(u)^{-2}$, so we must analyze the small $u$ portion of the integral, namely
\begin{equation}\label{eq:alphasmall}
\int_0^1\frac{1}{uM_{\nu}(u)^2}\, du,
\end{equation}
through other means. As we will show, this contribution is negligible.

\begin{theorem}\label{thm:smallalpha} The expression \eqref{eq:alphasmall} is polyhomogeneous conormal and decaying to infinite order in $\nu^{-1}$ as $\nu\to\infty$. \end{theorem}

To prove this theorem we need the following estimates on $M_{\nu}^2(z)$:
\begin{lemma}\label{lem:expdecayest} For $\nu\geq 2$ and $z\leq 1$, the following hold:
\begin{enumerate}
\item For each $k$, there is a constant $C_k$ depending only on $k$ so that
\[\frac{\partial^k}{\partial\nu^k}(M_{\nu}^2(z))\leq C_k(\nu+|\log z|)^kM_{\nu}^2(z);\]
\item There is a positive constant $c$ such that $M_{\nu}^2(z)\geq cz^{-1}e^{\nu}$.
\end{enumerate}
\end{lemma}
We defer the proof of these estimates for the moment and use them to prove Theorem \ref{thm:smallalpha}. It suffices to show that the expression \eqref{eq:alphasmall} and all of its derivatives in $\nu$ decay to infinite order as $\nu\to\infty$. Differentiating inside the integral sign and using Fa\`a di Bruno's formula \cite{fdb},
\[\frac{\partial^k}{\partial\nu^k}\int_0^1\frac{1}{uM_{\nu}(u)^2}\, du =\int_0^1\frac 1u( \sum_{\ell=1}^k \frac{(-1)^\ell \ell!}{(M_{\nu}(u)^2)^{(\ell+1)}}B_{k,\ell}(\partial_{\nu}M_{\nu}^2(u),\ldots,\partial_{\nu}^{k-\ell+1}M_{\nu}^2(u))\, du.\]
Here $B_{k,\ell}$ is a finite sum of monomials in its arguments, each of which has $\ell$ terms and involves a total of $k$ $\nu$-derivatives of $M_{\nu}^2(u)$. By Lemma \ref{lem:expdecayest}, part (1), we have the bound
\begin{multline}
|\frac{\partial^k}{\partial\nu^k}\int_0^1\frac{1}{uM_{\nu}(u)^2}\, du|\leq\int_0^1\frac 1u(\sum_{\ell=1}^k\frac{C}{(M_{\nu}(u)^2)^{(\ell+1)}})(\nu+|\log u|)^k(M_{\nu}^2(u))^{\ell})\, du\\
=\int_0^1\frac {Ck(\nu+|\log u|)^k}{uM_{\nu}^2(u)}\, du.
\end{multline}
By Lemma \ref{lem:expdecayest}, part (2), we obtain
\[|\frac{\partial^k}{\partial\nu^k}\int_0^1\frac{1}{uM_{\nu}(u)^2}\, du|\leq C\int_0^1e^{-\nu}(\nu+|\log u|)^k\, du\leq C2^k\int_0^1e^{-\nu}(\nu^k+|\log u|^k)\, du, \]
which is immediately seen to decay to infinite order in $\nu$. This completes the proof of Theorem \ref{thm:smallalpha}. 

\begin{proof}[Proof of Lemma \ref{lem:expdecayest}]
Let us begin with (1). We use Nicholson's formula, differentiate under the integral sign, and use $\sinh(2\nu t)\leq\cosh (2\nu t)$ to obtain
\[\frac{\frac{\partial^k}{\partial\nu^k}(M_{\nu}^2(z))}{M_{\nu}^2(z)}\leq 2^k\frac{\int_0^{\infty}K_0(2z\sinh t)t^k\cosh(2\nu t)\, dt}{\int_0^{\infty}K_0(2z\sinh t)\cosh(2\nu t)\, dt}.\]
Using the fact that $e^{-x}\cosh(x)\in [1/2,1]$ for $x\geq 0$,
\[\frac{\frac{\partial^k}{\partial\nu^k}(M_{\nu}^2(z))}{M_{\nu}^2(z)}\leq 4\cdot 2^k\frac{\int_0^{\infty}K_0(2z\sinh t)t^ke^{2\nu t}\, dt}{\int_0^{\infty}K_0(2z\sinh t)e^{2\nu t}\, dt}.\]
Make the change of variables $w=ze^t$ in both numerator and denominator, then canceling a $z^{2\nu}$, gives
\[\frac{\frac{\partial^k}{\partial\nu^k}(M_{\nu}^2(z))}{M_{\nu}^2(z)}\leq 4\cdot 2^k\frac{\int_z^{\infty}K_0(w+\frac{z^2}{w})(\log w+|\log z|)^k w^{2\nu-1}\, dw}{\int_z^{\infty}K_0(w+\frac{z^2}{w})w^{2\nu-1}\, dw}.\]
Using the estimate $(a+b)^k\leq 2^k(a^k+b^k)$ and the fact that $0\leq z\leq 1$, we have
\[\frac{\frac{\partial^k}{\partial\nu^k}(M_{\nu}^2(z))}{M_{\nu}^2(z)}\leq 2^{2k+2}|\log z|^{k} + 2^{2k+2}\frac{\int_0^{\infty}K_0(w+\frac{z^2}{w})(\log w)^k w^{2\nu-1}\, dw}{\int_1^{\infty}K_0(w+\frac{z^2}{w})w^{2\nu-1}\, dw}.\]
By the fact that $K_0$ is decreasing and $z\leq 1$,
\[\frac{\frac{\partial^k}{\partial\nu^k}(M_{\nu}^2(z))}{M_{\nu}^2(z)}\leq 2^{2k+2}|\log z|^{k} + 2^{2k+2}\frac{\int_0^{\infty}K_0(w)(\log w)^k w^{2\nu-1}\, dw}{\int_1^{\infty}K_0(w+\frac{1}{w})w^{2\nu-1}\, dw}.\]
It is a fact that
\[K_0(x)\leq\frac{1}{\sqrt x}e^{-x}\sqrt{\pi/2};\quad K_0(x)\geq \frac 12\frac{1}{\sqrt x}e^{-x}\sqrt{\pi/2}\textrm{ for }x\in [1,\infty).\]
Using this fact,
\[\frac{\frac{\partial^k}{\partial\nu^k}(M_{\nu}^2(z))}{M_{\nu}^2(z)}\leq 2^{2k+2}|\log z|^{k} + 2^{2k+3}\frac{\int_0^{\infty}e^{-w}(\log w)^k w^{2\nu-3/2}\, dw}{\int_1^{\infty}e^{-w}w^{2\nu-3/2}(e^{-1/w}\frac{1}{\sqrt{1+w^{-2}}})\, dw}.\]
Estimating the denominator from below when $w\geq 1$ leads to
\[\frac{\frac{\partial^k}{\partial\nu^k}(M_{\nu}^2(z))}{M_{\nu}^2(z)}\leq 2^{2k+2}|\log z|^{k} + 2^{2k+7/2}e\frac{\int_0^{\infty}e^{-w}(\log w)^k w^{2\nu-3/2}\, dw}{\int_1^{\infty}e^{-w}w^{2\nu-3/2}\, dw}.\]
Yet when $\nu\geq 2$ we have
\[\int_0^1e^{-w}w^{2\nu-3/2}\, dw\leq \int_1^{2}e^{-w}w^{2\nu-3/2}\, dw\leq \int_1^{\infty}e^{-w}w^{2\nu-3/2}\, dw\]
and so we may change the lower limit in the denominator to 0 at the cost of a factor of 2:
\[\frac{\frac{\partial^k}{\partial\nu^k}(M_{\nu}^2(z))}{M_{\nu}^2(z)}\leq 2^{2k+2}|\log z|^{k} + 2^{2k+9/2}e\frac{\int_0^{\infty}e^{-w}(\log w)^k w^{2\nu-3/2}\, dw}{\int_0^{\infty}e^{-w}w^{2\nu-3/2}\, dw}.\]
This may all be written in terms of derivatives of the gamma function:
\[\frac{\frac{\partial^k}{\partial\nu^k}(M_{\nu}^2(z))}{M_{\nu}^2(z)}\leq 2^{2k+2}|\log z|^{k} + 2^{2k+9/2}e\frac{\Gamma^{(k)}(2\nu-3/2)}{\Gamma(2\nu-3/2)}.\]
We see that part (1) of Lemma \ref{lem:expdecayest} follows immediately from the claim that for $x\geq 2$,
\begin{equation}\label{eq:gammaclaim}\frac{\Gamma^{(k)}(x)}{\Gamma(x)}\leq C_kx^k.\end{equation} 
We prove \eqref{eq:gammaclaim} by induction. It is true for $k=1$ with $C_1=1$, as the convexity of $\Gamma$ and the recursion relation imply that
\[\Gamma'(x)\leq\Gamma(x+1)-\Gamma(x)=(x-1)\Gamma(x)\leq x\Gamma(x).\]
For the inductive step, assume it is true for $i\in\{1,\ldots,k\}$. Then use Fa\`a di Bruno's formula, separating out the $\ell=1$ term and rearranging, to write
\[\frac{\Gamma^{(k+1)}(x)}{\Gamma(x)}=(\log\Gamma(x))^{(k+1)} - \sum_{\ell=2}^{k+1}C_{\ell}\frac{1}{(\Gamma(x))^{\ell}}B_{k+1,\ell}(\Gamma'(x),\ldots,\Gamma^{(k+1)-\ell+1}(x)).\]
The first term on the right is the polygamma function, bounded in absolute value for $x>2$ by $C\log x$ when $k=1$ and just by $C$ when $k>1$, so in all cases bounded in absolute value by $Cx^{k+1}$. On the other hand, each term in $B_{k+1,\ell}$ has $\ell$ factors, each $\Gamma^{(i)}(x)$ for some $i\in[1,k]$, with $k+1$ total derivatives of $\Gamma$. By the inductive hypothesis,
\[|B_{k+1,\ell}(\Gamma'(x),\ldots,\Gamma^{(k+1)-\ell+1}(x))|\leq Cx^{k+1}\Gamma(x)^{\ell}.\]
This proves the inductive step and hence \eqref{eq:gammaclaim}, and with it part (1) of Lemma \ref{lem:expdecayest}.

For part (2) of the Lemma, we need to find a positive lower bound, when $\nu\geq 2$ and $z\leq 1$, for
\begin{equation}\label{eq:needlower}
\int_0^{\infty}zK_0(2z\sinh t)e^{-\nu}\cosh(2\nu t)\, dt.
\end{equation}
This is straightforward. By monotonicity, \eqref{eq:needlower} is greater than or equal to
\[\int_0^{\infty}zK_0(ze^t)e^{-\nu}e^{2\nu t}\, dt.\]
Using the same substitution $w=ze^t$, and the fact that $2\nu-1\geq 3$, shows that \eqref{eq:needlower} is bounded below by
\[\int_z^{\infty}K_0(w)e^{-\nu}(w/z)^{2\nu-1}\, dw\geq e^{-\nu}\int_1^{\infty}K_0(w)w^{3}\, dw. \]
This is certainly greater than some positive constant, completing the proof of part (2) and thus of Lemma \ref{lem:expdecayest}.\end{proof}

Now we analyze the large-$u$ portion of the integral \eqref{eq:phaseformula}, namely:
\begin{equation}\label{eq:alphamain}
\int_1^z\frac{1}{uM_{\nu}^2(u)}\, du.
\end{equation}
This must be done via geometric microlocal analysis rather than explicit estimates. We know that the integrand is polyhomogeneous conormal on $U$ with the variables $(u,\nu)$. Let $\tilde u=u^{-1/3}$; the integral is then
\begin{equation}\label{eq:alphamainrewrite}
\int_{\zeta}^{1}\frac{1}{\tilde uM_{\nu}^2(\tilde u^{-3})}\, d\tilde u.
\end{equation}
To analyze the polyhomogeneity of the integral we perform a construction similar to the construction of the ``triple scattering space'' in \cite{mel94}. We construct a three-dimensional manifold with corners, $\hat U$, as follows.
\begin{enumerate}
\item Begin with $[0,1)_{\zeta}\times [0,1)_{\mu}\times [0,1)_{\tilde u}$.
\item Blow up first the triple intersection $\{\zeta=\mu=\tilde u=0\}$ and then the double intersections $\{\zeta=\mu=0\}$, $\{\zeta=\tilde u=0\}$, and $\{\tilde u=\mu=0\}$. Label the boundary faces of the resulting space $\mathcal F_{ijk}$, $i,j,k\in\{0,1\}$, according to which variables vanish there; for example, the front face of the blow-up of the triple intersection is $\mathcal F_{111}$ and the front face of the blow-up of $\{\zeta=\mu=0\}$ is $\mathcal F_{110}$. Also label the lifts of the diagonals: $\{\zeta=\mu\}$ is $D_{110}$, et cetera. Call this space $\hat U_1$.
\item Blow up $D_{111}\cap\mathcal F_{111}$, quadratically with respect to $D_{111}$.
\item Blow up $D_{110}\cap\mathcal F_{111}$, $D_{101}\cap\mathcal F_{111}$, and $D_{011}\cap\mathcal F_{111}$, all quadratically with respect to $D_{ijk}$.
\item Blow up $D_{110}\cap\mathcal F_{110}$, $D_{101}\cap\mathcal F_{101}$, and $D_{011}\cap\mathcal F_{011}$, all quadratically with respect to $D_{ijk}$. This gives $\hat U$.
\end{enumerate}
We also let $\hat U^{\flat}$ be the subset of $\hat U$ with $\tilde u\geq\zeta$. It is itself a manifold with corners, as the lift of $D_{101}$ is a p-submanifold of $\hat U$. See Figure \ref{fig:hatu}.

\begin{figure}
\centering
\begin{tikzpicture}
\draw[black,thick](-0.75,1.25)--(-0.75,3);
\draw[black,thick](0.75,1.25)--(0.75,3);
\draw[thick](-1.458,0.0245)--(-2.973,-0.85);
\draw[thick](-0.708,-1.275)--(-2.223,-2.15);
\draw[thick](1.458,0.0245)--(2.973,-0.85);
\draw[thick](0.708,-1.275)--(2.223,-2.15);


\draw[thick](0.6,0) arc (0:10:0.6);
\draw[thick](0.6,0) arc (0:-10:0.6);
\draw[thick](-0.6,0) arc (180:190:0.6);
\draw[thick](-0.6,0) arc (180:170:0.6);
\draw[thick](0.3,0.52) arc (60:50:0.6);
\draw[thick](0.3,0.52) arc (60:70:0.6);
\draw[thick](-0.3,0.52) arc (120:130:0.6);
\draw[thick](-0.3,0.52) arc (120:110:0.6);
\draw[thick](0.3,-0.52) arc (-60:-50:0.6);
\draw[thick](0.3,-0.52) arc (-60:-70:0.6);
\draw[thick](-0.3,-0.52) arc (-120:-130:0.6);
\draw[thick](-0.3,-0.52) arc (-120:-110:0.6);

\draw[thick](0.205,0.565)--(0.205,1.065);
\draw[thick](-0.205,0.565)--(-0.205,1.065);
\draw[thick](-0.205,0.565) .. controls (-0.1,0.45) and (0.1,0.45) .. (0.205,0.565);
\draw[thick](0.205,-0.565)--(0.205,-1.365);
\draw[thick](-0.205,-0.565)--(-0.205,-1.365);
\draw[thick](-0.205,-0.565) .. controls (-0.1,-0.45) and (0.1,-0.45) .. (0.205,-0.565);

\draw[thick](0.385,0.46)--(1.078,0.86);
\draw[thick](0.59,0.105)--(1.283,0.505);
\draw[thick](0.385,0.46) .. controls (0.365,0.36) and (0.49,0.155) .. (0.59,0.105);
\draw[thick](-0.385,0.46)--(-1.078,0.86);
\draw[thick](-0.59,0.105)--(-1.283,0.505);
\draw[thick](-0.385,0.46) .. controls (-0.365,0.36) and (-0.49,0.155) .. (-0.59,0.105);

\draw[thick](0.385,-0.46)--(0.818,-0.71);
\draw[thick](0.59,-0.105)--(1.023,-0.355);
\draw[thick](0.385,-0.46) .. controls (0.365,-0.36) and (0.49,-0.155) .. (0.59,-0.105);
\draw[thick](-0.385,-0.46)--(-0.818,-0.71);
\draw[thick](-0.59,-0.105)--(-1.023,-0.355);
\draw[thick](-0.385,-0.46) .. controls (-0.365,-0.36) and (-0.49,-0.155) .. (-0.59,-0.105);

\draw[thick] (0.75,1.25) .. controls (0.83,1.23) and (1.078,0.93) .. (1.078,0.86);
\draw[thick] (1.078,0.86) .. controls (1.358,0.92) and (1.443,0.655) .. (1.283, 0.505);
\draw[thick] (1.283,0.505) .. controls (1.333,0.455) and (1.458,0.1) .. (1.458,0.0245);
\draw[thick] (-0.75,1.25) .. controls (-0.83,1.23) and (-1.078,0.93) .. (-1.078,0.86);
\draw[thick] (-1.078,0.86) .. controls (-1.358,0.92) and (-1.443,0.655) .. (-1.283, 0.505);
\draw[thick] (-1.283,0.505) .. controls (-1.333,0.455) and (-1.458,0.1) .. (-1.458,0.0245);
\draw[thick] (-0.708,-1.275) .. controls (-0.608,-1.33) and (-0.305,-1.385) .. (-0.205,-1.365);
\draw[thick] (0.708,-1.275) .. controls (0.608,-1.33) and (0.305,-1.385) .. (0.205,-1.365);
\draw[thick] (0.205,-1.365) .. controls (0.1,-1.6) and (-0.1,-1.6) .. (-0.205,-1.365);

\draw[thick] (0.1,1.15)--(0.1,2.9);
\draw[thick] (-0.1,1.15)--(-0.1,2.9);
\draw[thick] (0.205,1.065) .. controls (0.17,1.125) .. (0.1,1.15);
\draw[thick] (-0.205,1.065) .. controls (-0.17,1.125) .. (-0.1,1.15);
\draw[thick] (0.205,1.065) .. controls (0.405,1.05) and (0.6,1.1) .. (0.75,1.25);
\draw[thick] (-0.205,1.065) .. controls (-0.405,1.05) and (-0.6,1.1) .. (-0.75,1.25);
\draw[thick] (0.1,1.15) .. controls (0.08,1.05) and (-0.08,1.05) .. (-0.1,1.15);

\draw[thick] (-1.046,-0.488)--(-2.561,-1.363);
\draw[thick] (-0.946,-0.662)--(-2.461,-1.537);
\draw[thick] (-1.025,-0.355) .. controls (-1.059,-0.415) .. (-1.046,-0.488);
\draw[thick] (-0.82,-0.71) .. controls (-0.889,-0.71) .. (-0.946,-0.662);
\draw[thick] (-1.025,-0.355) .. controls (-1.112,-0.174) and (-1.253,-0.03) .. (-1.458,0.0245);
\draw[thick] (-0.82,-0.71) .. controls (-0.707,-0.876) and (-0.653,-1.07) .. (-0.708,-1.275);
\draw[thick] (-1.046,-0.488) .. controls (-0.949,-0.456) and (-0.869,-0.594) .. (-0.946,-0.662);

\draw[thick] (1.046,-0.488)--(2.561,-1.363);
\draw[thick] (0.946,-0.662)--(2.461,-1.537);
\draw[thick] (1.025,-0.355) .. controls (1.059,-0.415) .. (1.046,-0.488);
\draw[thick] (0.82,-0.71) .. controls (0.889,-0.71) .. (0.946,-0.662);
\draw[thick] (1.025,-0.355) .. controls (1.112,-0.174) and (1.253,-0.03) .. (1.458,0.0245);
\draw[thick] (0.82,-0.71) .. controls (0.707,-0.876) and (0.653,-1.07) .. (0.708,-1.275);
\draw[thick] (1.046,-0.488) .. controls (0.949,-0.456) and (0.869,-0.594) .. (0.946,-0.662);

\draw[thick,dotted] (-0.476,0.275)--(-2.858,1.65);
\draw[thick,dotted] (0.476,-0.275)--(2.858,-1.65);
\draw[thick,dotted] (-0.476,0.275) .. controls (-0.2,0.3) and (0.3,0.175) .. (0.476,-0.275);

\draw[thick,<-] (1,3)--(1,2.5); 
\draw[thick,<-] (-1.75,-2.15)--(-1.32,-1.9);
\draw[thick,<-] (1.75,-2.15)--(1.32,-1.9);
\node at (1.25,2.75) {$\hat u$};
\node at (-1.1,-1.9) {$\zeta$};
\node at (1.1,-1.9) {$\mu$};

\end{tikzpicture}
\caption{The space $\hat U$. The subset $\hat U^{\flat}$ is above the dotted line.}
\label{fig:hatu}
\end{figure}
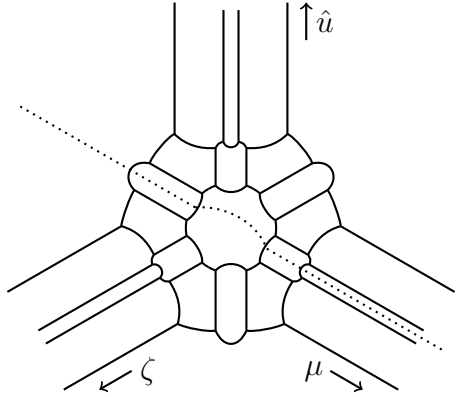

The reasoning is now straightforward. Since the map from $\hat U\to U$ given by projection off the $\zeta$-variable is a b-fibration by Proposition \ref{prop:triplescat}, the pullback theorem tells us that the function $(\tilde uM_{\nu}^2(\tilde u^{-3}))^{-1}$ is polyhomogeneous conormal as a function on $\hat U$. By restriction, it is polyhomogeneous conormal as a function on $\hat U^{\flat}$. Since the map from $\hat U^{\flat}$ to $U$ given by projection off the $\tilde u$-variable is a b-fibration by the same Proposition \ref{prop:triplescat}, the pushforward theorem tells us that the integral \eqref{eq:alphamainrewrite} is polyhomogeneous conormal on $U$, as desired. This proves Theorem \ref{thm:mainalpha}.


\appendix

\section{Background on geometric microlocal analysis}

The material in this section is originally due to Richard Melrose \cite{melrosebook}. The reference \cite{gribc} gives a very clear introduction; see also \cite{maz91}. This section is modeled closely on \cite[Section 2]{gtv}, giving a brief summary of the results we need for this particular work.

\subsection{Manifolds with corners and blow-up} A \emph{manifold with corners} is a space $X$ which is locally diffeomorphic near each point $p\in X$ to $\mathbb R_{+}^k\times\mathbb R^{N-k}$ for some $k\in\mathbb N_0$. We say that $k$ is the codimension of the point $p$. Manifolds with corners are examples of stratified spaces, stratified according to the codimension of each point. The faces of codimension 1 of $X$ are called \emph{boundary hypersurfaces}, and the faces of codimension $\geq 2$ are called \emph{corners}. Throughout we make the technical restriction that each boundary hypersurface $H\subseteq X$ is embedded, that is, that the inclusion map $\iota:H\to X$ is an embedding. This guarantees the existence, for each boundary hypersurface $H$, of a \emph{boundary defining function} $\rho_H:X\to\mathbb R^+$ for which $\rho_H=0$ on $X$ and $d\rho_H\neq 0$ on $X$. A subset $P\subseteq X$ is called a \emph{p-submanifold} of $X$ if it may be locally expressed as a coordinate submanifold of $\mathbb R_{+}^k\times\mathbb R^{N-k}$.

Any $p$-submanifold $P$ of a manifold with corners $X$ may be \emph{blown up} to produce a new manifold with corners $[X;P]$. This blow-up is an invariant way of introducing polar coordinates around $P$. Specifically, the blown-up space $[X;P]$ is constructed by taking $X\setminus P$ and gluing it to the inward pointing spherical normal bundle of $P\subseteq X$. This inward pointing spherical normal bundle is the \emph{front face} of the blow-up. Distinct points on the front face correspond to the endpoints of different rays approaching $P\subseteq X$. There is a ``blow-down map'' $\beta:[X;P]\to X$ which is the identity on $X\setminus P$ and maps every point in the front face to the corresponding point in $P$.  We borrow the terminology of \cite{gtv} verbatim to discuss lifts of $p$-submanifolds: if $Z\subseteq X$ is a p-submanifold which is not contained in $P$, then its \emph{lift} to $[X;P]$ is defined to be the closure of $\beta^{-1}(Z\setminus P)$.

A natural way to think about blown-up spaces is to use \emph{projective coordinates}. We illustrate this in the case $X=[0,1]\times[0,1]$, $P=\{(0,0)\}$. Near the intersection of the front face and the lift of $\{x=0\}$, we use the coordinates $(y,\frac xy)$, so that $y$ is a boundary defining function for the front face and $\frac xy$ is a boundary defining function for the lift of $\{x=0\}$; similarly, near the intersection of the front face and the lift of $\{y=0\}$, we use $(x,\frac yx)$. In the interior of the front face, either system may be used. See \cite[Section 2.3.4]{griscales} for a full explanation of projective coordinates for arbitrary $P$ and $X$.

We also need to consider a variant on blow-up known as \emph{quasihomogeneous blow-up}. This was originally introduced in \cite{melrosebook}; see \cite{griscales,grhu} for an introduction. In short, instead of introducing a new point for each distinct \emph{ray} approaching $P$, we introduce a new point for each distinct parabola (or cubic, or quartic, et cetera) approaching $p$. We again illustrate this procedure with $X=[0,1]\times[0,1]$ and $P=\{(0,0)\}$. We define the quasihomogeneous blow-up $[X;P]_{a,\{y=0\}}$ of order $a\in\mathbb N$ with respect to the face $\{y=0\}$ to be the space obtained by replacing the endpoint of each curve $y=cx^a$ with a separate point. Projective coordinates are explained in \cite[Section 2.3.4]{griscales}. Near the intersection of the front face and the lift of $\{x=0\}$ we use $(y^{1/a},\frac{x}{y^{1/a}})$, and near the intersection of the front face and the lift of $\{y=0\}$ we use $(x,\frac{y}{x^a})$. There is again a blow-down 
map $\beta:[X;P]_{a,\{y=0\}}\to X$.

\subsection{Polyhomogeneous conormal functions}
We first define polyhomogeneous conormal functions on $\mathbb R_+$ and then extend the definition to arbitrary manifolds with corners.
\begin{definition} A function $f(x)$ is polyhomogeneous conormal on $\mathbb R_+$, written $f\in\mathcal A_{\phg}(\mathbb R_+)$, if it has an expansion at $x=0$ of the form
\[f(x)\sim\sum_{(s,p)\in E}a_{s,p}x^s(\log x)^p,\]
where $E\subseteq\mathbb C\times\mathbb N_0$ is an \emph{index set}:
\begin{enumerate}
\item The real parts $\Re(s)$ have a finite minimum and, when arranged in increasing order, approach $\infty$;
\item At most finitely many $p\in\mathbb N_0$ have $(s,p)\in E$ for any given $s$;
\item If $(s,p)\in E$, then $(s+1,p)\in E$, and if additionally $p>0$, then $(s,p-1)\in E$\footnote{This condition requires $E$ to be a $C^{\infty}$-index set, which means that the space $\mathcal A_{\phg}$ is invariant under a smooth change of coordinates. We require this throughout.}.
\end{enumerate}
Furthermore, the asymptotic expansions must hold after $b$-derivatives are applied: for any $k\in\mathbb N_0$ and any $r\in\mathbb R$, we must have
\[|(x\partial_x)^k(f(x)- \sum_{(s,p)\in E\, |\, \Re(s)\leq r}a_{s,p}x^s(\log x)^p)|\leq C_{r,k}x^r.\]
\end{definition}
Note that the $b$-derivative property means that these expansions may be differentiated term by term with respect to $x$. This is one reason to consider polyhomogeneous conormal expansions rather than ``standard'' asymptotic expansions.

Now let $X$ be a manifold with corners. An \emph{index family} $\mathcal E$ on $X$ is an assignment of an index set $E_{\mathcal F}$ to each boundary hypersurface $\mathcal F$. Then we define the space $\mathcal A_{\phg}^{\mathcal E}(X)$ inductively on $\dim X$ as follows. Near each boundary hypersurface $\mathcal F$ of $X$ with defining function $\rho_{\mathcal F}$, we require
\[f(x)\sim\sum_{(s,p)\in E_{\mathcal F}}a_{s,p}\rho_{\mathcal F}^s(\log \rho_{\mathcal F})^p,\]
where the coefficients $a_{s,p}$ are now polyhomogeneous conormal functions on $\mathcal F$, which has dimension $\dim X - 1$. In particular, suppose that $\mathcal F_1,\ldots,\mathcal F_{\ell}$ are the boundary hypersurfaces which intersect $\mathcal F$. Then we require
\[a_{s,p}\in\mathcal A_{\phg}^{(E_{\mathcal F_1},\ldots,E_{\mathcal F_{\ell}})}(\mathcal F).\]
Moreover, we require that the expansions be uniform after the application of any number of $b$-vector fields, which are vector fields on $X$ tangent to the boundary of $X$. In particular, if $V_1,\ldots,V_k$ are $k$ such vector fields, then we require
\[|V_1\dots V_k(f(x)- \sum_{(s,p)\in E\, |\, \Re(s)\leq r}a_{s,p}\rho_{\mathcal F}^s(\log \rho_{\mathcal F})^p)|\leq Cx^r\prod_{i=1}^k\rho_{\mathcal F_i}^{-N},\]
where $C$ and $N$ are positive real constants which may depend on $r$, $V_1,\ldots,V_k$.

\subsection{b-maps, b-submersions, b-fibrations, pullback, and pushforward}
Here we discuss several categories of smooth maps between manifolds with corners under which polyhomogeneous conormal functions behave nicely. The material is again from \cite{melrosebook}. Suppose that $X$ and $Y$ are manifolds with corners.
\begin{definition} A smooth map $f:X\to Y$ is called a \emph{b-map} if for every boundary hypersurface $\mathcal G$ of $Y$, there exist non-negative integers $e_f(\mathcal F,\mathcal G)$ such that
\[f^*\rho_{\mathcal G}=a\prod_{\mathcal F\subseteq X}\rho_{\mathcal F}^{e_f(\mathcal F,\mathcal G)},\]
with $a$ a smooth function on $X$. 
\end{definition}
Any $b$-map $f:X\to Y$ induces a map $^bf_*$ called the $b$-differential between the $b$-tangent bundles of $X$ and $Y$. See \cite{melrosebook} for details. We say that
\begin{definition} A b-map $f:X\to Y$ is a \emph{b-submersion} if $^bf_*(x)$ is surjective for each $x\in X$.
\end{definition}
\begin{definition} A \textbf{b-fibration} $f:X\to Y$ is a surjective b-submersion with the property that for any boundary hypersurface $\mathcal F$ of $X$, the image of $\mathcal F$ is either all of $Y$ or a boundary hypersurface of $Y$.
\end{definition}
In particular a b-fibration may not map a boundary hypersurface of $X$ into a corner of $Y$. See \cite{melrosebook,gribc} for further details.

We can now state the pullback and pushforward theorems of Melrose \cite{melrosebook}:
\begin{theorem}\label{thm:pullback} [Pullback theorem] If $f:X\to Y$ is a b-map and $a\in\mathcal A_{\phg}(Y)$, then $f^*a\in\mathcal A_{\phg}(X)$.
\end{theorem}
\begin{theorem}\label{thm:pushforward} [Pushforward theorem] If $f:X\to Y$ is a b-fibration, and $a\in\mathcal A_{\phg}(X)$ is a polyhomogeneous density on $X$ for which $f_*a$ is well-defined, then $f_*a\in\mathcal A_{\phg}(Y)$ is a polyhomogeneous density on $Y$.
\end{theorem}
These are very weak versions of the theorems -- the full versions give formulas for the index sets of the pullback and pushforward distributions in terms of the index sets of the original distributions. However we do not need these formulas for this manuscript.

The following lemma may be proved with explicit calculations in local coordinates.
\begin{lemma}\label{lem:blowdown} Suppose $P\subseteq X$ is a p-submanifold, and let $[X;P]$ be a blow-up (either regular or quasihomogeneous). Let $\beta:[X;P]\to X$ be the blow-down map. Then:
\begin{enumerate}
\item The blow-down map $\beta$ is a b-map.
\item If $P$ is given in each local model $\mathbb R_{+}^k\times\mathbb R^{N-k}$ by the vanishing of any number of the $\mathbb R_{+}$ variables, then $\beta$ is a b-submersion.
\end{enumerate}
\end{lemma}


\section{New results on polyhomogeneous conormal functions}

Let $X$ be a mwc with bhses $\mathcal F_1,\ldots,\mathcal F_n$ and let $\mathcal E=(E_{\mathcal F_1},\ldots,E_{\mathcal F_n})$ be an index family. Assume for simplicity that all exponents $(s,p)\in E_{\mathcal F_j}$ for each $j$ have $s$ real. For any index set $E$, let $\inf E$ be the element $(s,p)$ of $\mathbb R\times\mathbb N_0$ for which $s\leq s'$ for all $s'\in E$ and $p\geq p'$ for all $(s,p')\in E$. As shorthand, let
\[x^{\inf E}:= x^{(s,p)}:=x^s(\log x)^p.\]
We say that $\inf E\geq 0$ if either $s>0$ or $s=0$ and $p=0$, and we say that $\inf\mathcal E\geq 0$ if $\inf E\geq 0$ for each element $E$ of $\mathcal E$. A polyhomogeneous function $u$ which is continuous, or equivalently bounded, has a minimal index family $\mathcal E$ such that $\inf\mathcal E\geq 0$, and vice versa.

Our first result is a general result concerning the composition of smooth functions with polyhomogeneous functions.

\begin{theorem}\label{thm:comp} Suppose that $g\in\mathcal A_{phg}^{\mathcal E}(X)$, with $\inf\mathcal E\geq 0$. Suppose that $f:\mathbb R\to\mathbb R$ is smooth on an open set compactly containing the image of $g$. Then the composition $f\circ g\in\mathcal A_{phg}^*(X)$.
\end{theorem}
\begin{proof} The proof proceeds via induction on $\dim X$. So first we assume that $\dim X=1$ and let $x$ be a boundary defining function for one of its boundary hypersurfaces. Let $N\in\mathbb N$ be arbitrary. Since $f$ is smooth on an open set compactly containing the image of $g$, $f^{(N+1)}$ is bounded on the image of $g$; let $M_{N+1}$ be a global bound for $|f^{(N+1)}|$ on that image. By Taylor's theorem applied to $f(x)$ at $a=g(0)$,
\begin{equation}\label{eq:polycomp}
(f\circ g)(x) = f(g(0))+f'(g(0))(g(x)-g(0))+\dots+\frac{f^{(N)}(g(0))}{N!}(g(x)-g(0))^N + R_{N}(g(x)),
\end{equation}
where
\[|R_N(g(x))|\leq\frac{M_{N+1}}{(N+1)!}(g(x)-g(0))^{N+1}.\]
Since $g(x)$ is polyhomogeneous and continuous down to $x=0$, there exists an $\epsilon>0$ for which $|g(x)-g(0)|=O(x^{\epsilon})$. So each term $(g(x)-g(0))^k$ is $O(x^{k\epsilon})$ and the remainder $|R_N(g(x))|$ is bounded by $Cx^{\epsilon(N+1)}$. On the other hand, $g(x)-g(0)$ is of course polyhomogeneous as well, and since polyhomogeneous functions form an algebra, so is $(g(x)-g(0))^k$ for any $k$. All this means that \eqref{eq:polycomp}, as $N$ goes to infinity, gives a complete asymptotic expansion for $f\circ g$ at $x=0$. To find the terms of order at most $s$ in this expansion one needs to take approximately $s/\epsilon$ terms.

We must still show that this expansion is a polyhomogeneous \emph{conormal} expansion, in that it still holds when derivatives of the form $(x\partial_x)^m$ are applied. All terms in \eqref{eq:polycomp} are polyhomogeneous conormal except possibly for the remainder $R_N(g(x))$, and so it suffices to show that for each $m$ and each $N$, we still have, for some constant $C$ possibly depending on $m$ and $N$, that
\[|(x\partial_x)^mR_N(g(x))|\leq Cx^{\epsilon(N+1)}.\]
By writing out $(x\partial_x)^m$ we see that it is equivalent to show that
\[|x^m(\partial_x)^mR_N(g(x))|\leq Cx^{\epsilon(N+1)}.\]

We now prove this latter estimate. Using Fa\'a di Bruno's formula \cite{fdb}, we have
\[x^m(\partial_x)^mR_N(g(x))=\sum_{k=1}^mR_N^{(k)}(g(x))x^mB_{m,k}(g'(x),g''(x),\ldots,g^{m-k+1}(x)),\]
where $B_{m,k}$ are exponential Bell polynomials. We know that $R_N(x)$ is a smooth function of $x$ around $a=g(0)$ and that it is order $O((x-a)^{N+1})$. So its $k$th derivative $R_N^{(k)}(x)$ is also smooth and is of order $O((x-a)^{N+1-k})$. Thus
\[|R_N^{(k)}(g(x))|\leq Cx^{\epsilon(N+1-k)}.\]
As for the exponential Bell polynomial $B_{m,k}$, it is a finite sum of monomials $x^{\alpha}$, where $\alpha$ is a multiindex with order $k$ and with
\[\sum_{i=1}^{m-k+1}i\alpha_i=m.\]
So $B_{m,k}(g'(x),g''(x),\ldots,g^{m-k+1}(x))$ will be a finite sum of monomials in $g'(x),\ldots,g^{m-k+1}(x)$. Each such monomial has $k$ factors and involves a total of $m$ derivatives of $g$. Thus, distributing the $x^m$ so one factor of $x$ goes in front of each derivative of $g$, $x^mB_{m,k}(g'(x),g''(x),\ldots,g^{m-k+1}(x))$ is a finite sum of monomials in
\[x(\partial_x)g(x),\ldots,x^{m-k+1}(\partial_x)^{m-k+1}g(x),\]
with a total of $k$ terms having a total of $m$ factors of $\partial_x$. However, since $g(x)$ is polyhomogeneous conormal at $x=0$, we know that for each $j>0$, $x^j(\partial_x)^jg(x)$ is $O(x^{\epsilon})$. Since there are $k$ factors, each term in $x^mB_{m,k}(g'(x),g''(x),\ldots,g^{m-k+1}(x))$ is $O(x^{\epsilon k})$ and thus
\[|x^mB_{m,k}(g'(x),g''(x),\ldots,g^{m-k+1}(x))|\leq Cx^{\epsilon k}.\]
Putting everything together gives the estimate we want for $x^m(\partial_x)^mR_N(g(x))$ and completes the proof in the case $\dim X=1$.

The proof in the case $\dim X>1$ follows easily by induction. Let $x$ be a boundary defining function for some boundary hypersurface $H$ in $X$. Then we may write down the same expression \eqref{eq:polycomp}, except that now the coefficients $f^{(j)}(g(0))$ are not numbers but rather functions on $H$. By the inductive hypothesis, each $f^{(j)}(g(0))$ is polyhomogeneous on $H$, with the same index set. The same argument gives the same remainder estimates. So at each boundary hypersurface of $X$, $f\circ g$ has a polyhomogeneous expansion in powers of the boundary defining function, with polyhomogeneous coefficients. This shows that $f\circ g$ is polyhomogeneous, as desired.
\end{proof}

Theorem \ref{thm:comp} may be applied to give sufficient criteria for powers of a polyhomogeneous conormal function to be polyhomogeneous conormal. It is immediately apparent, for example, that powers of polyhomogeneous conormal functions which are both bounded and bounded away from zero are polyhomogeneous conormal. This may be generalized somewhat, and the first step is the following definition:

\begin{definition} Let $X$ be a mwc and let $\mathcal E$ be an index family which is nonempty at each face. We say that an element $u\in\mathcal A^{\mathcal E}_{phg}(X)$ is \emph{proper} if there exists a function $u_0\in\mathcal A_{phg}^{(\mathcal E-\inf\mathcal E)_+}(X)$ and a constant $C>0$ for which
\[u=\pm\rho^{\inf\mathcal E}u_0,\quad \frac 1C\leq u_0\leq C,\]
where
\[\rho^{\inf\mathcal E}=\prod_{i=1}^n\rho_{\mathcal F_i}^{\inf E_{\mathcal F_i}}.\]
Here $(\mathcal E-\inf\mathcal E)_+$ is the index set obtained by taking $\mathcal E-\inf\mathcal E$ and removing any terms with $p<0$.
\end{definition}

This definition has a number of features. First is that $u$ cannot be an element of $A^{\mathcal E'}_{phg}(X)$ for any $\mathcal E'$ which does not also contain $\inf\mathcal E$. Second, $A$ must be nonzero on the interior $X^{\circ}$. Third, $A$ may not decay faster at any point of any boundary hypersurface -- or at any boundary \emph{face}, even -- than $\rho^{\inf\mathcal E}$.

\begin{example} Consider a local model for a manifold with corners given by $X=[0,1)_x\times [0,1)_y\times (-1,1)_z$.
\begin{itemize}
\item The functions $x$ and $y$ are proper elements of $A^{\mathcal E}_{phg}(X)$, with $\mathcal E$ given by $(\mathbb N,\mathbb N_0)$ and $(\mathbb N_0,\mathbb  N)$ respectively.
\item The function $xe^z$ is also proper, with $(\mathbb N,\mathbb N_0)$. But the function $xz$ is not a proper element of the same space, because $xz=x\cdot z$ and $z$ has a zero in $X$. In fact $xz$ is not a proper element of any $A^{\mathcal E}_{phg}(X)$, because leading orders of 1 and 0 are the only choices where the definition could work and nevertheless it does not.
\item The function $u=x^2y+xy^2$ is not a proper element of any $A^{\mathcal E}_{phg}(X)$. If it were, the leading order of $\mathcal E$ at $\{x=0\}$ and $\{y=0\}$ would have to be 1, and yet if we write $u=xyv$, then $v=x+y$, which is not positive on $X$.
\item Care must be taken with logarithms. The function $\log x e^z$ is proper, as it is $\log x$ times $u_0=e^z$. But the function $(\log x - 1)e^z$ is not, because if we try to write it as $\log x$ times a function, that function is $u_0=(1-(1/\log x))e^z$, which is bounded above and below but is not polyhomogeneous.
\end{itemize}
\end{example}

\begin{definition} Let $\mathcal E$ be an index family on a mwc $X$. Then $\mathcal E$ is \emph{non-logarithmic at leading order} if each of its component index sets $E$ has $\inf E=(s,0)$ for some $s\in\mathbb R$.
\end{definition}

The definitions we have given allow us to characterize pc functions which have pc powers, including inverses $(a=-1)$ and square roots $(a=1/2)$:
\begin{theorem}\label{thm:appendixproperpowers} Suppose that $\mathcal E$ is non-logarithmic at leading order and that $u\in\mathcal A^{\mathcal E}_{phg}(X)$ is proper. Let $a\in\mathbb R$ be such that $u^a$ is well-defined. Then $u^a\in\mathcal A^*_{phg}(X)$.
\end{theorem}

\begin{proof} Since $\mathcal E$ is non-logarithmic at leading order, the function $(\rho^{\inf\mathcal E})^{a}$ is pc. Since pc functions form an algebra, it therefore suffices to assume that $u=u_0$, i.e. that $u$ itself is positive, bounded, and bounded away from zero. The proof is then an immediate consequence of Theorem \ref{thm:comp}, since $x^{a}$ is smooth on $(0,\infty)$, which compactly contains the image of $u$.
\end{proof}

Much of the analysis of Bessel functions via Nicholson's formula deals with exponentials of functions. It is useful to have criteria on a pc function $u$ which guarantee the polyhomogeneity of $e^{-u}$. The following is such a criterion.
\begin{theorem}\label{thm:exppc} Let $X$ be a mwc and let $\mathcal E$ be an index family on $X$. Suppose that $u\in\mathcal A^{\mathcal E}_{phg}(X)$. Furthermore, suppose that for any boundary hypersurface $\mathcal F\subseteq X$, one of the following two things happens:
\begin{enumerate}
\item $\mathcal F$ is a \textbf{type 1 bhs} for $u$, meaning that $\inf E_{\mathcal F}\geq 0$;
\item $\mathcal F$ is a \textbf{type 2 bhs} for $u$: there exist positive constants $c$ and $\epsilon$ for which, in a neighborhood of $\mathcal F$,
\[u\geq c\rho_{\mathcal F}^{-\epsilon}.\] 
\end{enumerate}
Then $e^{-u}\in\mathcal A^{\mathcal E'}_{phg}(X)$, where $\mathcal E'$ is a non-negative index family which is empty at each type 2 bhs.
\end{theorem}
\begin{proof} We do this locally, in a neighborhood of each corner of $X$. Use coordinates
\[(x_1,\ldots,x_k,y_1,\ldots,y_{\ell},z_1,\ldots,z_m),\]
where $x_i\in [0,1)$ are boundary defining functions for the type 1 hypersurfaces, $y_i\in[0,1)$ are boundary defining functions for the type 2 hypersurfaces, and $z_i\in(-1,1)$ are coordinates in the other variables.

The proof will consist of induction on $k$. In the base case, $k=0$, and it is then straightforward from the definition of a type 2 bhs that there exist positive constants $c$ and $\epsilon$ for which
\[u\geq c(\min(y_1,\ldots,y_{\ell}))^{-\epsilon}.\]
It is immediate that then $e^{-u}$ decays faster than any polynomial in $y_1,\ldots,y_{\ell}$ at each type 2 bhs. Furthermore, derivatives of $e^{-u}$ consist of polynomials in $u$ and its derivatives, times $e^{-u}$. Since $u$ is polyhomogeneous, each polynomial in $u$ and its derivatives decays faster than \emph{some} power of $y_i$, and therefore each derivative of $e^{-u}$ decays faster than any polynomial as well. This is precisely what it means for $e^{-u}$ to be polyhomogeneous with empty index set at each type 2 bhs.

Now suppose we know the theorem for $k$ and take a function $u$ of $(x_1,\ldots,x_{k+1},y_1,\ldots,y_{\ell},z_1,\ldots,z_m)$. Consider its expansion at any type 1 bhs -- without loss of generality, consider $\{x_{k+1}=0\}$. Then $u$ has a polyhomogeneous expansion
\[u=a_0+a_1x_{k+1}^{s}(\log x_{k+1})^p+\dots,\]
where $s>0$ and each $a_i$ is a polyhomogeneous function of the variables \[(x_1,\ldots,x_k,y_1,\ldots,y_{\ell},z_1,\ldots,z_m).\] Define $g=u-a_0$. Then
$e^{-u}=e^{-a_0}e^{-g}$. It is immediate that $a_0$ is a function on $\mathcal F$ which satisfies the hypotheses of our theorem. Therefore, by the inductive hypothesis, $e^{-a_0}$ is polyhomogeneous, with infinite order decay in each $y_i$ variable. It is also continuous down to each $\{x_i=0\}$, $1\leq i\leq k$.

The analysis now proceeds as a variation on the proof of Theorem \ref{thm:comp}. For any $N$ we may write
\[e^{-g}=1-g(x)+\frac 1{2!}(g(x))^2+\dots+\frac{(-1)^N}{N!}(g(x))^N+R_N(g(x)).\]
Although it is true that $g(x)\to 0$ as $x\to 0$, and that $g(x)$ is polyhomogeneous, we do \emph{not} necessarily obtain a polyhomogeneous expansion for $e^{-g}$ by plugging in the expansions for $g(x)$ and its powers. The reason is that $g(x)$ is not necessarily bounded -- it may blow up at one or many $\{y_i=0\}$. As such the sequence of powers of $g(x)$ may get worse at such a boundary hypersurfaces as $N$ increases, which means that the coefficients of the expansion in $x$ do not live in the same fixed $\mathcal A^{\mathcal E}_{phg}(\mathcal F)$. However, when we multiply this formal expansion by $e^{-a_0}$, we get
\[e^{-u}=e^{-a_0}-e^{-a_0}g+\dots+e^{-a_0}\frac{(-1)^N}{N!}g^N + e^{-a_0}R_N(g(x)).\]
Now we can plug in the expansions for $g(x)$ and its powers, because when we multiply each of them by $e^{-a_0}$, they decay to infinite order in each $y_i$. Thus the coefficients do live in the same space, namely $\mathcal A^{\mathcal E'|_{\mathcal F}}_{phg}(\mathcal F)$. Assuming that we have the same estimates as in the proof Theorem \ref{thm:comp} on $e^{-a_0}R_N(g(x))$ and its derivatives, we see that $e^{-u}$ has a polyhomogeneous expansion at each type 1 bhs $\mathcal F$ with coefficients in  $\mathcal A^{\mathcal E'|_{\mathcal F}}_{phg}(\mathcal F)$ and therefore is itself polyhomogeneous conormal.

And the remainder estimates are indeed very similar to those in the proof of Theorem \ref{thm:comp}. Although $g(x)$ is no longer continuous, there certainly exist positive constants $C$, $\epsilon$, and $M$ such that for all $j$ between $0$ and $N+1$,
\[|(x_{k+1}\partial_{x_{k+1}})^jg(x)|\leq Cx_{k+1}^{\epsilon}(y_1y_2\ldots y_{\ell})^{-M}.\]
We then use Fa\'a di Bruno's formula precisely as in the proof of Theorem \ref{thm:comp}; the extra factors of $(y_1y_2\ldots y_{\ell})^{-M}$ are only finite in number and are absorbed into the super-polynomial decay of $e^{-a_0}$ in the $y$-variables. We end up with the estimate
\[|x_{k+1}^m\partial_{x_{k+1}}^m e^{-a_0}R_N(g(x))|\leq C(y)x_{k+1}^{\epsilon(N+1)},\]
where $C(y)$ decays to infinite order in the $y$-variables. This is precisely what is needed to complete the proof. \end{proof}

\section{Combinatorics of b-maps and b-fibrations}

Throughout this section we let $\pi$ be the projection map from $Q_0\times[0,\infty]$ to $Q_0$ given by $\pi(z,\nu,t)=(z,\nu)$.

\begin{proposition}\label{prop:bfibx4} The map $\pi$ lifts to a b-fibration $\pi_i$ from $X_i\to U_1$ for $i\in\{1,2,3,4\}$.
\end{proposition}
\begin{proof} Observe that $X_1=U_1\times [0,\infty]$. Product-type projections are b-fibrations, so $\pi_1$ is a b-fibration.

For $\pi_2$, observe that the blow-down map $\beta_2$ from $X_2$ to $X_1$ is a b-submersion. Since $\pi_2=\pi_1\circ\beta_2$ and the composition of b-submersions is a b-submersion, $\pi_2$ is a b-submersion. Moreover the new front face bif has image be, so $\pi_2$ is a b-fibration.

For $\pi_3$ we use the same reasoning: the blow-down map $\beta_3$ from $X_3\to X_2$ is a b-submersion by Lemma \ref{lem:blowdown}, and the new front face bf$_0$ has image be.

And again, the same holds for $\pi_4$. The blow-down map $\beta_4$ is a b-submersion by Lemma \ref{lem:blowdown} and the two front faces of$_0$ and af$_0$ have images oe and ae respectively.
\end{proof}

\begin{proposition}\label{prop:bfibx6} The projection map $\pi$ lifts to a b-fibration $\pi_6$ from $X_6\to U$.
\end{proposition}
\begin{proof} This argument is analogous to the proof of \cite[Lemma 2.5]{hmm}. In $U_1$, we blow up the intersection of the diagonal with be to create $U$. The pre-image under $\pi_4:X_4\to U_1$ of that intersection is the union of the three $p$-submanifolds $D\cap$bf$_0$, $D\cap$bif, and $D\cap$bf. So in order to lift $\pi_4$ to a b-fibration onto $U$, we must blow up each of the three elements of that union, in any order. These are precisely the three blow-ups done to create $X_6$.

We cannot use \cite[Lemma 2.5]{hmm} directly, as some of our blow-ups are quasihomogeneous. Instead we mimic the proof, which consists of analysis in local coordinates. Each of the new boundary faces ff, bff, and ff$_0$ have image fe, so it is enough to prove that $\pi_6$ lifts to a b-submersion from $X_6\to U$ Consider coordinate patches on $X_4$ near the diagonal before the blow-up. In the image we use the coordinates $(\mu,w)$ throughout.

Near bf$_0\cap$tf$\cap D$ we can use $(\mu,w,s')$. The projection map $\pi_4$ takes $(\mu,w,s')$ to $(\mu,w)$. We blow up $\{w=\mu=0\}$ in the image to get $U$, quadratically with respect to $\{w=0\}$, and in the pre-image, the sixth blow-up is a blow-up of $\{w=\mu=0\}$, also quadratic with respect to $\{w=0\}$. Thus in this region $\pi_6$ lifts to a product-type projection map, which is a b-submersion.

Near bif$\cap$bf$_0\cap D$ we can use $(\kappa',w,\sqrt s)$. The projection map $\pi_4$ is given by
\[\pi_4(\kappa',w,\sqrt s)=(\kappa'\sqrt s,w).\]
In the preimage we blow up first $\{\kappa'=w=0\}$ quadratically with respect to $w$ and then $\{\sqrt s=w=0\}$ quadratically with respect to $w$. We get three different local projective coordinate systems (where $w\geq 0$; there are others with $w$ replaced by $-w$ when $w<0$, but we omit these as the analysis is identical):
\[(\sqrt w,\frac{\sqrt s}{\sqrt w},\kappa'),\quad (\frac{\sqrt w}{\sqrt s},\frac{\kappa'\sqrt{s}}{\sqrt w},\sqrt s),\quad (\frac{w}{s(\kappa')^2},\kappa',\sqrt s).\]
The former two are mapped by the lift of $\pi_4$ into the region where $(\sqrt w,\frac{\mu}{\sqrt w})$ are good coordinates, and the last is mapped into the region where $(\frac{w}{\mu^2},\mu)$ are good coordinates. Thus $\pi_4$ lifts to the following maps in each of the three coordinate systems:
\[(x_1,x_2,x_3)\to (x_1,x_2x_3);\quad (x_1,x_2,x_3)\to(x_1x_3,x_2);\quad (x_1,x_2,x_3)\to(x_1,x_2x_3).\]
The map $(x_1,x_2)\to x_1x_2$ is a b-submersion from $\mathbb R_+^2$ to $\mathbb R_+$. So each of these three maps extends by continuity to a well-defined b-submersion (note that the only $\mathbb R$ coordinate, the first coordinate in the third map, is left alone). Thus $\pi_4$ lifts to a well-defined b-submersion $\pi_6$ in this region.

Finally, we have to consider the region near bif$\cap$bf$\cap D$, where we use $(\kappa,w,t)$. But the projection map $\pi_4$ is given by $\pi_4(\kappa,w,t)=(\kappa t,w)$ and the analysis is identical, after rearranging the variables, to the analysis near bif$\cap$bf$_0\cap D$. Thus $\pi_4$ lifts to a well-defined b-submersion $\pi_6$ here as well, completing the proof.
\end{proof}

\begin{proposition}\label{prop:bfibxsharp} The map $\pi$ lifts to a b-fibration from $X^{\sharp}\to U^{\sharp}$.
\end{proposition}
\begin{proof} First we show that it lifts to a b-fibration from $X^{\sharp}_6\to U^{\sharp}$. We know it lifts to a b-fibration from $X_6\to U$, which is certainly a b-map, and therefore that $\pi^{-1}(\rho_{obe})$ is a product of positive integer powers of $\rho_{obif}$, $\rho_{obf}$, and $\rho_{obf_0}$. Thus $\pi^{-1}(\sqrt{\rho_{obe}})$ is a product of positive integer powers of $\sqrt{\rho_{obif}}$, $\sqrt{\rho_{obf}}$, and $\sqrt{\rho_{obf_0}}$. So $\pi$ definitely lifts to a b-map from $X^{\sharp}_6\to U^{\sharp}$. This map is a b-submersion as well, as it is easy to see that, for example,
\[\rho_{obe}\partial_{\rho_{obe}}= \frac 12\sqrt{\rho_{obe}}\partial_{\sqrt{\rho_{obe}}}.\]
There is thus a natural isomorphism between the b-tangent bundles of $X_6$ and $X_6^{\sharp}$, and between the b-tangent bundles of $U$ and $U^{\sharp}$. Since $\pi$ is a b-submersion, its b-differential $\pi^*$ is surjective from $T_bX_6\to T_bU$. Since it respects these natural isomorphisms, $\pi^*$ is also surjective from $T_bX_6^{\sharp}\to T_bU^{\sharp}$, and therefore $\pi$ lifts to a b-submersion from $X^{\sharp}_6\to U^{\sharp}$. Finally, $\pi$ does not map any boundary hypersurface of $X_6^{\sharp}$ into a corner of $U^{\sharp}$, and therefore $\pi$ is a b-fibration as claimed.

Now we must prove that we can lift the domain from $X^{\sharp}_6$ to $X^{\sharp}$ and lift $\pi$ along with it. All blow-down maps are b-maps so $\pi$ certainly lifts to a b-map. To show that the lift is a b-submersion, we first observe that the blow-down map from $[X_6^{\sharp};of\cap tif]\to X_6^{\sharp}$ is a b-submersion, and so its composition with $\pi$ is a b-submersion. Since the image of the new front face tof is oe, it is also a b-fibration. The next two blow-ups, creating tofx and obfx, do \emph{not} have blow-down maps which are b-submersions -- but both blow-ups are b-transversal to $\pi$. Moreover the images of tofx and obfx are oe and obe respectively. Thus, by \cite[Lemma 2.7]{hmm},\footnote{This lemma is stated for regular blow-ups but it holds for quasihomogeneous blow-ups as well, with the same proof.} $\pi$ does indeed lift to a b-fibration from $X^{\sharp}\to U^{\sharp}$, completing the proof.
\end{proof}

We also need:

\begin{proposition}\label{prop:easybmap} The map $(z,\nu,t)\to 2z\sinh t$ lifts to a b-map from $X_4$ to $[0,\infty]$, where the image is taken as a manifold with corners with defining functions $x$ at $x=0$ and $x^{-1}$ at $x=\infty$.
\end{proposition}
\begin{proof} Make a direct computation of the pullbacks of $x$ and $x^{-1}$ under this map. The function $x$ pulls back to $2z\sinh t$, which can be written $2zt(\frac{\sinh t}{t})$. Away from tif, $2\frac{\sinh t}{t}$ is smooth and nonvanishing and can be ignored. As functions on $X_4$,
\[z=\zeta^{-3}\cong \rho_{af}^{-3}\rho_{bf}^{-3}\rho_{bif}^{-3}\rho_{bf_0}^{-3}\rho_{af_0}^{-3};\quad t\cong \rho_{tf}\rho_{bif}\rho_{bf_0}^{3}\rho_{af_0}^3,\]
where $\cong$ denotes equality up to a multiple of a smooth nonvanishing function. Thus $zt\cong \rho_{af}^{-3}\rho_{bf}^{-3}\rho_{bif}^{-2}\rho_{tf}$, which is a smooth nonvanishing multiple of $\rho_{tf}$ in a neighborhood of tf (which intersects bf$_0$, of$_0$, and af$_0$ but not any of the other boundary hypersurfaces -- this is why the fourth blow-up was required as otherwise af would intersect tf). This shows that $x$ pulls back to a smooth nonvanishing function times a product of powers of boundary defining functions, specifically a single power of $\rho_{tf}$.

As for $x^{-1}$, it can be written $2z^{-1}t^{-1}(\frac{\sinh t}{t})^{-1}$. Away from tif, again, we have $z^{-1}t^{-1}\cong \rho_{af}^3\rho_{bf}^3\rho_{bif}^2\rho_{tf}^{-1}$. This is a smooth nonvanishing function times a product of powers of boundary defining functions in a neighborhood of af, bf, and bif (which is disjoint from tf). The same analysis applies. Near tif, the pullback is $z\tau^{6}/(1+\tau^{12})$, which is $\rho_{af}^3\rho_{bf}^3\rho_{tif}^6$, times a smooth nonvanishing function. So the pullback of $x^{-1}$ is a smooth nonvanishing function times a product of powers of defining functions, namely $\rho_{af}^3\rho_{bf}^3\rho_{bif}^2\rho_{tf}^6$. This shows that the map $(z,\nu,t)\to 2z\sinh t$ is a b-map from $X_4$ to $[0,\infty]$.
\end{proof}

A similar argument, but omitting the portion of $X_4$ near tif, shows the following:

\begin{proposition}\label{prop:easybmap2} For any $C$, the map $(z,\nu,t)\to 2\nu t$ lifts to a b-map from $X_4\cap\{t < C\}$ to $[0,\infty]$, where again the image is taken as a manifold with corners with defining functions $x$ at $x=0$ and $x^{-1}$ at $x=\infty$.
\end{proposition}

Finally, we must prove:

\begin{proposition}\label{prop:triplescat} Let $\pi$ be a projection map from $[0,1)_{\zeta}\times [0,1)_{\mu}\times[0,1)_{\hat u}$ given by omitting one of the variables. Then $\pi$ lifts to a b-fibration from $\hat U$ to $U$. If the omitted variable is $\tilde u$, its restriction to $\hat U^{\flat}$ is also a b-fibration.
\end{proposition}
\begin{proof} Without loss of generality, let $\pi$ be the projection map off the third variable $\tilde u$. It is a standard result that $\pi$ lifts to a b-fibration from $\hat U_1\to Q_1$ (see for example \cite{mesi}). We claim now that $\pi$ lifts to a b-fibration from the following space to $U$:
\[\hat U_2:=[\hat U_1;D_{111}\cap\mathcal F_{111};D_{110}\cap\mathcal F_{111}; D_{110}\cap\mathcal F_{110}].\]
Indeed, this would follow from \cite[Lemma 2.5]{hmm} if the blow-ups were not quasihomogeneous. However, mimicking the proof of that Lemma, precisely as in the proof of Proposition \ref{prop:bfibx6}, is enough to prove the claim.

Now we need to lift $\pi$ to a b-fibration from $\hat U$ to $U$. This is a straightforward application of \cite[Lemma 2.7]{hmm}, as all four remaining blow-ups -- $D_{101}\cap\mathcal F_{111}$, $D_{011}\cap\mathcal F_{111}$, $D_{101}\cap\mathcal F_{101}$, and $D_{011}\cap\mathcal F_{011}$ - are of p-submanifolds which are b-transversal to $\pi$. So $\pi$ lifts to a b-fibration from $\hat U$ to $U$. As for its restriction to $\hat U^{\flat}$, first observe that $\hat U^{\flat}$ is in fact a manifold with corners in its own right, and that the restriction is still a b-map and a b-submersion. Finally, the new face comprising the boundary of $\hat U^{\flat}$ is mapped into the interior of $U$, and all other face images are unchanged, so the restriction is a b-fibration, as desired.\end{proof}


\begin{thebibliography}{ABCDE}
\bibitem[CdV10]{cdv} Yves Colin de Verdi\`{e}re. \emph{On the remainder in the Weyl formula for the Euclidean disk}. S\'{e}minaire de th\'{e}orie spectrale et g\'{e}om\'{e}trie \textbf{29} (2010-2011), 1--13.
\bibitem[EPT16]{ept} Suresh Eswarathasan, Iosif Polterovich, and John A. Toth. \emph{Smooth Billiards with a Large Weyl Remainder.} Int. Math. Res. Not. \textbf{2016} (2016), no. 12, 3639--3677.
\bibitem[GR]{gr} I. S. Gradshteyn and I. M. Ryzhik. \emph{Table of Integrals, Series, and Products}. 7th edition. Elsevier Academic Press, Burlington, MA, USA (2007).
\bibitem[Gri01]{gribc} Daniel Grieser. \emph{Basics of the $b$-calculus.} In \emph{Approaches to Singular Analysis}, Adv. in PDE, Birkhauser. Basel (2001), 30--84.
\bibitem[Gri17]{griscales} Daniel Grieser. \emph{Scales, blow-up and quasimode constructions.} Contemp. Math. AMS \textbf{700}, Amer. Math. Soc., Providence, RI (2017), 207--266.
\bibitem[GrHu09]{grhu} Daniel Grieser and Eug\'enie Hunsicker. \emph{Pseudodifferential operator calculus for generalized $\mathbb Q$-rank 1 locally symmetric spaces}. J. Funct. Anal. \textbf{257} (2009), no. 12, 3748--3801.
\bibitem[GTV20]{gtv} Daniel Grieser, Mohammad Talebi, and Boris Vertman. \emph{Spectral geometry on manifolds with fibred boundary metrics I: low energy resolvent}. Preprint at \url{https://arxiv.org/pdf/2009.10125.pdf}.
\bibitem[GMWW]{gmww} Jingwei Guo, Wolfgang M\"{u}ller, Weiwei Wang, and Zuoqin Wang. \emph{The Weyl formula for planar annuli.} J. Funct. Anal. \textbf{281} (2021), no. 4.
\bibitem[GWW19]{gww} Jingwei Guo, Weiwei Wang, and Zuoqin Wang. \emph{An improved remainder estimate in the Weyl formula for the planar disk.} J. Fourier Anal. Appl. \textbf{25} (2019), 1553--1579.
\bibitem[HMM95]{hmm} Andrew Hassell, Rafe Mazzeo, and Richard B. Melrose. \emph{Analytic surgery and the accumulation of eigenvalues}. Comm. Anal. Geom. \textbf{3} (1995), no. 1--2, 115--222.
\bibitem[HBRV]{hbrv} Z. Heitman, J. Bremer, V. Rokhlin, and B. Vioreanu, {\em Asymptotics of Bessel functions in the Fresnel regime.} Applied and Computational Harmonic Analysis \textbf{39}, p. 347-356, 2015.
\bibitem[Hor17]{horsley} David Horsley, {\em Bessel phase functions: calculation and application.} Numerische Mathematik \textbf{136}, 2017.
\bibitem [IKKN06]{latticesurvey} A. Ivi\`{c}, E. Kr\"atzel, M. K\"uhleitner, and W. G. Nowak. \emph{Lattice points in large regions and related arithmetic functions: recent developments in a very classic topic}. Elementare und analytische Zahlentheorie, Schr. Wiss. Ges. Johann Wolfgang Goethe Univ. Frankfurt am Main, \textbf{20}, Franz Steiner Verlag Stuttgart, 2006, 89--128.
\bibitem[Ivr80]{ivrii} Victor Ivrii. \emph{The second term of the spectral asymptotics for a Laplace-Beltrami operator on manifolds with boundary}. Funct. Anal. Appl. \textbf{14} (1980), no. 2, 25--34.
\bibitem[Joh02]{fdb} Warren P. Johnson. \emph{The curious history of Fa\`a di Bruno's formula.} American Mathematical Monthly \textbf{109} (2002), no. 3, 217--234.
\bibitem[KuFe65]{kufe} N. V. Kuznecov and B. V. Fedosov. \emph{An asymptotic formula for eigenvalues of a circular membrane}. Differ. Uravn. \textbf{1} (1965), 1682--1685.
\bibitem[Lau12]{laugesen} Richard S. Laugesen. \emph{Spectral Theory of Partial Differential Equations - Lecture Notes.} Preprint available at \url{https://arxiv.org/pdf/1203.2344.pdf}.
\bibitem[LPS22]{lps22} M. Levitin, I. Polterovich, and D.A. Sher. \emph{P\'olya's conjecture for the disk: a computer-assisted proof.} In preparation, to be published on arXiV by March 16, 2022.
\bibitem[Maz91]{maz91} Rafe Mazzeo. \emph{Elliptic theory of differential edge operators I.} Comm. in PDE \textbf{16} (1991), no. 10, 1615--1664.
\bibitem[Mel93]{tapsit} Richard Melrose. \emph{The Atiyah-Patodi-Singer index theorem}. Research Notes in Mathematics, A. K. Peters, Wellesley, MA (1993).
\bibitem[Mel94]{mel94} Richard Melrose. \emph{Spectral and scattering theory for the Laplacian on asymptotically Euclidean spaces}. In \emph{Spectral and scattering theory (Sanda, 1992)}, Lecture Notes in Pure and Appl. Math. \textbf{161}. Dekker, New York (1994), 85--130.
\bibitem[Mel96]{melrosebook} Richard B. Melrose. \emph{Differential analysis on manifolds with corners.} Book in preparation, available at \url{http://www-math.mit.edu/~rbm/}.
\bibitem[MeSi08]{mesi} Richard Melrose and Michael Singer. \emph{Scattering configuration spaces.} Preprint at \url{https:arxiv.org/pdf/0808.2022.pdf}.
\bibitem[Olv54]{olver54} Frank W. J. Olver. \emph{The asymptotic expansion of Bessel functions of large order}. Philos. Trans. Roy. Soc. London Ser. A. \textbf{247} (1954), 328--368.
\bibitem[Olv97]{olver97} Frank W. J. Olver. \emph{Asymptotics and Special Functions}. A. K. Peters, Wellesley, MA (1997).
\bibitem[OlWo21]{dlmfch2} Frank W. J. Olver and R. Wong, editors. \emph{Digital Library of Mathematical Functions, Chapter 2: Asymptotic Analysis.} National Institute of Standards, Gaithersburg, MD, USA 2021. \url{https://dlmf.nist.gov/2}. 
\bibitem[Olv21]{dlmfch9} Frank W. J. Olver, editor. \emph{Digital Library of Mathematical Functions, Chapter 9: Airy and Related Functions.} National Institute of Standards, Gaithersburg, MD, USA 2021.  \url{https://dlmf.nist.gov/9}. 
\bibitem[OlMa21]{dlmf} Frank W. J. Olver and Leonard C. Maximon, editors. \emph{Digital Library of Mathematical Functions, Chapter 10: Bessel Functions.} National Institute of Standards, Gaithersburg, MD, USA 2021. \url{https://dlmf.nist.gov/10}. 
\bibitem[Wat22]{watson} George N. Watson. \emph{A treatise on the theory of Bessel functions}. University Press, Cambridge (1922).


\end{thebibliography}
\end{document}